\documentclass[12pt]{amsart}
\usepackage{}
\usepackage{amsfonts}
\usepackage{latexsym,bm}
\usepackage{amssymb}
\usepackage{amsmath}
\usepackage{amsthm, amscd}
\usepackage{mathrsfs}
\usepackage[all]{xy}
\usepackage{xcolor}
\numberwithin{equation}{section}

\theoremstyle{plain}
\newtheorem{thm}{Theorem}[section]
\newtheorem{lem}[thm]{Lemma}
\newtheorem{prop}[thm]{Proposition}
\newtheorem{cor}[thm]{Corollary}
\newtheorem{lem-defi}[thm]{Lemma-Definition}
\newtheorem{prop-defi}[thm]{Proposition-Definition}

\theoremstyle{remark}
\newtheorem{rmk}[thm]{Remark}

\newtheorem{example}[thm]{Example}
\theoremstyle{definition}
\newtheorem{defi}[thm]{Definition}

\newcommand{\mcA}{{\mathcal A}}
\newcommand{\mcB}{{\mathcal B}}
\newcommand{\mcC}{{\mathcal C}}
\newcommand{\mcD}{{\mathcal D}}

\newcommand{\mcG}{{\mathcal G}}

\newcommand{\mcK}{{\mathcal K}}
\newcommand{\mcL}{{\mathcal L}}
\newcommand{\mcM}{{\mathcal M}}
\newcommand{\mcN}{{\mathcal N}}
\newcommand{\mcO}{{\mathcal O}}

\newcommand{\mcS}{{\mathcal S}}
\newcommand{\mcT}{{\mathcal T}}
\newcommand{\mcU}{{\mathcal U}}
\newcommand{\mcV}{{\mathcal V}}
\newcommand{\mcW}{{\mathcal W}}
\newcommand{\mcX}{{\mathcal X}}
\newcommand{\mcY}{{\mathcal Y}}
\newcommand{\mcZ}{{\mathcal Z}}

\newcommand{\mbA}{{\mathbb A}}

\newcommand{\C}{{\mathbb C}}
\newcommand{\mbF}{{\mathbb F}}
\newcommand{\mbG}{{\mathbb G}}

\newcommand{\mbP}{{\mathbb P}}
\newcommand{\mbQ}{{\mathbb Q}}
\newcommand{\mbR}{{\mathbb R}}
\newcommand{\mbS}{{\mathbb S}}

\newcommand{\mbZ}{{\mathbb Z}}

\newcommand{\msO}{{\mathscr O}}

\newcommand{\msT}{{\mathscr T}}

\newcommand{\mfC}{{\mathfrak C}}
\newcommand{\mfD}{{\mathfrak D}}

\newcommand{\mff}{{\mathfrak f}}
\newcommand{\mfg}{{\mathfrak g}}
\newcommand{\mfh}{{\mathfrak h}}

\newcommand{\mfl}{{\mathfrak l}}
\newcommand{\mfm}{{\mathfrak m}}

\newcommand{\bP}{{\mathbb P}}

\newcommand{\id}{{\rm id}}

\newcommand{\Pic}{\text{Pic}}
\newcommand{\Aut}{\text{Aut}}
\newcommand{\Hom}{\text{Hom}}
\newcommand{\Sym}{\text{Sym}}

\newcommand{\End}{\text{End}}

\newcommand{\diag}{\text{diag}}

\newcommand{\fg}{{\mathfrak g}}
\newcommand{\fh}{{\mathfrak h}}

\newcommand{\aut}{{\mathfrak a}{\mathfrak u}{\mathfrak t}}
\newcommand{\BP}{{\mathbb P}}
\newcommand{\0}{{\mathcal O}}

\def\RatCurves{\mathop{\rm RatCurves}\nolimits}
\def\Aut{\mathop{\rm Aut}\nolimits}

\def\Gr{\mathop{\rm Gr}\nolimits}
\def\IG{\mathop{\rm IG}\nolimits}
\def\OG{\mathop{\rm OG}\nolimits}
\def\Sp{\mathop{\rm Sp}\nolimits}
\def\SO{\mathop{\rm SO}\nolimits}

\def\Sym{\mathop{\rm Sym}\nolimits}
\def\Pic{\mathop{\rm Pic}\nolimits}
\def\End{\mathop{\rm End}\nolimits}
\def\Hom{\mathop{\rm Hom}\nolimits}

\def\PGL{\mathop{\rm PGL}\nolimits}
\def\PSL{\mathop{\rm PSL}\nolimits}

\newcommand\LG{\mathop{\rm LG}}

\newcommand{\sC}{{\mathcal C}}

\newcommand{\sH}{{\mathcal H}}

\newcommand{\sK}{{\mathcal K}}
\newcommand{\sL}{\mathcal{L}}

\newcommand{\sO}{{\mathcal O}}

\newcommand{\sU}{{\mathcal U}}

\newcommand{\sW}{{\mathcal W}}
\newcommand{\sX}{{\mathcal X}}

\newcommand{\sZ}{{\mathcal Z}}
\newcommand{\n}{{\rm n}}

\def\cone{\mathop{\rm cone}\nolimits}
\def\Exc{\mathop{\rm Exc}\nolimits}
\def\End{\mathop{\rm End}\nolimits}
\def\Div{\mathop{\rm Div}\nolimits}
\def\aff{\mathop{\rm aff}\nolimits}

\def\FRC{\mathop{\rm FRC}\nolimits}

\begin{document}

\title[Rigidity of wonderful compactifications]{Rigidity of wonderful  group compactifications under Fano deformation}

\author{Baohua Fu and Qifeng Li}

\thanks{\today}

\maketitle

\begin{abstract}
 For  a complex connected semisimple linear algebraic group $G$ of adjoint type and of rank $n$, De Concini and Procesi constructed its wonderful compactification $\bar{G}$, which is a smooth Fano $G \times G$-variety of Picard number $n$ enjoying  many interesting properties.
In this paper, it is shown that the wonderful compactification $\bar{G}$  is rigid under Fano deformation. Namely, for any regular family
 of  Fano manifolds over a connected base, if one fiber is isomorphic to $\bar{G}$, then so are  all other fibers. This answers a question raised by Bien and Brion in their work on the local rigidity of wonderful varieties.
\end{abstract}

\tableofcontents

\section{Introduction}

Throughout this paper, we work over the complex number field. A smooth projective variety $X$ is said to be {\em rigid under projective deformation} if for any smooth projective family over a connected base  $ \sX \to B$ with one fiber $\sX_{b_0}$ isomorphic to $X$, then all fibers $\sX_b$ are isomorphic to $X$. In this article, by {\it rigid} without further qualification, we will mean {\it rigid under projective deformation}. The projective space $\mbP^n$ is rigid by the characterization of projective spaces (\cite{KO}). But in general, the rigidity is a strong property  which is difficult to check.  We refer to  surveys \cite{S} and \cite{H06} for an account of the history and the  development of this problem.

To show the rigidity of a smooth projective variety $X$, it is natural to prove first that $X$ is {\em locally rigid}, namely $\sX_b \simeq X$ for all $b$ in an analytic open neighborhood of $b_0$.  By Kodaira-Spencer deformation theory and by Akizuki-Nakano vanishing theorem(\cite{AN}), a smooth Fano variety $X$ is locally rigid if and only if $H^1(X, T_X)=0$. The latter can be effectively computed in many concrete cases.  For example, rational homogeneous spaces $G/P$ or more generally  spherical  Fano manifolds without color are locally rigid (\cite[Proposition 4.2]{BB96}).

A locally rigid variety $X$ is rigid if and only if it is {\em rigid under specializations}, namely if for any smooth projective family over the unit disk $ \sX \to \Delta$ with  $\sX_{t} \simeq X$ for all $t \neq 0$,  we have $\sX_0 \simeq X$.  It is a difficult and subtle problem to prove the rigidity under specializations and there is no known  cohomological characterization  of it.  Even for rational homogeneous varieties $G/P$ of Picard number one, the rigidity does not always hold. To wit, let $B_3/P_2$ be the variety of lines on a $5$-dimensional smooth hyperquadric $\mathbb{Q}^5$. Pasquier and Perrin constructed in \cite{PP} an explicit family specializing $B_3/P_2$ to a smooth projective $G_2$-variety.
In \cite{HL2}, it is shown that this is the only smooth non-isomorphic  specialization of $B_3/P_2$.
It turns out that $B_3/P_2$ is the only exception among all $G/P$ with Picard number one, as shown by the following
\begin{thm}[\cite{Hw97}, \cite{HM98}, \cite{HM02}, \cite{HM}] \label{t.HMRigidity}\footnote{In the statement of \cite{Hw97} there was an error in which $B_3/P_2$ was also stated to be rigid under K\"ahler deformation. The error was not realized in \cite{HM02} and \cite{HM}, and was corrected in \cite{HL2}.}
 A rational homogeneous variety of Picard number 1 is rigid except for $B_3/P_2$.
\end{thm}

The main ingredient for the proof is the VMRT theory developed by Hwang and Mok.  For a uniruled projective manifold $X$ and a fixed family of minimal rational curves $\sK$, the variety of minimal rational tangents (VMRT for short) of $X$ (with respect to $\sK$) at a general point $x$ is the closed subvariety $\sC_x \subset \mbP T_xX$ consisting of all tangent directions at $x$ of curves in $\sK$ passing through $x$. The projective geometry of $\sC_x \subset \mbP T_xX$ encodes a lot of global geometry of $X$ and in some cases, we can reconstruct $X$ from its VMRT at general points, which is the case for all $G/P$ of Picard number one by series of works (\cite{Mok08}, \cite{Hong-Hwang}, \cite{HL1}, \cite{HLT}).

%For rational homogeneous space $G/P$ associated to a long root, it turns out its VMRT is again homogeneous, hence we can use an induction to prove the invariance of VMRT.

Let us recall the main strategy of the proof of Theorem \ref{t.HMRigidity}. As $G/P$ is locally rigid, we need to check the rigidity under specializations $\sX \to \Delta$. The family of minimal rational curves on $G/P \simeq \sX_t$ deforms to a family of minimal rational curves on $\sX_0$.  The key step is to show that the VMRT of $\sX_0$ at a general point  is projectively equivalent to  that of $G/P$, then we can apply the aforementioned theorem to reconstruct $G/P$ from its VMRT.
The rational homogeneous variety $B_3/P_2$ is an exception in Theorem \ref{t.HMRigidity} precisely because its VMRT is isomorphic to $\BP^1 \times
\mathbb{Q}^1 \subset \mathbb{P}^5$, which  can be deformed to the Hirzebruch surface $S_{1,3} = \mbP_{\mbP^1} (\mathcal{O}(1) \oplus \mathcal{O}(3)) \subset \mathbb{P}^5$ as embedded projective subvarieties. Here to remember the projective embedding, we use $\mathbb{Q}^1$ to  denote a smooth conic in $\mathbb{P}^2$.

For higher Picard number case, the rigidity problem becomes even more difficult.  As a simple example, the surface $\mbP^1 \times \mbP^1$ can be specialized to any even-degree
Hirzebruch surface $\mbP_{\mbP^1}(\0(-a) \oplus \0(a))$.  This is the reason that  we will only consider {\em rigidity under Fano deformation}, namely the fibers of $\sX \to B$  are assumed to be Fano.   We feel this is the right framework for higher Picard number cases.  Note that for Picard number one case, the two notions of rigidity are the same.

One of the nice features of Fano specializations is that they preserve product structures, as shown by the second-named author in \cite{Li}. In particular, a product of Fano manifolds is rigid under Fano deformation if and only if  each factor is rigid.
As an immediate corollary,  a product of rational homogeneous varieties of Picard number one is rigid under Fano deformation if and only if  none of its factors is isomorphic to $B_3/P_2$.

It is proven in \cite{WW} that complete flag varieties $G/B$ are rigid under Fano deformation, by using the characterization of complete flag varieties as the only projective Fano manifolds whose elementary Mori contractions are all $\mbP^1$-fibrations. Recently, the second-named author has proven in \cite{Li2} the rigidity under Fano deformation of many rational homogeneous varieties  of higher Picard numbers by using Cartan connections.  However, one should bear in mind that the rigidity under Fano deformation does not always hold, and a well-known  counter-example of higher Picard number (see for example \cite[Example 3.2]{WW}) is $\BP T_{\mbP^{2n}}$, which can be specialized to $\mbP_{\mbP^{2n-1}}(\mathcal{N}(1) \oplus \mathcal{O}(2))$, where $\mathcal{N}$ denotes the null-correlation bundle on $\mbP^{2n-1}$.  It is still an open problem to classify rational homogeneous varieties which are rigid under Fano deformation.
%while for $\BP T^*_{\mbP^n}$ we can show it is rigid under Fano deformation.

The main purpose of this paper is to study the rigidity under Fano deformation of wonderful compactifications of semisimple linear algebraic groups, which gives an answer to a question raised by Bien and Brion in \cite[Section 4.5]{BB96}. For any semisimple adjoint linear algebraic group $G$ of rank $n$, there exists (\cite{dCP}) a unique smooth Fano $G \times G$-equivariant compactification of $G$, say $\bar{G}$ with the following properties: (1) the boundary divisor $\partial \bar{G} = \cup_{i=1}^n D_i$ is a union of $n$ smooth divisors with normal crossing; (2) the $G\times G$-orbit closures  in $\partial \bar{G}$ are exactly intersections of these boundary divisors.  The main result in this paper is the following:

\begin{thm}\label{t.main}
The wonderful compactification $\bar{G}$ of a semisimple linear algebraic group $G$ of adjoint type is rigid under Fano deformation.
\end{thm}

We do expect that the wonderful compactification $\bar{G}$ is rigid under projective deformation, but our method depends   on the Fano condition in a crucial way (cf. Remark \ref{r.Fanoness}(ii)) and it seems difficult to remove this assumption. 

Let us explain the proof. As shown in \cite[Proposition 4.2]{BB96}, $\bar{G}$ is locally rigid, hence we only need to prove the rigidity under Fano specializations.
Writing $G = G_1 \times \cdots \times G_k$ as a product of simple adjoint linear algebraic groups $G_i$, $1 \leq i \leq k$,  we have a decomposition of the wonderful compactification $\bar{G}$ of $G$ as a product of wonderful compactifications $\bar{G}_i$.  By the aforementioned result in \cite{Li},  $\bar{G}$ is rigid under Fano deformation if and only if so is each factor $\bar{G}_i$.  This reduces the proof of Theorem \ref{t.main} to the case where $G$ is simple, a situation which we will focus on from now on.  Note that if $G$ is of rank 1, then $\bar{G}$ is isomorphic to $\mathbb{P}^3$, which is rigid.  Hence in the following, we always assume that the rank of $G$ is at least 2.

The key ingredient to the proof is again the VMRT theory. For the wonderful compactification $\bar{G}$ of a simple adjoint group, it is shown in \cite{BF15} that it has a unique family of minimal rational curves.  When  $G$ is not of type $A$,  the associated VMRT at a general point of $\bar{G}$ is the unique closed orbit $\mbP \msO \subset \mbP \fg$ in the projective adjoint representation of $G$. For type $A_n$, the VMRT  at a general point of $\bar{A}_n$ is projectively equivalent  to the projection from a general point of the Segre variety $\mbP^n \times \mbP^n \subset \mbP^{n^2+2n}$.

Following \cite{HM98},  a regular family in this paper  means a smooth projective map $\sX \to \Delta$.
Now consider  a regular  family of Fano manifolds $\sX \to \Delta$ such that $\sX_t \simeq \bar{G}$ for $t\neq 0$ and $\sX_0$ is Fano.  We first show that the family of minimal rational curves on $\bar{G}$ deforms to a family of minimal rational curves on $\sX_0$ (Proposition \ref{p.MRC}), by using the uniqueness of family of minimal rational curves on $\bar{G}$.
Then the first key step is to show the invariance of VMRT, namely the VMRT of $\sX_0$ at a general point is projectively equivalent to that of $\bar{G}$.
Note that for $G$ not of type $A$, the VMRT $\mbP \msO$ of $\bar{G}$ at a general point is a rational homogeneous space of Picard number one.  Take a section $\sigma: \Delta \to \sX$ through general points, then the normalized VMRTs (which are smooth) along $\sigma$ give a regular family of  projective varieties, with general fibers isomorphic to $\mbP \msO$.
Whence we can apply Theorem \ref{t.HMRigidity} to obtain the invariance of the normalized  VMRT and then the invariance of VMRT (Proposition \ref{p.invVMRT}).
  For type $A_n$, its wonderful compactification $\bar{A_n}$ can be explicitly  constructed  by successive blowups from
  $\mbP {\rm End}(\mathbb{C}^{n+1})$. We extend this construction to $\sX$ and then show the invariance of VMRT through a careful study (Proposition \ref{p.invVMRTtypeA}).

  The above method to prove the invariance of VMRT does not apply  to  $B_3$, as in this case $\mbP \msO =B_3/P_2$, which is not rigid.
   To remedy for this, we consider the variety of lines $\mathbf{L}_{z_0} \subset \BP^5$ on $\mcC_{x_0}$ passing through a general point $z_0$ on it, where $\mcC_{x_0}$ is the VMRT of $\mcX_0$ through a general point $x_0$. It turns out that $\mathbf{L}_{z_0} \subset \BP^5$ is a linearly non-degenerate surface of degree 4, which has only two possibilities:  either $\BP^1 \times \mathbb{Q}^1$ or $S_{1,3}$.  For the former case, we  conclude that the VMRT of $\mcX_0$ is isomorphic to $B_3/P_2$ by using the recognization of $B_3/P_2$ by its VMRT (\cite[Main Theorem]{Mok08}). It remains to exclude the possibility of $\mathbf{L}_{z_0} \simeq S_{1,3}$.
   To that end,  we first give a geometric construction (Proposition \ref{p.OGtoSpinor}) of the wonderful compactification $\bar{B}_n$ by successive blowups from the spinor variety  ${\rm OG}(2n+1, W_1 \oplus W_2)$, where $(W_i, o_i)$ are vector spaces of dimension $2n+1$ endowed with a non-degenerate symmetric quadric form $o_i$. This gives in particular a birational map $\phi: \bar{B}_3 \to \mathbb{S}:={\rm OG}(7, \mathbb{C}^{14})$, which extends to a birational morphism $\Phi: \sX \to \mcS$. The key point is to show that $\mcS$ is an isotrivial family with fibers isomorphic to $\mathbb{S}$.  Passing to the variety of lines on VMRT, this gives an injective
   holomorphic map  $S_{1,3} \simeq \mathbf{L}_{z_0} \to \BP^1 \times \BP^4$, which is a specialization  of the natural embedding $\BP^1\times \mathbb{Q}^1 \to \BP^1 \times \BP^4$.  This gives an embedding of $S_{1,3}$ into $\BP^4$, which is clearly not possible, concluding the proof of the invariance of VMRT (Proposition \ref{p.invVMRTB3}).

To continue,  we borrow an idea from  \cite{Park}, where the rigidity of odd Lagrangian Grassmannians is proven.  The  observation is  that if $H^1(\sX_0, T_{\sX_0})=0$, then $\sX_0$ is locally rigid, hence $\sX_0 \simeq \sX_t$ for small $t$ and  we are done.  To show the vanishing of $H^1(\sX_0, T_{\sX_0})$, we note that $\chi(\sX_t, T_{\sX_t})$ is invariant and $H^i(\sX_t, T_{\sX_t})=0$ for all $i \geq 2$ by the Akizuki-Nakano vanishing theorem (\cite{AN}), which gives
$$h^1(\sX_0, T_{\sX_0}) = h^0(\sX_0, T_{\sX_0}) - \chi(\sX_t, T_{\sX_t}) = h^0(\sX_0, T_{\sX_0})- 2\dim G.$$
We can now use VMRT theory to bound $h^0(\sX_0, T_{\sX_0})$, which gives that if $G$ is not of type $C$, then $h^0(\sX_0, T_{\sX_0}) \leq  2\dim G+1$, with equality if and only if the VMRT structure is locally flat.  As a consequence, when $\fg$ is not of type $C$, we have that  either $\mcX_0$ is isomorphic to $\bar{G}$ or $\mcX_0$ is an equivariant compactification of the vector group $\mathbb{G}_a^g$ with $g = \dim G$.

To exclude the latter case, we observe that if  $\mcX_0$ is an equivariant compactification of a vector group, then its Picard group is freely generated by its boundary divisors.  By the invariance of pseudo-effective cones of divisors under Fano deformation (cf. Theorem \ref{t.InvCones}), the Picard group of $\bar{G}$ is generated by the boundary divisors of $\bar{G}$.
As the boundary divisors of $\bar{G}$ correspond to simple roots of $G$ and ${\rm Pic}(\bar{G})$ coincides with the  weight lattice,
this implies that the root lattice and the weight lattice of $G$ are the same, which is not possible except for $G_2, F_4$ or $E_8$, whence this proves the rigidity if $G$ is not $C_n, G_2, F_4$ or $E_8$ (Theorem \ref{t.rigidity}). For $G_2, F_4$ and $E_8$, we take another approach to exclude the equivariant compactification case.
Let $T$ be a maximal torus of $G$, and let $T'$ be the subtorus of dimension two that is annihilated by $\alpha_2,\ldots,\alpha_{n-1}$, where $\alpha_1,\ldots,\alpha_n$ is a set of simple roots in Bourbaki's numbering order. Applying the results on fixed point schemes by Fogarty \cite{F}, we obtain a family of smooth projective surfaces $\mcY'/\Delta$ such that the central fiber is an equivariant compactification of $\mathbb{G}_a^2$ while the general fiber is isomorphic to $\bar{T'}$, the closure of $T'$ in the wonderful compactification $\bar{G}$.  The stabilizer $W_{T'}$ of $T'$, under the action of the Weyl group $W(G)$ on $T$, acts on the family $\mcY'/\Delta$ and stabilizes the open orbit of each fiber $\mcY'_t$. There is a subgroup $W'$ of $W_{T'}$ such that the set of prime boundary divisors of $\mcY'_0$ consists of two $W'$-orbits, and each orbit contains at least two elements (Proposition \ref{p.W' action on Y'0}). This is impossible, because there does not exist any equivariant compactification of  $\mathbb{G}_a^2$ admitting such a finite group action (Lemma \ref{p.vector compact surface}), whence we conclude the proof of the rigidity for wonderful compactifications of  $G_2, F_4$ and  $E_8$ (Theorem \ref{t.rigidity_EFG}).

It remains to prove the rigidity under Fano deformation for the wonderful compactification of $C_n$, for which  we use  the theory of spherical varieties.   Let $(W_i, \Omega_i)$ be two  symplectic vector spaces of dimension $2n$ and ${\rm LG}(2n, W_1 \oplus W_2)$ the Lagrangian Grassmannian. One notices that both $\bar{C_n}$ and ${\rm LG}(2n, W_1 \oplus W_2)$ have the same locally flat VMRT-structure.
By using the theory of spherical varieties, we show that there exists a birational morphism $\phi: \bar{C_n} \to  Z:={\rm LG}(2n, W_1 \oplus W_2)/\tau$ which is the composition of successive blowups along explicit strata, where $\tau$ is the involution induced from  $(1_{W_1}, -1_{W_2}) \in {\rm Sp}(W_1) \times {\rm Sp}(W_2)$.  The morphism $\phi$ extends to a morphism $\Phi: \sX \to \sZ$ by the invariance of Mori cones under Fano deformation (cf. Theorem \ref{t.InvCones}).  We then show that $\sZ_0$ is in fact isomorphic to $Z$, making $\sZ \to \Delta$ an isotrivial family, with a $G\times G$-action. As $\Phi_0$ is birational, the vector fields of $G\times G$-action on $\sZ_0$ lift to an open subset of $\sX_0$, which coincide with vector fields on $\sX|_{\Delta^*}$ coming from the $G \times G$-action.  By Hartogs' extension theorem, these vector fields extend to the whole of $\sX$, making $\sX_0$ a $G \times G$-variety.  It follows that $\sX_0$ is a spherical variety under the $G \times G$-action. Then we show that its colored cone is the same as that of $\bar{C_n}$, which concludes our proof (Theorem \ref{t.RigidityTypeC}).

It is worthwhile to point out that the construction of wonderful compactifications of $\bar{B}_n$ by successive blowups from the spinor variety   ${\rm OG}(2n+1, W_1 \oplus W_2)$ also works for type $D$. Together with similar results for  $A_n$ and $C_n$, we obtain explicit constructions of wonderful compactifications of all classical types by successive blowups, which could be of independent interest. As pointed out by Michel Brion, there are constructions of smooth log homogeneous compactifications of classical groups by successive blowups  in \cite{Hu}.  It turns out that our construction of $\bar{B}_n$ is the same as that in \cite{Hu}, while the other cases can be deduced from constructions in \cite{Hu}.  For example, the wonderful compactification of ${\rm Sp}(2n)$ is constructed by successive blowups in \cite{Hu} and we can deduce from it the similar construction for the wonderful compactification of ${\rm PSp}(2n)$ (Section 5.3).

%We have kept our construction in this paper as it is more transparent and we hope this makes our paper more self-contained.

The paper is organized as follows: after recalling the study of minimal rational curves on wonderful compactifications of simple adjoint groups from \cite{BF15} in Section 2, we prove the invariance of VMRT (except $B_3$) in Section 3, where the difficulty is to deal with  $A_n$ case, as its VMRT is a projected Segre variety.  In Section 4, we prove Theorem \ref{t.main} for all $G$  except $C_n$.  We first  show that the central fiber $\mcX_0$  is an equivariant compactification of a vector group if it is not isomorphic to $\bar{G}$.  Then we use a careful study to exclude the case of equivariant compactifications of vector groups.  Section 5 is devoted to prove that the wonderful compactification of $C_n$ is rigid,  where we mainly use the theory of spherical varieties.   In the last section, we prove the invariance of VMRT  for $B_3$. \\

{\em Convention:}  Throughout this paper,  by {\em open set} without further qualification we  mean {\em Zariski open and non-empty set}, while we use {\em analytic open set} means it is open in the complex topology.

{\em Acknowledgements:} We are very grateful to Michel Brion and  Jun-Muk Hwang for  helpful discussions and suggestions.
  We would like to thank the six anonymous  referees  for the careful readings and numerous suggestions,  which help to improve our presentation and  lead to the new Section 6.2.
Baohua Fu is supported by the NSFC grant no. 12288201 and CAS Project for Young Scientists in Basic Research, grant no. YSBR-033. Qifeng Li is supported by the NSFC grant No. 12201348.

%
%The aim of this paper is to show the following
%
%\begin{thm}\label{thm. global rigidity}
%Let $\pi: \mcX\rightarrow\mcY$ be a holomorphic family of connected complex projective manifolds, where $\mcY$ is a connected complex variety. If there exists $y_0\in\mcY$ such that the fiber $\pi^{-1}(y_0)$ is the wonderful compactification of a simple algebraic group of adjoint type, then $\pi^{-1}(y)\cong\pi^{-1}(y_0)$ for all $y\in\mcY$.
%\end{thm}
%
%\begin{rmk}\label{rmk. local rigidity}
%Let $X$ be the wonderful compactification of a simple algebraic group $G$ of adjoint type. Then $X$ is a Fano manifold of Picard number $n$ (cf. \cite[Section 1]{BF15}), where $n$ is the rank of $G$. Hence $H^i(X, T(X))=0$ for $i\geq 1$ by \cite[Proposition 4.2]{BB96}. It follows that $X$ is locally rigid.
%\end{rmk}
%
%By the local rigidity of $X$, it suffices to prove the following
%
%\begin{thm}\label{thm. limit rigidity}
%Let $\pi: \mcX\rightarrow\Delta\ni 0$ be a holomorphic family of connected complex projective manifolds such that $\mcX_t\cong X$ for all $t\neq 0$, where $X$ is the wonderful compactification of a simple algebraic group $G$ of adjoint type. Then $\mcX_0\cong X$ too.
%\end{thm}

\section{Minimal rational curves on wonderful compactifications of simple groups}

\subsection{Minimal rational curves and geometric structures}

For a uniruled projective
manifold $X$, let $\RatCurves^\n(X)$
denote the normalization of the space of rational curves on $X$
(see \cite[II.2.11]{Kollar}). Every irreducible component $\sK$ of
$\RatCurves^\n(X)$ is a (normal) quasi-projective variety equipped with
a quasi-finite morphism to the Chow variety of $X$; the image consists
of the Chow points of irreducible, generically reduced rational curves.
There is a universal family $\sU$ with projections
$\upsilon : \sU \to \sK$, $\mu : \sU \to X$, and $\upsilon$ is
a $\bP^1$-bundle (for these results, see
\cite[II.2.11, II.2.15]{Kollar}).

For any $x \in X$, let $\sU_x := \mu^{-1}(x)$ and $\sK_x := \upsilon(\sU_x)$. We call
 $\sK$ a \emph{family of minimal rational curves}
if $\sK_x$ is non-empty and projective for a general point $x$.
There is a rational map $\tau_x: \sK_x \dasharrow \bP T_x X$
(the projective space of lines in the tangent space at $x$)
that sends any curve which is smooth at $x$ to its tangent direction.
The closure of the image of $\tau_x$ is denoted by $\sC_x$
and called the \emph{variety of minimal rational tangents}
(VMRT) at the point $x$.
By \cite[Thm.~1]{HM2} and \cite[Thm.~3.4]{Kebekus},
composing $\tau_x$ with the normalization map $\sK_x^\n \to \sK_x$
yields the normalization of $\sC_x$. Also, $\sK_x^\n$ is
a union of components of the variety $\RatCurves^\n(x, X)$ defined
in \cite[II.2.11.2]{Kollar}, and hence is smooth for a general point  $x \in X$ by \cite[II.3.11.5]{Kollar}. In this case, $\sU_x \simeq \sK_x^n$ is smooth and the rational map $\tau_x$ induces  a birational morphism $\sU_x \simeq \sK_x^n \to \sC_x$, which is still denoted by $\tau_x$ by abuse of notation. Since $\mcU_x$ is both the normalization of $\sK_x$ and that of $\mcC_x$, we call $\mcU_x$ the normalized Chow space or the normalized VMRT.
%By abuse of notation, we denote by  $\sK_x$ its normalization, which is then smooth with a birational morphism $\tau_x: \sK_x \to \sC_x$.

 The closure $\sC \subset \BP T X$ of the union of  $\sC_x \subset \BP T_xX$  is the {\em VMRT-structure} on $X$. The
natural projection $\sC \to X$ is a proper surjective morphism.
The VMRT-structure $\sC$ is said to be {\em locally flat} if for a general $x \in X$,  there exists  an analytical open subset $U$  of $X$ containing $x$ with an open immersion $\phi: U \to
 \C^n, n= \dim X,$ and a projective subvariety $S \subset \BP^{n-1}$
with $\dim S= \dim \sC_x$ such that $\phi_*: \BP T U \to \BP
T \C^n$ maps $\sC|_{U}$ into the trivial fiber subbundle $\C^n
\times S$ of the trivial projective bundle $\BP T \C^n = \C^n \times \BP^{n-1}.$

\begin{example} \label{e.contact}
Let $\fg$ be a simple Lie algebra on which its adjoint group $G$ acts by adjoint action. In the projective space $\BP \fg$, there is a unique closed $G$-orbit
$\BP \mathscr{O}$, which is a homogeneous Fano contact manifold. When $\fg$ is not of type $A$, the variety $\BP \mathscr{O}$ is of Picard number one, hence of the form $G/P$ for some maximal parabolic subgroup $P$ of $G$. The contact structure on $\BP \mathscr{O}$  gives a corank one subbundle $\sW \subset T G/P$ and the isotropy action of $P$ on $T_o G/P$ at the base point $o$ induces an irreducible representation on $\sW_o$. It turns out (cf. \cite[Proposition 1]{HM02}) that the VMRT of $\BP \mathscr{O}$
is the highest weight variety of this representation. In particular, the VMRT of $\BP \mathscr{O}$ is linearly degenerate in this case. For example, when $\fg$ is of type $B_n$ (with $n \geq 3$), the orbit $\BP \mathscr{O}$ is isomorphic to $B_n/P_2$, which is the variety of lines on the smooth hyperquadric  $\mathbb{Q}^{2n-1}$.  Its VMRT is given by the Segre embedding of $\BP^1 \times \mathbb{Q}^{2n-5}$ to a hyperplane in $\BP T_o G/P$.
\end{example}

\begin{rmk} \label{r.freecurves}
 Given an irreducible component $V$ of $\RatCurves^\n(X)$, there is a proper closed subvariety $N$ of $X$ such that if $C$ is a rational curve represented by an element in $V$ and if $C$ is not contained in $N$, then $C$ is a free rational curve (see \cite[Proof of II.2.11]{Kollar}). In other words, a rational curve through a general point is free.
Let $\FRC(X)$ be the open subset of $\RatCurves^\n(X)$ parameterizing free rational curves on $X$. 
%For each $x\in X$, let $\FRC(x, X)$ be the open subset of $\RatCurves^n(x, X)$ parameterizing free rational curves on $X$ passing through the point $x$. 
For an irreducible component $\mathcal{F}$ of $\RatCurves^n(X)$, if it intersects  $\FRC(X)$, then it will contain an irreducible component of   $\FRC(X)$ as an open subset. Conversely, given an irreducible component of $\FRC(X)$, it is an open subset of an irreducible component of $\RatCurves^n(X)$ which is a dominant famliy.

%
%
%its intersection with $\FRC(X)$ is either empty or an irreducilbe component of $\FRC(X)$. In the first case, $\mathcal{F}$ is not a dominant family. In the second case, $\mathcal{F}$ is a dominnat family, and its intersection with $\FRC(X)$ is  an open subset of $\mathcal{F}$. Conversely, given an irreducible component of $\FRC(X)$, it is an open subset of an irreducible component of $\RatCurves^n(X)$ which is a dominant famliy.
\end{rmk}

\begin{defi} \label{d.prolong}
Let $V$ be  a complex vector space and let $\fg \subset {\rm End}(V)$ be
a Lie subalgebra. The {\em $k$-th prolongation} (denoted by
$\fg^{(k)}$) of $\fg$ is the space of symmetric multi-linear
homomorphisms $A: \Sym^{k+1}V \to V$ such that for any fixed $v_1,
\cdots, v_k \in V$, the endomorphism $A_{v_1, \ldots, v_k}: V \to
V$ defined by $$v\in V \mapsto A_{v_1, \ldots, v_k, v} := A(v,
v_1, \cdots, v_k) \in V$$ is in $\fg$. In other words, $\fg^{(k)}
= \Hom(\Sym^{k+1}V, V) \cap \Hom(\Sym^kV, \fg)$.
\end{defi}

We are interested in the prolongation of infinitesimal automorphisms of a projective subvariety.  Let $S \subsetneq \BP V$
be a smooth linearly non-degenerate  projective subvariety and let $\hat S \subset V$ be the corresponding
affine cone.  Let $\aut(\hat{S})\subset\mathfrak{gl}(V)$ be the Lie algebra
of infinitesimal linear automorphisms of $\hat S$, which is the subalgebra of $\mathfrak{gl}(V)$ preserving $\hat{S}$.
It is shown in \cite{HM} that  the second prolongation of $\aut(\hat{S})$ always vanishes, namely $\aut(\hat{S})^{(2)}=0$.

For a uniruled projective manifold $X$ with a VMRT structure $\mathcal{C} \subset \BP TX$,
we denote by $\aut(\mathcal{C}, x)$ the Lie algebra of infinitesimal automorphisms of $\mathcal{C}$, which consists of germs of vector fields on $X$ whose local flow preserves $\mathcal{C}$ near $x$.
Note that the action of ${\rm Aut}^0(X)$ on $X$ sends minimal rational curves to minimal rational curves, hence it preserves the VMRT structure, which gives a natural inclusion
$\aut(X) \subset \aut(\mathcal{C}, x)$ for $x\in X$ general.

The following result is a combination of Propositions 5.10, 5.12 and 5.14 in \cite{FH}.

\begin{prop} \label{p.prolong}
Let $X$ be a smooth Fano variety. Assume that the  VMRT   $\mathcal{C}_x \subsetneq \BP T_xX$ at a general point $x\in X$ is smooth irreducible and linearly non-degenerate. Then
$$
\dim \aut(\mathcal{C}, x) \leq \dim X + \dim \mathfrak{aut}(\hat{\mathcal{C}_x})+\dim \mathfrak{aut}(\hat{\mathcal{C}_x})^{(1)}.
$$
The equality holds if and only if the VMRT structure $\mathcal{C}$ is locally flat.
\end{prop}

\begin{example} \label{e.IHSS}
An irreducible Hermitian symmetric space (IHSS for short) is a rational homogeneous variety $G/P$ of Picard number one such that the isotropic representation of $P$ on $T_o (G/P)$ is irreducible.   The highest weight variety of  this representation is the VMRT $\mcC_o$ of $G/P$ at the base point $o$. The VMRT-structure on $G/P$ is locally flat and the equality in Proposition \ref{p.prolong} holds (see for example \cite{FH}). Furthermore, in this case we have $\mathfrak{aut}(\hat{\mathcal{C}_x})^{(1)} \simeq T^*_o(G/P)$.
The rank of $G/P$ is the least number $r$ such that a general point of $\BP T_o G/P$ is contained in the linear span of $r$ points on $\mcC_o$. In particular, $G/P$ is of rank 1 if and only if it is isomorphic to a projective space.
The following is the list of IHSS with their VMRT and ranks.

\begin{center}
%\begin{tabular*}{0.908\textwidth}{|c| c| c| c| c| c|}
\begin{tabular}{|c| c| c| c| c| c| c| }
\hline  IHSS  $G/P$ & ${\rm Gr}(a, a+b)$ & $D_n/P_n$   & $C_n/P_n$ &    $\mathbb{Q}^r$ & $E_6/P_1$ & $E_7/P_7 $ \\
\hline  VMRT $\mathcal{C}_o$   &  $\BP^{a-1} \times \BP^{b-1}$ &     ${\rm Gr}(2, n)$    &$\BP^{n-1}$  & $\mathbb{Q}^{r-2}$ & $D_5/P_5$ & $E_6/P_1$  \\
\hline  $\mathcal{C}_o \subset \BP T_o(G/P)$ & Segre & Pl\"ucker &second Veronese  & Hyperquadric  & Spinor  & Severi\\
\hline  rank of $G/P$ & {\rm min} $\{a, b\}$ & $[\frac{n}{2}]$ & $n$ &  2  & 2 & 3 \\
\hline
\end{tabular}
\end{center}
\end{example}

\begin{defi}\label{d.GStructure}
Let $M$ be an $m$-dimensional complex manifold and let $G \subset {\rm GL}(\mathbb{C}^m)$ be a complex Lie subgroup.

(i) The {\em frame bundle} of $M$ is the principal ${\rm GL}(\mathbb{C}^m)$-bundle $\mathcal{F}(M)$, whose fiber at a point $x\in M$ is $\mathcal{F}(M)_x = {\rm Isom}(\mathbb{C}^m, T_xM)$.

(ii) A {\em $G$-structure} on $M$ is a $G$-principal subbundle $\mathcal{G}$ of $\mathcal{F}(M)$.

(iii) A $G$-structure $\mathcal{G}$ on $M$ is said to be  {\em locally flat} if for $x \in M$, there exists an analytic open subset $U_x \subset M$ such that
the restricted $G$-structure $\mathcal{G}|_{U_x}$ is equivalent to the trivial $G$-structure on some analytic open subset of $\mathbb{C}^m$.
\end{defi}

Consider a uniruled projective manifold $X$ of dimension $m$ with a minimal rational component $\mathcal{K}$ of minimal rational curves.  Assume  there exists a subvariety
$S \subset \BP^{m-1}$ such that the VMRT $\mathcal{C}_x$ is projectively equivalent to $S \subset \BP^{m-1}$ for all $x$ in an analytic open subset $X^o$ of $X$. In this case, the VMRT structure $\mathcal{C} \subset \BP TX$ induces a $G$-structure $\mathcal{G}$ on $X^o$, given by $\mathcal{G}_x ={\rm Isom}(\hat{S}, \hat{\mathcal{C}_x})$, where $G$ is the linear automorphism group of $\hat{S} \subset \mathbb{C}^m$.  One notices that the local flatness of VMRT-structure and that of  its associated $G$-structure are equivalent.

\begin{example} \label{e.IHSS2}
Let $G/P$ be an IHSS of rank $\geq 2$ with $G = {\rm Aut}^0(G/P)$. Let $G_0 \subset {\rm GL}(T_o G/P)$ be the image of the representation of $P$.
Then $G/P$ carries naturally a $G_0$-structure $\mcG_o$ coming from the VMRT-structure, which is locally flat by Example \ref{e.IHSS}.  Let $\fg = \fg_{-1} \oplus \fg_0 \oplus \fg_1$ be the grading associated to $P$. It is known that $\fg_0$ is the Lie algebra of $G_0$, $\mathfrak{p} = \fg_{-1} \oplus \fg_0$,  $\fg_{-1} = \aut(\hat{\mcC}_o)^{(1)}$ and $\fg_1 \simeq T_o G/P$.
It follows from \cite[Proposition 5.14]{FH} that $\aut(\mcG_o, o) \simeq \fg = \fg_{-1} \oplus \fg_0 \oplus \fg_1$ and the set of vector fields vanishing at $o$ is isomorphic to $\mathfrak{p}$.
\end{example}

\begin{rmk} \label{r.LocFlatGStructure}
There is a cohomological way to detect the local flatness of a $G$-structure. To wit,
for a  $G$-structure $\mcG$, we can define certain vector-valued functions $c^k, k=0, 1,2, \cdots$ on $\mcG$, the vanishing of which implies the local flatness of $\mcG$ (cf. \cite{Gu}). The $G_0$-structure on an IHSS $G/P$ of rank $\geq 2$ is locally flat, and the corresponding functions $c^k$ vanish (cf. proof of Proposition 5.12 in \cite{FH}).
\end{rmk}

\subsection{Wonderful compactifications of simple groups}\label{s.wonderful compact}

Let us first recall some basic constructions and properties of wonderful compactifications of simple adjoint groups from \cite{BK} and \cite{dCP}.

Let $G$ be a simple linear algebraic group
of adjoint type and of rank $n$, and let $\fg$ be the
corresponding Lie algebra.  Fix a Borel subgroup
$B \subset G$ as well as a maximal torus $T \subset B$.
 We denote by $R$ the root system of $(G,T)$ and
by $R^+ \subset R$ the subset of positive roots consisting of
roots of $B$. The corresponding set of simple roots is denoted
by $\{ \alpha_1, \cdots, \alpha_n \}$, where we use Bourbaki's numbering order
for simple roots. The half-sum of positive
roots is denoted by $\boldsymbol \rho$.  Let $\boldsymbol \theta$ be the highest root of $R$.  We denote by $P_i \subset G$ the standard maximal parabolic subgroup corresponding to $\alpha_i$. Then $G/P_i$ is a rational homogeneous variety of Picard number one.

%The coroot of any $\alpha \in R$ is denoted by $\alpha^\vee$;
%this is a one-parameter subgroup of $T$.
%The coroots form the dual root system $R^\vee$. The pairing
%between characters and one-parameter subgroups is denoted by
%$\langle, \rangle$; we have
%$\langle \alpha, \alpha^\vee \rangle = 2$ for any $\alpha \in R$.

We denote by $\Lambda$ the weight lattice, with the submonoid
$\Lambda^+$ of dominant weights, and the fundamental weights
$\omega_1,\ldots,\omega_n$. For any $\lambda \in \Lambda^+$,
we denote by $V_\lambda$ the irreducible representation of
the simply-connected cover of $G$ with highest weight
$\lambda$. This gives a projective representation
\[ \varphi_\lambda : G \to \PGL(V_\lambda). \]
Moreover, the $G$-orbit of the highest weight line in $V_\lambda$
yields the unique closed orbit in the projectivization $\mbP V_\lambda$;
it is isomorphic to $G/P_\lambda$, where the parabolic subgroup
$P_\lambda$ only depends on the type of $\lambda$, i.e.,
the set of simple roots that are orthogonal to that weight.

Note that we have a natural open embedding ${\rm PGL}(V_\lambda) \subset \mbP \End V_\lambda$.
The closure of the image of $\varphi_\lambda$ in $\mbP \End V_\lambda$ will be denoted
by $X_\lambda$. This is a projective variety on which
$G \times G$ acts via its action on $\mbP \End V_\lambda$ by
left and right multiplication.

%Moreover, $X_\lambda$ contains
%a unique closed orbit of $G \times G$; it is isomorphic to
%$G/Q_\lambda \times G/P_\lambda$, where $Q_\lambda$ denotes
%the parabolic subgroup containing $T$ and opposite to $P_\lambda$.

When the dominant weight $\lambda$ is regular, $X_\lambda$
turns out to be smooth and independent of the choice of
$\lambda$; this defines the wonderful compactification
$X$ of $G$, which is sometimes denoted by $\bar{G}$. The identity component of $\Aut(X)$ is $G\times G$. The boundary $\partial X := X \setminus G$ is a union of $n$ smooth
irreducible divisors $D_1, \cdots, D_n$ with simple normal
crossings. The $G \times G$-orbits in $X$ are
indexed by the subsets of $\{1,\ldots, n \}$, by assigning
to each such subset $I $
the unique open orbit in the partial intersection
$D_I:=\cap_{i\in I} D_i$. The orbit closure $D_I$ is equipped with a $G\times G$-equivariant fibration $f_I: D_I\rightarrow G/P_\lambda^-\times G/P_\lambda$, where $\lambda:=\sum_{i\in I}\omega_i\in\Lambda^+$ and $P_\lambda^-$ is the opposite parabolic subgroup of $P_\lambda$.  The fiber of $f_I$ at the base point is isomorphic to the wonderful compactification of the adjoint group of $L_I:=P_\lambda\cap P_\lambda^-$ (a Levi subgroup of both).  In particular, there exists a unique closed $G \times G$-orbit $D_{1,2,\cdots, n}:=\cap_{i=1}^n D_i$, which is isomorphic to $G/B^- \times G/B$.
%
%In particular,
%the open orbit $X_0$ is $\msO_{\emptyset} = (G \times G)/ \diag(G)$,
%and the closed orbit is $\msO_{\{1,\cdots, n\}}$.
%Each orbit closure is equipped with a $G \times G$-equivariant
%fibration
%\[ f_I : \overline{\msO_I} \to G/Q_I \times G/P_I, \]
%where $P_I$ denotes the parabolic subgroup associated
%with the dominant weight $\sum_{i \in I} \omega_i$,
%and $Q_I$ stands for the opposite parabolic subgroup.
%The fiber of $f_I$ at the base point of $G/Q_I \times G/P_I$
%is isomorphic to the wonderful compactification of the
%adjoint group of $L_I := P_I \cap Q_I$ (a Levi subgroup of both).
%In particular, the closed orbit $\sO_{1,\ldots,n}$ is isomorphic to
%$G/B \times G/B$.

For an arbitrary $\lambda$, the variety $X_\lambda$ may be
singular. The homomorphism $\varphi_\lambda$ extends to a
$G \times G$-equivariant morphism $X \to X_\lambda$ that
we shall still denote by $\varphi_\lambda$. 
Recall that every nef line bundle on $X$ is globally generated.
The pull-backs
$\sL_X(\lambda) := \varphi_\lambda^* \sO_{\mbP \End V_\lambda}(1)$,
$\lambda \in \Lambda^+$, are exactly the globally generated
line bundles on $X$; moreover, $\sL_X(\lambda)$ is ample
if and only if $\lambda$ is regular. In particular,
$X$ admits a unique minimal ample line bundle, namely,
$\sL_X(\boldsymbol \rho)$. The assignment $\lambda \mapsto \sL_X(\lambda)$
extends to an isomorphism
$\Lambda \stackrel{\cong}{ \to} \Pic(X)$.

We shall index the boundary divisors so that
$\sO_X(D_i) = \sL_X(\alpha_i)$ for $i = 1, \ldots, n$.
 The anti-canonical bundle of $X$ is given by
$-K_X = \sL_X(2{\boldsymbol \rho}+ \sum_i \alpha_i)$, which is in particular ample, hence $X$ is Fano.
By \cite{BB96}, $X$ satisfies $H^1(X, T_X)=0$, hence it is locally rigid.

\begin{example}
If $G$ is of type $A_1$, i.e. $G=\PGL(\C^2)$,  its  wonderful compactification $X$ is the projective space $\BP(\End(\C^2))=\BP^3$. In this case, $\sL_X(\omega_1)=\mcO_{\BP^3}(1)$, $\mcO_X(D_1)=\mcO_{\BP^3}(2)$,  and $-K_X=\mcO_{\BP^3}(4)$. Furthermore, the projective space is known to be rigid (under K\"ahler deformation) by the characterization of Kobayashi-Ochiai (\cite{KO}).
%(in fact $\BP^n$ is known to have a unique K\'ahler structure by  Hirzebruch-Kodaira and Yau).
\end{example}

\subsection{Minimal rational curves on wonderful group compactifications}

For any $\alpha \in R$, we denote by $U_\alpha$ the corresponding
root subgroup of $G$ (with Lie algebra the root subspace
$\fg_\alpha \subset \fg$)
and by $C_\alpha$ the closure of $U_\alpha$ in $X$.  This gives a rational curve on $X$, which is in general not minimal.
The main result in \cite{BF15} shows that for the highest root $\boldsymbol \theta$, the deformations of $C_{\boldsymbol \theta}$ form the unique minimal rational component of $X$. More precisely, we have
\begin{thm}[\cite{BF15}]\label{t.VMRTWonderful}
Let $X$ be the wonderful compactification of a simple algebraic
group $G$ of adjoint type. Let $e \in G$ be the identity element and $\boldsymbol \theta$ the highest root.  Then

\item{\rm (i)} $C_{\boldsymbol \theta}$ is the unique $B$-stable irreducible curve on $X$ through $e$.

\item{\rm (ii)} There exists a unique family of minimal rational curves
$\sK$ on $X$, and it consists of deformations of $C_{\boldsymbol\theta}$. Moreover, $\sK_e$ is smooth and the normalization
map $\tau : \sK_e \to \sC_e$ is an isomorphism.

\item{\rm (iii)} $\sC_e$ is the unique closed $G$-orbit in $\bP \fg$,
if $G$ is not of type $A$.

\item{\rm (iv)} When $G$ is of type $A_n$,
so that $G = \PGL(V)$ for a vector space $V$ of dimension $n + 1$,
the VMRT $\sC_e$ is the image of $\bP V \times \bP V^*$
under the Segre embedding $\bP V \times \bP V^* \to \bP \End(V)$,
followed by the projection
$\bP \End(V) \dasharrow \bP(\End(V)/ \C \id) \simeq \bP \fg$.
\end{thm}

%For convenience of discussion later, let us write down the VMRT explicitly.
%
%\begin{eqnarray}\label{eqn. VMRT information of wonderful compactification}
%\begin{tabular}{|c|c|c|c|c|}\hline
%Type of $G$ & VMRT & embedding & complete linear system? & $\tfg^{(1)}=0$? \\ \hline
% $A_1$ & $\mbP^2$ & $\mcO(1)$ & Yes & No \\ \hline
% $A_k, k\geq 2$ & $\mbP^k\times\mbP^k$ & $\mcO(1, 1)$ & of codimension 1 & {\color{red}To check} \\ \hline
% $B_k, {\color{red}k\geq 3}$ & $B_k/P_2$ & $\mcO(1)$ & Yes & Yes \\ \hline
% $C_k, {\color{red}k\geq 2}$ & $C_k/P_1$ & $\mcO(2)$ & Yes & No \\ \hline
% $D_k, k\geq 4$ & $D_4/P_2$ & $\mcO(1)$ & Yes & Yes \\ \hline
% $E_6$ & $E_6/P_2$ & $\mcO(1)$ & Yes & Yes \\ \hline
% $E_7$ & $E_7/P_1$ & $\mcO(1)$ & Yes & Yes \\ \hline
% $E_8$ & $E_8/P_8$ & $\mcO(1)$ & Yes & Yes \\ \hline
% $F_4$ & $F_4/P_1$ & $\mcO(1)$ & Yes & Yes \\ \hline
% $G_2$ & $G_2/P_2$ & $\mcO(1)$ & Yes & Yes \\ \hline
%\end{tabular}
%\end{eqnarray}

The following is a reformulation of Remark 3.6 in \cite{BF15}.

\begin{cor}\label{c.minimal}
For any irreducible curve $C\subseteq X$ such that $C\nsubseteq\partial X$ and for any nef line bundle $L$ on $X$, we have $L\cdot C_{\boldsymbol \theta} \leq L\cdot C$. If $L$ is moreover ample, then $L\cdot C_{\boldsymbol \theta} = L\cdot C$ if and only if $C\in\sK$.
\end{cor}

\begin{proof}
We may assume $C$ passes through the point $e \in G \subset X$. By Remark 3.6 of \cite{BF15}, $C$ is rationally equivalent to an effective $B$-stable $1$-cycle $C_1$ through $e$, which is equal to $m C_{\boldsymbol \theta} + C'$ for some positive integer $m$ and
some effective $B$-stable cycle $C'$ not containing $e$ by Theorem \ref{t.VMRTWonderful}(i).
 Since $L$ is nef, we have $L\cdot C =m L\cdot C_{\boldsymbol \theta} + L \cdot C'\geq mL\cdot C_{\boldsymbol \theta} \geq L\cdot C_{\boldsymbol \theta} $. When $L$ is moreover ample,  $L\cdot C'=0$ if and only if $C'=0$, completing the proof.
\end{proof}

When $\fg$ is not of type $A$, the VMRT of $X$ at a general point  is isomorphic to the unique closed $G$-orbit in $\BP \fg$ by Theorem \ref{t.VMRTWonderful}(iii).
 Let $\msO \subset \mfg$ be the minimal nilpotent orbit, then $\BP \msO \subset \BP \fg$ is the unique  closed $G$-orbit in $\BP \fg$.
We are now going to study some geometrical properties of $\msO$.
For $x \in \msO$, the tangent space $T_x\msO$ is naturally identified with $[\mfg, x]$.  We consider the following affine subvariety in $\wedge^2 \fg$:
$$
\msT_\msO := \{x \wedge [z, x] | x\in \msO, z\in \mfg\}  \subset \wedge^2 \mfg.
$$

\begin{lem}\label{l.nondegenerate}
The subvariety $\msT_\msO \subset \wedge^2 \mfg$ is linearly non-degenerate.
\end{lem}
\begin{proof}

Note that for $\mfg =\mathfrak{sl}_2$,  $\mbP \msO \subset \mbP \mfg$ is the conic curve in $\mbP^2$, whose variety of tangential lines is linearly non-degenerate.
Now consider the general case. For any $x \in \msO$,  there exists an $\mathfrak{sl}_2$ triplet $(x, y, h)$ by the Jacobson-Morozov theorem, namely the Lie sub algebra  $\mathfrak{l}:=\C \langle x, y, h \rangle$ is isomorphic to $\mathfrak{sl}_2$. 
 Let $\mathfrak{n} \subset \mathfrak{l}$ be the set of all non-zero nilpotent elements, then the corresponding subvariety $\msT_\mathfrak{n}:= \{x \wedge [z, x] | x\in \mathfrak{n}, z\in \mathfrak{l}\}  \subset \wedge^2 \mathfrak{l}$ associated to $\mathfrak{n}$ is linearly non-degenerate in $\wedge^2\mathfrak{l}$. 
 As a consequence,   the vector $x \wedge y$ is in the linear span of $\msT_\mathfrak{n}$, hence also in the linear span of $\msT_\msO \subset \wedge^2 \mfg$.  As $\msT_\msO$ is $G$-invariant, we get that $G \cdot (x \wedge y)$ is contained in the linear span of $\msT_\msO$.

On the other hand, $G \cdot (x, y) \subset \msO \times \msO$ is dense by \cite{KY} (p. 69).   In particular, $G \cdot (x \wedge y)$ is dense in $\{x \wedge x' | x, x'\in \msO\}$.
As $\msO \subset \mfg$ is linearly non-degenerate,  the set $\{x \wedge x' | x, x'\in \msO\}$ is linearly non-degenerate in $\wedge^2 \mfg$.  This implies that $G \cdot (x \wedge y)$ is linearly non-degenerate in $\wedge^2 \mfg$, concluding the proof.
\end{proof}

To conclude this section, we prove a geometrical property of the VMRT of wonderful compactifications.  For a smooth projective subvariety $Z \subset \BP V$, the variety of tangential lines of $Z$ is the subvariety $\mcT_Z \subset {\rm Gr}(2, V)  \subset \BP \wedge^2 V$ consisting of tangential lines of $Z$.
\begin{prop} \label{p.nondegenerate}
Let $Z \subset \BP \fg$ be the VMRT at a general point of the wonderful compactification $X$. Then the variety of tangential lines $\mcT_Z \subset \BP \wedge^2 \fg$ is linearly non-degenerate.
\end{prop}
\begin{proof}
By Theorem \ref{t.VMRTWonderful}, we have  $\BP \msO \subset Z \subset \BP\fg$ in all cases. Then the variety of tangential lines $\mcT_Z \subset \BP \wedge^2 \fg$ is linearly non-degenerate by Lemma \ref{l.nondegenerate}.
\end{proof}

The relevance of this property to us is the following integrability result from \cite[Proposition 9]{HM98}.
\begin{prop} \label{p.integrable}
Let $M$ be a uniruled projective manifold and $\sK$ a family of minimal rational curves on $M$. Let $\sC_x \subset \BP T_xM$ be the VMRT at a general point and  let $W_x \subset \BP T_xM$ be its linear span.  Assume that the variety of tangential lines of $\sC_x \subset W_x$ is linearly non-degenerate in $\BP \wedge^2 W_x$.  Then the distribution $\mathcal{W}$ defined by $W_x$ is integrable on an open subset of $M$.
\end{prop}

%The following is from \cite[Proposition 11]{HM98}.
%\begin{prop} \label{p.Walgintegrable}
%Let $M$ be a uniruled quasi-projective manifold and $\sK$ a family of minimal rational curves on $M$. Assume that the distribution $\mathcal{W}$ defined by the linear spans of the VMRTs is integrable.
%Then there exists a closed subvariety $S(\mathcal{W}) \subsetneq M$   such that every leaf $F$ of  $\mathcal{W}$  is closed in $M \setminus S(\mathcal{W})$ and its topological closure $\bar{F}$ is a complex-analytic subvariety in $M$.   This induces a  meromorphic fibration  $M \dasharrow B$ whose generic fibers are closures of the leaves.
%\end{prop}

\section{Invariance of varieties of minimal rational tangents}

\subsection{Rigidity properties of Fano deformation}

Let $X$ be a normal projective variety. Consider the $\mbR$-space $N^1(X): = ({\rm Pic}(X)/\equiv)\otimes \mbR$, where $\equiv$ is the numerical equivalence.
The {\em nef cone} ${\rm Nef}(X) \subset N^1(X)$ is the closure of the cone spanned by ample classes.  The closure of the cone spanned by effective classes in $N^1(X)$ is the  {\em pseudo-effective cone} ${\rm PEff}(X) \subset N^1(X)$.  The {\em movable cone} ${\rm Mov}(X) \subset N^1(X)$ is the closure of  the cone spanned by  classes of divisors moving in a linear system with no fixed components.  The inclusion relation among these cones is given by
$$
{\rm Nef}(X)  \subset {\rm Mov}(X) \subset {\rm PEff}(X) \subset N^1(X).
$$

The {\em Mori cone} $\overline{{\rm NE}(X)} \subset N_1(X)$ is the dual of the nef cone ${\rm Nef}(X) \subset N^1(X)$. By Kleiman's criterion, the Mori cone $\overline{{\rm NE}(X)}$ is the closure of the cone spanned by classes of effective curves.

When $X$ is a smooth Fano variety, all these cones are rational polyhedral cones.
\begin{example} \label{e.cones}
Let $X$ be the wonderful compactification of a simple algebraic group $G$. The Picard group ${\rm Pic}(X)$ is identified with the weight lattice $\Lambda$ of $G$.
By \cite[Section 2]{Br07}, the extremal rays of the pseudo-effective cone ${\rm PEff}(X)$ are generated by simple roots, while those of the nef cone ${\rm Nef}(X)$ are generated by  fundamental weights. In particular, the nef cone ${\rm Nef}(X)$ is then identified with the positive Weyl chamber.
\end{example}

%{\color{blue}Remark: The two words "fundamental weights" and "simple roots" in Example \ref{e.cones} are switched. Besides, all the conclusions in Example \ref{e.cones} could already be deduced from Subsection 1.2, and thus we may regard \cite{Br07} as a supplementary reference.
%}

It turns out that these cones behave well under Fano deformation, which follows from a series of works of Siu, Wi\'sniewski and de Fernex-Hacon.  We refer to the survey  \cite{dFH} and the references therein for more details.

%{\color{blue}Remark: This is a general question including results appearing somewhere else. For the results such as Theorem \ref{t.InvCones}, do we need to cite the original reference or just a survey such as \cite{dFH}?}

\begin{thm}[\cite{dFH}] \label{t.InvCones}
Let $\pi: \mcX\rightarrow\Delta$ be a regular family of Fano manifolds.  Then
\begin{itemize}
\item[(i)] $N^1(\mcX/\Delta) \simeq N^1(\mcX_t) \simeq H^2(\mcX_t, \mbR)$  and ${\rm Pic}(\mcX/\Delta) \simeq {\rm Pic}(\mcX_t)$ for any $t \in \Delta$.
\item[(ii)] The nef cones, movable cones,  pseudo-effective cones are constant in the family under the isomorphism in (i).
\item[(iii)]For any $\mcL\in\Pic(\mcX/\Delta)$ and any $t\in\Delta$, the homomorphism $H^0(\mcX, \mcL)\rightarrow H^0(\mcX_t, \mcL|_{\mcX_t})$ is surjective, and thus $h^0(\mcX_t, \mcL|_{\mcX_t})=h^0(\mcX_0, \mcL|_{\mcX_0})$.
\end{itemize}
\end{thm}

\subsection{Invariance of VMRT for general cases}\label{s.vmrt}

We are now going to study the behavior of VMRTs in a regular family.
Let $\pi: \mcX\rightarrow\Delta\ni 0$ be a regular family of  Fano varieties. 
 Let $\FRC(\mathcal{X})$ be the space of free rational curves on $\mathcal{X}$, see \cite[Section 2.3]{HM98} for the construction and properties of $\FRC(\mathcal{X})$. % It is known that the connected components of  $\FRC(\mathcal{X})$ are complex manifolds.  
 Let $C$ be a free rational curve on $\mcX$.  As $C$ is projective, it must be contained in a fiber $\mathcal{X}_t:=\pi^{-1}(t)$ for some $t\in \Delta$, hence it is free on $\mathcal{X}_t$. Conversely, given a free rational curve $f: \BP^1 \to  \mathcal{X}_t$ for some $t \in \Delta$, we have an exact sequence 
$$
0 \to f^*T_{\mcX_t} \to f^* T_\mcX \to f^* N_{\mcX_t | \mcX} \simeq f^* \mathcal{O}_{\mcX_t} \to 0.
$$
It follows that $f^*T_\mcX$ is again free. This implies that a free rational curve on $\mcX$ is a free rational curve in some fiber of $\mathcal{X}/\Delta$, namely we have
 $\FRC(\mathcal{X}/\Delta)=\FRC(\mathcal{X})$. Moreover, there is a natural map $\FRC(\mathcal{X}/\Delta)\rightarrow\Delta$ such that the fiber at a point $t\in\Delta$ is nothing else but $\FRC(\mathcal{X}_t)$. For each $x\in\mathcal{X}$, let $\FRC(x, \mathcal{X})$ be the space of free rational curves on $\mathcal{X}$ passing through the point $x$. Then $\FRC(x, \mathcal{X})$ is just $\FRC(x, \mathcal{X}_t)$, where $t=\pi(x)\in\Delta$.
Let $\mathcal{K}$ be an irreducible component of $\FRC(\mathcal{X}/\Delta)$. The fiber $\mathcal{K}^t$ of the family $\mathcal{K}/\Delta$ at any point $t\in\Delta$ is either empty or a union of finitely many components of $\FRC(\mathcal{X}_t)$. As in the absolute case, there is a universal family $\sU$ with projections $\upsilon : \sU \to \sK$ and $\mu : \sU \to \mcX$. Furthermore, $\upsilon$ is a $\bP^1$-bundle, and $\mu$ is a smooth dominant map.

We are very grateful for an anonymous referee for providing the following result which we have neglected in a previous version.
\begin{lem} \label{l.Mok}
Let $\pi: \mcX \to \Delta$ be a regular family. Choose a point $t_0 \in \Delta^*:= \Delta \setminus \{0\}$. Then the deformation of any minimal rational curve on $\mcX_{t_0}$ covers $\mcX$.
\end{lem}
\begin{proof}
Take  a general minimal rational curve $C$ on $\mcX_{t_0}$ and let $\mathcal{Q}$ be an irreducible component of the Barlet space of $\mcX$ containing the reduced irreducible cycle $[C]$ as a point. 
As $\pi: \mcX \to \Delta$ is projective, there exists a K\"ahler metric on $\mcX$ with K\"ahler class $\omega$. Then for any cycle $[D] \in \mathcal{Q}$, we have 
\begin{equation}\label{e.volume}
{\rm volume}(D, \omega) = {\rm volume}(C, \omega).
\end{equation}
Let $\rho: \mathcal{V} \to \mathcal{Q}, \mu: \mathcal{V} \to \mcX$ be the universal family. Then for any compact subset $\Gamma \subset \mcX$ and for any $x\in \Gamma$, $\mu^{-1}(x)$ consists of cycles in $\mathcal{Q}$ passing through $x$. From \eqref{e.volume}, it follows that $\mu^{-1}(\Gamma) \subset \mathcal{V}$ is compact. Recall that Bishop's Theorem affirms that on a Euclidean ball, any sequence of subvarieties with bounded volume contains a convergent subsequence, from which we deduce that $\mu: \mathcal{V} \to \mcX$ is proper.  As $C$ is general, it is free in $\mcX$, then $\mu(\mathcal{V})$ contains an open neighborhood of $C$ in $\mcX$. By the proper mapping theorem applied to $\mu$ we have $\mu(\mathcal{V}) =\mcX$.  This concludes the proof.
\end{proof}

From now on, let $\pi: \mcX\rightarrow\Delta\ni 0$ be a regular family of  Fano varieties such that $\mcX_t  \simeq  X$ for all $t \in \Delta^*$, where $X$ is the wonderful compactification of a simple algebraic group $G$ of adjoint type. Let $\sK$ be the irreducible component of $\FRC(\mcX/\Delta)$ which contains the unique minimal rational component of $\mcX_t$ for $t \neq 0$. %Denote by $\upsilon: \mcU \to \sK$ the universal family and by $\mu: \mcU \to \mcX$ the evaluation map. 
Let $\overline{\sK^t}$ be the Zariski closure in $\RatCurves^n(\mathcal{X}_t)$ of the fiber $\sK^t$ of $\sK/\Delta$ at the point $t\in\Delta$, then for $t\neq 0$, $\overline{\sK^t}$ is a minimal rational component of $\mcX_t$  and it contains $\sK^t$ as an open subset.

The following Lemma is well-known (see for example  \cite[Propositions 4 and 5]{HM98}), but for the reader's convenience, we recall the proof here.

\begin{lem}\label{l.Hom-scheme}
The morphism $\mu$ is smooth with connected fibers. For any $x_t$ in the open $G \times G$-orbit of $\mcX_t$ with $t\neq 0$, $\mcU_{x_t}$ is projective and the tangent map $\tau_{x_t}: \mcU_{x_t}\to \mcC_{x_t}$ is an isomorphism.
\end{lem}

\begin{proof}
Let $\sH$ be the set of the points $[f]$ in the relative Hom-scheme $\Hom^n_{\text{bir}}(\mbP^1, \mcX/\Delta)$ representing an element in $\mcK$, i.e. $f(\mathbb{P}^1)$ is free on $\mathcal{X}$. Then $\sH$ is smooth and the evaluation map $\bP^1\times\sH \to \mcX$ is also smooth. 
Furthermore, $\mcU=(\bP^1\times\sH)/\Aut(\bP^1)$ and $\mu: \mcU \to \mcX$ is smooth by \cite[Proposition 4]{HM98}. More precisely, given a point $z\in \mcU$ represented by $(p, [f])\in \bP^1\times\sH$, the fact that $H^1(\bP^1, N_{f(\mbP^1)/\mcX}\otimes\mfm_p)=0$ implies that $\mu$ is smooth at $z$, where $N_{f(\mbP^1)/\mcX}=f^*T\mcX/T\bP^1$ is the normal bundle of $f(\bP^1)$ in $\mcX$, and $\mfm_p=\mcO_{\mbP^1}(-1)$ is the maximal ideal of $p$ in $\mbP^1$.

Take $x_t$ in the open  $G \times G$-orbit of $\mcX_t$ with $t\neq 0$. Since any element in $\overline{\sK^t}_{x_t}$ represents a free rational curve, $\sK_{x_t}=\sK^t_{x_t}=(\overline{\sK^t})_{x_t}$ is projective. Then $\mcU_{x_t}$ is projective and it is a smooth projective variety. Note that the tangent morphism $\tau_{x_t}: \mcU_{x_t} \to \mcC_{x_t}$ is defined as follows: if $z\in\mcU_{x_t}$ is represented by $(p, [f])\in\bP^1\times\sH$, then $\tau_{x_t}(z)=[(df)_p(T_p\mbP^1)]\in\mcC_{x_t}$. Now $\mcU_{x_t}$ is the normalization of the VMRT $\mcC_{x_t}$, and the latter is smooth and connected by Theorem \ref{t.VMRTWonderful}. Then $\mcU_{x_t}\cong\mcC_{x_t}$ and it is connected. Since general fibers $\mcU_{x_t}$ of $\mu$ are connected, so are all fibers.
\end{proof}

Now we are going to show that the limit $\sK^0$ contains a minimal rational component of $\mcX_0$.
\begin{prop}\label{p.MRC}
\begin{itemize}
\item[(i)] The central fiber $\sK^0: = \{[C] \in \sK| C \subset \mcX_0\}$ of $\sK/\Delta$ is not empty.
\item[(ii)] There is a unique irreducible component of $\sK^0$ dominating $\mcX_0$, denoted by $\sK^{0, g}$. 
\item[(iii)]The family $\sK^{0, g}$ is a minimal rational component of $\mcX_0$.
\item[(iv)] Take a general point $x\in\mcX_0$. Then $\sK^{0, g}_x=\sK^0_x=(\overline{\sK^0})_x$ is projective, and $\mcU_x$ is a connected smooth projective variety.
\end{itemize}
\end{prop}

\begin{proof}
By Lemma \ref{l.Mok}, the deformation of a general minimal rational curve $C_t$ on $\mcX_t$ covers $\mcX$.
Take a family $\{C_t\}_{t\in \Delta^*}$ of such curves deforming to $C_0 \subset \mcX_0$ such that $C_0$ passes through a general point $x$ of $\mcX_0$. We claim that $C_0$ is a minimal rational curve on $\mcX_0$.

Take a relative ample line bundle $\mcL \in {\rm Pic}(\mcX/\Delta)$. For all $t$, the line bundle $\mcL_t:=\mcL|_{\mcX_t}$ is ample on $\mcX_t$ and $\mcL_t \cdot C_t = \mcL_0 \cdot C_0$.
Let $C'_0 \subset C_0$ be an irreducible reduced component of $C_0$ through $x$, then $C'_0$ is a free rational curve on $\mcX_0$ as $x \in \mcX_0$ is general.
It follows that $C'_0$ is also free on $\mcX$, hence it deforms to $C'_t$ on $\mcX_t$ such that
$$
\mcL_t \cdot C'_t = \mcL_0 \cdot C'_0 \leq \mcL_0 \cdot C_0 = \mcL_t \cdot C_t.
$$
By Corollary \ref{c.minimal}, $\mcL_t \cdot C'_t \geq \mcL_t \cdot C_t$, implying that $C_0 = C'_0$ is irreducible, reduced and free. In particular, (i) holds.

In a similar way, one shows that $C_0$ cannot break into a cycle  passing through $x$ with several components  having lower intersection numbers with $\mcL_0$, hence $C_0$ is a minimal rational curve. This shows that $\sK^0_x$ is proper.
It follows that $\sK^0_x=(\overline{\sK^0})_x$ is projective, and $\mcU_x$ is projective. By Lemma \ref{l.Hom-scheme}, $\mcU_x$ is a smooth connected projective variety. In particular, $\mcU_x$ and $\sK_x=\sK^0_x$ are irreducible, implying the existence and uniqueness of the irreducible component $\sK^{0, g}$ of $\sK^0$ dominating $\mcX_0$. Furthermore, this unique component satisfies that $\sK^{0, g}_x=\sK^0_x$ and thus is projective.
Since $C_0$ is free on $\mcX$, the corresponding point $[C_0]$ is a smooth point of $\FRC(\mcX/\Delta)$. Thus $\sK^{0, g}$ is an irreducible component of $\FRC(\mcX_0$), and it is a minimal rational component of $\mcX_0$.
\end{proof}

In the following when mentioning the VMRT of $\mcX_0$, we mean the one associated with the irreducible family of minimal rational curves $\sK^{0, g}$.

The first key step to the proof of Theorem \ref{t.main} is the following:
\begin{thm} \label{t.invVMRT}
Let $G$ be a simple linear  algebraic group of the adjoint type and let $X$ be its wonderful compactification. Let $\pi: \mcX \to \Delta$ be a regular family of Fano manifolds such that $\mcX_t \simeq X$ for all $t \neq 0$.  Let $x \in \mcX_0$ be a general point.  Then the VMRT of $\mcX_0$ at $x$ is projectively equivalent to  the VMRT of $X$ at a general point.
\end{thm}

The case of type $A$ will be proved in the next subsection (Proposition \ref{p.invVMRTtypeA}), by using the explicit construction of the wonderful compactification $\bar{A}_n$ by successive blowups, while we defer the case of type $B_3$ to the last section (Proposition \ref{p.invVMRTB3}), as its proof is more involved.  For all other types,  we will use the following result of deformation rigidity of rational homogeneous contact manifolds of Picard number one, which is a special case of Theorem \ref{t.HMRigidity}.

\begin{thm}[\cite{Hw97}]\label{t.FanoContact} \footnote{The case of $\fg=B_3$ was incorrectly included in the original statement in \cite{Hw97}.}
Let $\fg$ be a simple Lie algebra. Assume $\fg$ is neither  of type $A$ nor of type $B_3$. Let $\BP \msO \subset \BP \fg$ be the projectivized minimal nilpotent orbit,  which is a homogeneous contact manifold of Picard number one.
Let $\mcY\rightarrow\Delta$ be a smooth projective family such that $\mcY_t\cong \BP \msO$ for all $t\neq 0$. Then $\mcY_0\cong \BP \msO$.

%(i) If $\fg$ is not of type $B_3$, then $\mcY_0\cong \BP \msO$.

%(ii) If $\fg$ is of type $B_3$, then either $\mcY_0\cong S \cong B_3/P_2$ or $\mcY_0\cong\mathbb{X}$, where $\mathbb{X}$ is the horospherical variety $(G_2, \alpha_2, \alpha_1)$ described in \cite{Pas09}.
\end{thm}

\begin{cor}\label{c.Kx}
Assume $\fg$ is neither of type $A$ nor of type $B_3$. Let $x \in \mcX_0$ be a general point. Then the normalized Chow space $\mcU_x$ of $\mcX_0$ at $x$ is isomorphic to $\BP \msO$.
%(ii) If $G=B_3$, then either $\mcK_x\cong\msK_s$ or $\mcK_x\cong\bX^5$.
\end{cor}

\begin{proof}
Take a general section $\sigma: \Delta\rightarrow\mcX$ of $\pi$ passing through the general point $x$ in $\mcX_0$. Shrinking $\Delta$ if necessary, we can assume that $\sigma(t)\notin \partial \mcX_t$ for each $t\neq 0$.
 The normalized Chow spaces $\mcU_{\sigma(t)}$ along this section give a family of smooth projective varieties such that $\mcU_{\sigma(t)} \simeq \BP \msO$ for $t \neq 0$ by Theorem \ref{t.VMRTWonderful}. Now the claim follows from Theorem \ref{t.FanoContact}.
\end{proof}

\begin{prop}\label{p.invVMRT}
Assume $\fg$ is neither of type $A$ nor of type $B_3$. Let $x \in \mcX_0$ be a general point.  Then the VMRT of $\mcX_0$ at $x$ is projectively equivalent to $\BP \msO \subset \BP \fg$.
\end{prop}

\begin{proof}
Take a general section $\sigma: \Delta\rightarrow\mcX$ of $\pi$ passing through the general point $x$ in $\mcX_0$ such that $\sigma(t)\notin \partial \mcX_t$ for each $t\neq 0$. Then for $t \neq 0$, the following map
$$
\mcU_{\sigma(t)} \simeq \mathcal{C}_{\sigma(t)} \hookrightarrow \BP T_{\sigma(t)} \mcX_t
$$
is induced by the linear system of $\0_{\BP \msO}(1):=\0_{\BP \fg}(1)|_{\BP \msO}$ (or $\0_{\BP \msO}(2)$ in the case of type $C$) under the identification $\mathcal{C}_{\sigma(t)} \simeq \BP \msO$.   This implies that the map
$\mcU_x \to \mathcal{C}_x \hookrightarrow \BP T_{x} \mcX_0$ is given by a sublinear system of $\0_{\BP \msO}(1)$ (or $\0_{\BP \msO}(2)$ in the case of type $C$).  Let $W_x \subset \BP T_{x} \mcX_0$ be the linear span of $\mathcal{C}_x$. Then $\mathcal{C}_x \subset W_x$ is a linear projection of $\BP \msO \subset \BP \fg$.  If $W_x=\BP T_{x} \mcX_0$,  then the sublinear system is complete and the claim follows. It remains to show that $W_x=\BP T_{x} \mcX_0$.

As the variety of tangential lines of  $\BP \msO \subset \BP \fg$ is linearly non-degenerate by Proposition \ref{p.nondegenerate}, so is the variety of tangential lines of $\mathcal{C}_x \subset W_x$.  By Proposition \ref{p.integrable}, the distribution  $\mathcal{W}$ defined by $W_x$ is integrable on an open subset of $\mcX_0$, which induces a dominant rational fibration $f: \mcX_0 \dasharrow B$. For $C\in\mcK^{0, g}$ general on $\mcX_0$, the projective curve $C$ is disjoint from the indeterminacy locus of $f$, and $f(C)$ is a single point.

If $W_x \neq \BP T_{x} \mcX_0$, then $\dim B >0$. Take an ample divisor $H_B$ on $B$, then $H:=f^* H_B$ is a movable divisor on $\mcX_0$ such that $H \cdot C =0$ for $C \in \sK_0$.  By Theorem \ref{t.InvCones},  this gives a rational fibration $f': X \dasharrow B'$ with $\dim B' >0$ such that the corresponding movable divisor $H'$ satisfies $H' \cdot C_{\boldsymbol\theta}=0$, i.e. $f'$ contracts general minimal rational curves on $X$.  This implies that for $x \in X$ general, all minimal rational curves through $x$ are contained in the fiber of $f'$ through $x$,  hence the VMRT $\mathcal{C}_x$ is contained in the tangent space of a fiber, which contradicts the fact that VMRT of $X$ is linearly non-degenerate.
% As the VMRT on $X$ is non-degenerate, any two general points are chain connected by minimal rational curves, hence they are contracted to the same point by $f'$. This implies $\dim B'=0$, a contradiction.
\end{proof}

%{\color{blue}Remark: In the proof of Proposition \ref{p.invVMRT}, the foliation induced integral meromorphic distribution $W$ is a priori NOT necessarily to be algebraic. Hence the words marked by underlines "\underline{$\cdots$}" may be not true. More precisely, $\eta: \mcX_0 \dasharrow B$ is just a dominant meromorphic fibration, and it is possible there does not exist ample divisor on $B$. On the other hand, $\eta$ factors through $\xi: \mcX_0\dasharrow\mcY_0$, where $\xi$ is the dominant rational fibration induced by the irreducible component of $\Chow(\mcX_0)$ corresponding to $\mcK^0\subset\RatCurves^\n(\mcX_0)$. Afterwards, replace $\eta: \mcX_0 \dasharrow B$ by $\xi: \mcX_0\dasharrow\mcY_0$ in the last paragraph of the proof and everything works now.
%}

\subsection{Invariance of VMRT for $A_n$}

Fix an integer $n\geq 2$.
  Let $V$ be a vector space of dimension $n+1$ and $G={\rm PGL}(V)$.  Let $Z = \mbP {\rm End}(V)$ and $Z_i \subset Z$ the locus of elements in ${\rm End}(V)$ of rank $\leq i$, where $i=1, \cdots, n$. Note that $Z_1 \simeq \mbP V \times \mbP V^*$ and the embedding $Z_1 \subset \mbP {\rm End}(V)$ is the Segre embedding.  The subvariety $Z_n$ is the determinant hypersurface, which is of degree $n+1$.

   The wonderful compactification $X$ of $G$ is given by the composition of  successive blowups $\phi: X \to Z$ along the strict transforms of $Z_i$, from the smallest $Z_1$ to the biggest $Z_{n-1}$.  We denote by $D_1, \cdots, D_{n-1}$ the exceptional divisors and by $D_n$ the strict transform of $Z_n$.  Let $\mcK^X$ be the family of minimal rational curves on $X$, whose members are strict transforms of lines in $Z$ intersecting $Z_1$.

 The following is straight-forward.
\begin{prop}\label{p.typeAComp}
(1) The secant variety of $Z_1$ is $Z_2$, which is different from $Z$ if $n \geq 2$.

(2) We have $H^2(X, \mbQ) = \oplus_{i=1}^n \mbQ [D_i]$.

(3) For any $C \in \mcK^X$, we have $D_1 \cdot C = D_n \cdot C =1$ and $D_j \cdot C=0$ for $j=2, \cdots, n-1$.

(4) The pseudo-effective cone ${\rm PEff}(X)=\sum_{i=1}^n\mbR^+ D_i$ is simplicial.

(5) The pull-back $\phi^* H $ of any hyperplane $H \subset Z$  is equal to $\sum_{i=1}^n (1-\frac{i}{n+1}) D_i$ in $H^2(X, \mbQ)$.
\end{prop}

We are grateful to a referee for providing the following lemma to improve our argument for Lemma \ref{l.K}.

\begin{lem}\label{l.trivialbundle}
Let $\phi: \mathcal{M}\rightarrow\Delta$ be a regular family of compact complex manifolds over the unit disk $\Delta$ and let $L$ be a line bundle over $\mathcal{M}$ such that, denoting $\mathcal{M}_t$ for $t\in\Delta$, the holomorphic line bundle $L|_{\mathcal{M}_t}$ is holomorphically trivial for $t\neq 0$, then $L$ is also holomorphically trivial on $\mathcal{M}_0$.
\end{lem}

\begin{proof}
Denote by $\mathcal{L}$ the sheaf of germs of holomorphic sections of $L$. By the hypothesis and the direct image theorem,  $\phi_*(\mathcal{L})$ is a locally free sheaf of rank $1$ on $\Delta$, hence $\phi_*(\mathcal{L})\cong\mathcal{O}_\Delta$. Let $s\in\Gamma(\Delta, \phi_*(\mathcal{L}))\cong\Gamma(\Delta, \mathcal{O}_\Delta)$ be a nowhere zero section on $\Delta$ and let $\xi$ be the corresponding section.
Then, for $t\neq 0$ we have $0\neq \xi|_{\mathcal{M}_t}=\Gamma(\mathcal{M}_t, L|_{\mathcal{M}_t})\cong\Gamma(\mathcal{M}_t, \mathcal{O}_{\mathcal{M}_t})$, hence $\xi|_{\mathcal{M}_t}$ is nowhere zero, while $\xi|_{\mathcal{M}_0}\in\Gamma(\mathcal{M}_0, L|_{\mathcal{M}_0})$ is not identically zero. As a consequence, the zero set of $\xi\in\Gamma(\mathcal{M}, \mathcal{L})$ is a nowhere dense subspace  of $\mathcal{M}_0$, thus of codimension $\geq 2$ in $\mathcal{M}$, hence empty, and we conclude that $L$ is holomorphically trivial on $\mathcal{M}$. In particular $L|_{\mathcal{M}_0}\cong\mathcal{O}_{\mathcal{M}_0}$, as desired.
\end{proof}

Let $\pi: \mcX\to \Delta$ be a Fano deformation of $X$ such that $\mcX_t \simeq X$ for all $t \neq 0$.  Let $\mcK$ and $\mathcal{U}$ be as in Section \ref{s.vmrt}. Then, for a general point $x \in \mcX_t$ ($t \neq 0$), we have $\mcU_x \xrightarrow{\simeq} \mcC_x \simeq \mbP V \times \mbP V^*$.  
We now extend this to the case of  $t=0$.
%The main purpose of this subsection is to show this holds also for $t=0$.

\begin{lem} \label{l.K}
For $x \in \mcX_0$ general, we have $\mcU_x \simeq \mbP V \times \mbP V^*$.
\end{lem}
\begin{proof}

Let $\mcX^\circ \subset \mcX$ be the open subset such that for $t \neq 0$, $\mcX^\circ_t$ is the open $G\times G$-orbit in $\mcX_t$, and for $x \in \mcX_0^\circ$, $\mcU_x$ is smooth and the map $\tau_x: \mcU_x \to \mcC_x$ is the normalization map.  Let $\mcU^\circ$ be the pre-image of $\mcX^\circ$ and let $\tau: \mcU^\circ \to \mcC^\circ \subset \mbP(T_{\mcX^\circ/\Delta})$ be the universal tangent map.   Let $\mcN$ be the pull-back to $\mcU^\circ$ of the relative line bundle $\0(1)$ on $\mbP(T_{\mcX^\circ/\Delta})$ via the map $\tau$.
As $\tau$ is a finite map (\cite{Kebekus}), $\mcN$ is relative ample for the map $\mcU^\circ \to \mcX^\circ$, i.e. for any $x \in \mcX^\circ$, the line bundle $\mcN|_{\mcU^\circ_x}$ is ample.

Consider the line bundle $\mcA = \mcN^{\otimes (n+1)} \otimes K_{\mcU^\circ/\mcX^\circ}$.   For $y \in \mcX^\circ_t$ with $t \neq 0$, $\tau_y: \mcU^\circ_y \to \mcC_y $ is an isomorphism and $\mcN|_{\mcU^\circ_y} \simeq \0(1,1)$ on $\mbP V \times \mbP V^*$, hence $\mcA|_{\mcU^\circ_y}$ is holomorphically trivial. By Lemma \ref{l.trivialbundle}, $\mcA|_{\mathcal{U}^\circ_x}$ is also holomorphically trivial, and thus  $K^{-1}_{\mcU^\circ_x} = \mcN^{\otimes (n+1)}|_{\mcU^\circ_x}$ is ample. Hence $\mcU_x$ is Fano.

Take a section $\sigma: \Delta \to \mcX^\circ$ through $x$, then $\mcU_{\sigma(t)}$ is a smooth Fano deformation of $\mbP V \times \mbP V^*$, hence $\mcU_x$ is itself isomorphic to $\mbP V \times \mbP V^*$ by \cite{Li}.
\end{proof}

 By Theorem \ref{t.InvCones}(i), there exists a line bundle $\mcL$ on $\mcX$ such that $\mcL_t$ is the pull-back line bundle $\phi^*(\0_Z(H))$ for $t\neq 0$, where $H$ is a hyperplane in $Z$.  Let $\mcZ = \mbP((\pi_* \mcL)^*) \simeq Z \times \Delta$ and let $\Phi: \mcX \dasharrow \mcZ$ be the rational map induced by the linear system $|\mcL|$. Then for $t \neq 0$, the map $\Phi_t$ coincides with $\phi$ and the indeterminacy locus $F$ of $\Phi$ is properly contained in $\mcX_0$. Note that $\Phi_0$ is induced from
 the linear system $|\mcL_0|:=|\mcL|_{\mcX_0}|$, hence by Theorem \ref{t.InvCones}(iii) the indeterminacy locus of $\Phi_0$ is $F$.

\begin{lem} \label{l.Phi0birational}
(i) For any divisor $D \in |\mcL_0|$ through a general point $x\in \mcX_0$, $D$ is smooth at $x$.

(ii) The rational map $\Phi_0: \mcX_0 \dasharrow \mcZ_0$ is birational.
\end{lem}
\begin{proof}

(i) Fix a point $y \in \mcX_0 \setminus D$. Take two sections $\sigma_i: \Delta \to \mcX$ such that $\sigma_1(0)=x$ and $\sigma_2(0)=y$. Let ${\mathbf l}_t$ be the line in $\mcZ_t$ joining the two points $\Phi(\sigma_1(t))$ and $\Phi(\sigma_2(t))$ and let $C_t$ be the strict transform of ${\mathbf l}_t$ for $t \neq 0$. Denote by $C_0$ the limit cycle of $C_t$ on $\mcX_0$, which joins $x$ to $y$.
As $\mcL \cdot C_t=1$ for $t \neq 0$, we have  $\mcL \cdot C_0 = 1$.  Let $C_0 = C_0^1 \cup \cdots \cup C_0^k$ be the irreducible components, with $x \in C_0^1$.
There exists a component, say $C_0^i$, such that $C_0^i \nsubseteq D$ and $C_0^i \cap D \neq \emptyset$, which implies that $D \cdot C_0^i \geq 1$.  As $D$ is nef by Theorem \ref{t.InvCones}(ii) and as $D\cdot C_0=1$, we have $D\cdot C_0^i=1$ and $D \cdot C_0^j = 0$ for all $j\neq i$. As $C_0^1$ passes through a general point $x$ of $\mcX_0$, it is free on $\mcX_0$ and deforms to a free curve $C_t^1$ on $\mcX_t$ with $t\neq 0$. By Corollary \ref{c.minimal}, we have $D\cdot C_0^1=\mcL_t\cdot C_t^1\geq\mcL_t\cdot C_{\boldsymbol \theta}=1$. Hence $C_0^1 = C_0^i$ and $D \cdot C_0^1=1$, implying $D$ is smooth at $x$.

(ii) Let us first show that $\Phi_0$ is dominant. If not, then the closure $Z_0$ of $\Phi(\mcX_0)$ is properly contained in $\mcZ_0$. The embedded tangent space $T_zZ_0$ at $z=\Phi_0(x) \in Z_0$ is contained in some hyperplane $H$ of $\mcZ_0$. By Theorem \ref{t.InvCones}(iii), $Z_0$ is linearly non-degenerate in $\mcZ_0$. Then $D_H:= \Phi_0^{-1}(H)$ is a divisor of $\mcX_0$ that is singular at $x$. This contradicts the claim (i), hence $\Phi_0$ is dominant.

 Let $\mcY$ be the graph closure of $\Phi$ and $\mcX \xleftarrow{\Psi_1} \mcY \xrightarrow{\Psi_2} \mcZ$ the two projections.  Let $Y_0 \subset \mcY_0$ be the strict transform of $\mcX_0$.  As $\Psi_2$ is projective and dominant, it is surjective. In particular, $\Psi_{2,0}: \mcY_0 \to \mcZ_0$ is surjective.  As $\Psi_2$ is birational and $\mcZ$ is smooth, $\Psi_2$ has connected fibers. This implies that $\Psi_{2,0}: \mcY_0 \to \mcZ_0$ is birational.  As a consequence, $\Phi_0: \mcX_0 \dasharrow \mcZ_0$ is birational.
\end{proof}

 Let $\mcB \subset \mcZ$ be the irreducible subvariety such that $\mcB_t = Z_1$ under the isomorphism $\mcZ_t \simeq Z$ for all $t \neq 0$.
 Similarly, we denote by $\mcD_i \subset \mcX$ the irreducible divisor corresponding to $D_i$ on $X$. Then similar results as those in Proposition \ref{p.typeAComp} hold for $\mcD_i$ and also for the specialization $\mcD_{i,0}:=\mcD_i|_{\mcX_0}$ to $t=0$.

\begin{lem}
We have ${\rm PEff}(\mcX_0)=\sum_{i=1}^n\mbR^+\mcD_{i, 0}$. The divisor $\mcD_{1,0}$ is irreducible, and $C\cdot\mcD_{1,0}=1$ for $[C] \in\mcK^0$ on $\mcX_0$.
\end{lem}

\begin{proof}
By Proposition \ref{p.typeAComp}(3), we have $C\cdot\mcD_{1,0}=1$ for $C\in\mcK^0$ on $\mcX_0$. Since pseudo-effective cones are invariant under Fano deformation by Theorem \ref{t.InvCones}, we know from Proposition \ref{p.typeAComp}(4) that the pseudo-effective cone ${\rm PEff}(\mcX_0)=\sum_{i=1}^n\mbR^+D_{i, 0}$ with $\mbR^+\mcD_{1, 0}$ being an extremal ray, and each irreducible component $E$ of $\mcD_{1, 0}$ satisfies that $E\in\mbR^+\mcD_{1, 0}$, implying that $C\cdot E\geq 1$. Hence $\mcD_{1,0}=E$ and it is irreducible.
\end{proof}

 \begin{lem} \label{l.ANondegenerate}
 The subvariety $\mcB_0 \subset \mcZ_0$ is linearly non-degenerate.
 \end{lem}
 \begin{proof}
 Assume that $\mcB_0$ is contained in a hyperplane $H$ of $\mcZ_0$. Note that $\Phi_0$ is the limit of $\Phi_t$, which is defined outside $F$. The map $\Phi_0$ sends the irreducible divisor $\mcD_{1, 0}$ to $\mcB_0$, as so is $\Phi_t$. Then the pull-back $\Phi_0^* H$ can be written as $k\mcD_{1,0} + E$, where $k\geq 1$ and $E$ is an effective divisor. Take a general member $[C] \in\mcK^{0, g}$ on $\mcX_0$.  As $C \cdot \Phi_0^* H =1$ and $C \cdot E \geq 0$, we have $k=1$ and $C \cdot E=0$. By Theorem \ref{t.InvCones}(ii) and Proposition \ref{p.typeAComp}(4), we have $E \in \oplus_{i=2}^{n-1} \mbQ[\mcD_{i,0}]$, implying that $\Phi_0^* H \in \oplus_{i=1}^{n-1} \mbQ[\mcD_{i, 0}]$. This contradicts the conclusion (5) of Proposition \ref{p.typeAComp}.
 \end{proof}

% \begin{lem}
% Assume that $\mcB_0 \subset \mcZ_0$ is linearly non-degenerate. Then for $x \in \mcX_0$ general, its VMRT $\mcC_x \subset \mbP T_x \mcX_0$ is projectively equivalent to $Z_1 \subset Z$.
%
% %For $x \in \mcX_0$ a general point, let $z = \Phi_0(x)\in \mcZ_0$.  Let $p_z: \mcZ_0 \dasharrow \mbP T_z \mcZ_0$ be the projection from $z$.  Then the map
% %$\Phi_0: \mcX_0 \dasharrow \mcZ_0$ induces a projective equivalence between $\mcC_x \subset \mbP T_x \mcX_0$ and $p_z(\mcB_0) \subset \mbP T_z \mcZ_0$.
% \end{lem}

  \begin{prop}\label{p.invVMRTtypeA}
Assume $\fg=\mathfrak{sl}(V)$ for a vector space $V$ of dimension $n+1\geq 3$. Let $x \in \mcX_0$ be a general point.  Then the VMRT of $\mcX_0$ at $x$ is projectively equivalent to the projection from a general point of the Segre embedding
$\mbP V \times \mbP V^* \subset \mbP(V \otimes V^*)$.
\end{prop}
 \begin{proof}
 Let $z = \Phi_0(x)\in \mcZ_0$ and  $p_z: \mcZ_0 \dasharrow \mbP T_z \mcZ_0$ the projection from $z$.
 Note that the secant variety of $\mcB_0$ is contained in the limit of secant varieties of $\mcB_t$, which is not the full space by Proposition \ref{p.typeAComp}(1).  We may assume that $z$ is not in the secant variety of $\mcB_0$.

 Take a curve $[C]\in \mcK_x$ and deform it to $C_t \subset \mcX_t$,  then the image $\Phi_0(C)$ is the limit of $\Phi_t(C_t)$, the latter being  lines in $\mcZ_t$ meeting $\mcB_t$. This implies that $\Phi_0(C)$ is again a line through $z$ in $\mcZ_0$ meeting $\mcB_0$.  As $z$ is not in the secant variety of $\mcB_0$, this line intersects $\mcB_0$ at a unique point, which gives a map $\mcU_x=\mcK^n_x\to \mcK_x \to \mcB_0$.
 Note that  the line $\Phi_0(C)$ is uniquely determined by its tangent direction at $z$, hence this induces a map $f_x: \mcC_x \to p_z(\mcB_0)$ which fits into the following commutative diagram:
 $$\begin{CD}
 \mcC_x @>f_x>> p_z(\mcB_0) \\
 @VVV    @VVV \\
 \mbP T_x\mcX_0 @>d_x\Phi_0>\simeq>  \mbP T_z\mcZ_0
 \end{CD}
 $$

 Note that $f_x$ is injective and  $\mcC_x$ has the same dimension as $\mcB_0$, hence the image $f_x(\mcC_x)$ is an irreducible component of $p_z(\mcB_0)$.  On the other hand, the degree of $\mcC_x$ is the same as that of $Z_1 \subset Z$, which is also the same as the degree of  $\mcB_t$ for all $t$. Hence $f_x(\mcC_x)$ has the same degree as that of $\mcB_0$, which implies that $f_x$ is also surjective, hence $f_x$ is bijective and $\mcB_0$ is irreducible.

 Take a general section $\sigma: \Delta \to \mcX$ through $x$ and
 consider the composition map $g_t:  \mcU_{\sigma(t)} \to \mcC_{\sigma(t)}\to \mcB_t \to \mcZ_t$. For $t \neq 0$, this is induced from the line bundle $\0(1,1)$ on $\mbP V \times \mbP V^*$. This implies that the central map $g_0$ is induced from a linear subsystem (say $L$) of $\0(1,1)$.
 As $\mcB_0$ is linearly non-degenerate in $\mcZ_0$ by Lemma \ref{l.ANondegenerate}, the linear system $L$ must be the complete linear system of $\0(1,1)$.  This implies that $g_0$ is the Segre embedding of $\mcU_x$, hence the maps $\tau_x$ and $f_x$ are both isomorphisms, concluding the proof.
 \end{proof}

\section{Rigidity under Fano deformation}

\subsection{Equivariant compactifications of vector groups}

A vector group of dimension $g$ is the additive group $\mathbb{G}_a^g$.   An equivariant compactification of $\mathbb{G}_a^g$ is
a smooth projective $\mathbb{G}_a^g$-variety $Y$  which admits an open $\mathbb{G}_a^g$-orbit $O$ isomorphic to $\mathbb{G}_a^g$.
The boundary $\partial Y = Y \setminus O$ is a union of irreducible reduced divisors $\cup_j E_j$.

Typical examples of equivariant compactifications of $\mathbb{G}_a^g$ are $\mathbb{P}^g$ and its blowup  along a smooth subvariety contained in a hyperplane.
It turns out that $\mathbb{P}^g$ can even have infinitely many different $\mathbb{G}_a^g$-equivariant compactification structures as soon as $g \geq 6$ (\cite{HT}).
As easily seen, there exists a unique  equivariant compactification of $\mathbb{G}_a$, which is given by $\mathbb{P}^1$.

\begin{example} \label{e.SEC_surface}
Consider the case  $g=2$. Let $Y$ be an equivariant compactification of $\mathbb{G}_a^2$ and let $f: Y \to Y_{min}$ be the natural birational morphism to the minimal model $Y_{min}$ of $Y$.

(i)  By \cite[Proposition 5.1]{HT}, $Y_{min}$ is an equivariant compactification of $\mathbb{G}_a^2$ and   the map $f$ is  $\mathbb{G}_a^2$-equivariant,  which is the composition of successive blowups along  $\mathbb{G}_a^2$-fixed points.

(ii) By \cite[Proposition 5.2]{HT}, $Y_{min}$ is isomorphic to $\mathbb{P}^2$ or a Hirzebruch surface $F_k:=\mbP_{\mbP^1}(\0 \oplus \0(k))$ with $k \geq 0$.

(iii) The plane $\mathbb{P}^2$ admits two different equivariant compactification structures and the  boundary for both is a line.

(iv) The product $\mathbb{P}^1 \times \mathbb{P}^1$ admits a unique equivariant compactification structure, induced from the one on  each $\mathbb{P}^1$ by \cite[Proposition 5.5]{HT}. The boundary divisor consists of two components isomorphic to $\mbP^1$.

(v) The Hirzebruch surface $F_k$ with $k >0$ admits two different equivariant compactification structures by \cite[Proposition 5.5]{HT} and the boundary divisor consists of a fiber and a section.
\end{example}

The following result is from \cite[Theorem 2.5, Theorem 2.7]{HT}.

\begin{prop}\label{p.SEC_HT}
Let $Y$ be an equivariant compactification of $\mathbb{G}_a^g$ with boundary divisors $\cup_{i=1}^l E_i$.   Then

(i)  the Picard group of $Y$ is freely generated by $E_1, \cdots, E_l$  and the pseudo-effective cone ${\rm PEff}(Y)$ is given by $\oplus_{i=1}^l \mathbb{R}_{\geq 0} E_i$;

(ii) the anti-canonical divisor of $Y$ is given by $-K_Y = \sum_{i=1}^l a_i E_i$ with $a_i \geq 2$ for all $i$.
\end{prop}

As a consequence, we get the following elementary result, which will be useful later.
\begin{lem}\label{p.vector compact surface}
Let $Y$ be a smooth projective surface which is an equivariant compactification of the vector group $\mathbb{G}_a^2$ and $O \subset Y$ the open $\mathbb{G}_a^2$-orbit. Let $A\subset\Aut(Y)$ be a finite subgroup such that $A\cdot O= O$. Let $S$ be the set of irreducible components of the boundary $\partial Y:=Y\setminus O$. If $S$ consists of two $A$-orbits, then there exists an irreducible component of $\partial Y$ which is stabilized by $A$.
\end{lem}

\begin{proof}
By Example \ref{e.SEC_surface}, $Y$ is a successive blowup of $Y_{min}$ along $\mathbb{G}_a^2$-fixed points, namely
$Y=Y_\ell \xrightarrow{f_\ell} Y_{\ell-1} \to \cdots \to Y_1 \xrightarrow{f_1} Y_{min}=Y_0$, where $f_i: Y_i \to Y_{i-1}$ is the blowup along a $\mathbb{G}_a^2$-fixed point $y_{i-1} \in Y_{i-1}$. Let us denote by $E_i \subset Y_i$ the exceptional fiber of $f_i$.  By abuse of notation, we also denote by
$E_i$ its strict transformation in $Y_j$ for $j \geq i$.
There are three possibilities for $Y_{min}$, namely $\mathbb{P}^2$, $\mathbb{P}^1 \times \mathbb{P}^1$ or a  Hirzebruch surface $F_k$ (with $k > 0$).

Write $H$ (resp. $H_1, H_2$) for irreducible components of the boundary divisor of $\mbP^2$ (resp. $\mathbb{P}^1 \times \mathbb{P}^1$ and $F_k$) which contain the fixed point set. Then the boundary set of
$Y$ consists of $H$ (resp. $H_1, H_2$) and the exceptional divisors $E_j$.
 As $-K_Y$ is $A$-invariant and the boundary divisor set $S$ has only two $A$-orbits,  there are at most  two coefficients in the expression of $-K_Y$  by Proposition \ref{p.SEC_HT}.

Assume $Y_{min} = \mbP^2$. The boundary of $Y_{min}$ is a line $H$, which is also the fixed point set of the $\mathbb{G}_a^2$-action.  By abuse of notation, we denote by $H$ its strict transform to any $Y_j$.
As the blowup point $y_0$ lies on  $H$, we have $-K_{Y_1} = 3 H + 2 E_1$.  The expression of $-K_Y$ implies that each blowup center $y_i$ is on $H$ but not in any other irreducible component of $E_i$, since otherwise there would be at least three different coefficients in the expression of $-K_Y$.
It follows that $-K_{Y} = 3H +2 \sum_i E_i$, and thus $H$ is the $A$-stable irreducible component.

Suppose now $Y_{min}=\mbP^1\times\mbP^1$. The boundary of $Y_{min}$ has two irreducible components $H_1, H_2$, which are  fibers of two $\mbP^1$-fibrations $Y_{min}\rightarrow\mbP^1$ respectively. Note that $-K_{Y_{min}}=2H_1 + 2 H_2$ and $y_0=H_1 \cap H_2$ is the unique $\mathbb{G}_a^2$-fixed point on $Y_{min}$.
Then $f_1$ must be the blowup at $y_0$, hence  $-K_{Y_1}=3 E_1+2H_1+ 2H_2$. The expression of $-K_Y$ implies that each blowup center $y_i$ is in $E_1$  but not in any other irreducible component of $\partial Y_i$. Hence $-K_Y=3E_1+2 (H_1 + H_2+ \sum_{j\geq 2} E_j)$, and $E_1$ is the $A$-stable irreducible component.

Now assume  $Y_{min}=F_k$ with $k\geq 1$. The boundary of $Y_{min}$ has two irreducible components $H_1, H_2$,  which are respectively a fiber and the minimal section of the $\mbP^1$-fibration $F_k\to \mbP^1$. Note that  $-K_{Y_{min}}=(k+2)H_1 +2 H_2$ and all $\mathbb{G}_a^2$-fixed points lie in $H_1$. The expression of $-K_Y$ implies that $k=1$ and each blowup center $y_i$ belongs to $H_1$  but not in any other irreducible component of $\partial Y_i$. This gives that $-K_Y=3 H_1+2 (\sum_{j} E_j + H_2)$, hence $H_1$ is the $A$-stable irreducible component.
\end{proof}

\begin{rmk}
(i) It can be seen from the proof of Lemma \ref{p.vector compact surface} that if the action of $A$ on $S$ is transitive, then $Y$ is isomorphic to $\mbP^2$ or $\mbP^1\times\mbP^1$.

(ii) For $G$ of type $G_2, F_4$ or $E_8$,  if $\bar{G}$ was  not rigid, then we will show  via case-by-case arguments that  there would exist a surface $Y$ with an irreducible component of $\partial Y$ stabilized by $A$, contradicting Lemma  \ref{p.vector compact surface}.
\end{rmk}

\begin{prop} \label{p.SECFlat}
Let $Y$ be a smooth uniruled projective variety of dimension $d$.

(i) If $Y$ is an equivariant compactification of $\mathbb{G}_a^d$, then the VMRT structure is locally flat.

(ii) Assume that (a) $Y$ is of Picard number one; (b) the VMRT structure is locally flat; and (c) the VMRT at a general point is smooth irreducible and linearly non-degenerate. Then $Y$ is an equivariant compactification of $\mathbb{G}_a^d$.

(iii) Take a general point $y\in Y$. Assume that (a) the VMRT $\mcC_y$ is smooth irreducible and linearly non-degenerate; (b) the linear space $\mathfrak{aut}(\hat{\mathcal{C}_y})^{(1)}=0$; and  (c)
$\dim \aut(Y) =  \dim Y + \dim \mathfrak{aut}(\hat{\mathcal{C}_y})$.  Then $Y$ is an equivariant compactification of $\mathbb{G}_a^d$.
\end{prop}
\begin{proof}
(i) Let $O \subset Y$ be the open $\mathbb{G}_a^d$-orbit, which is isomorphic to $\mathbb{C}^d$. As the $\mathbb{G}_a^d$-action preserves the VMRT structure $\mcC \subset \BP TY$, its action on $O$ trivializes $\mcC|_O$ as a trivial subbundle of
$\BP TY|_O \simeq O \times \BP^{d-1}$, hence the VMRT structure is locally flat.

(ii) This follows from \cite[Proposition 6.13]{FH}.

(iii) As ${\rm Aut}^0(Y)$ preserves the VMRT structure on $Y$, we have $\aut(Y) \subset \aut(\mcC, y)$.  The assumptions (b) and (c) imply that the equality in Proposition \ref{p.prolong} holds, hence $Y$ has locally flat VMRT structure and $\aut(Y) = \aut(\mcC, y)$.
By \cite[Proposition 5.14]{FH}, we have $\aut(Y)=\aut(\mcC, y) = \mathbb{C}^d \rtimes \mathfrak{aut}(\hat{\mathcal{C}_y})$. By considering the natural representation $\Aut(Y)\to {\rm GL}(H^0(Y, T_Y))={\rm GL}(\aut(Y))$, we know that the adjoint group of $\Aut^0(Y)$ is $\mathbb{G}_a^d\rtimes \Aut^0(\hat{\mathcal{C}_y})$. It gives rise to a subgroup $\mathbb{G}_a^d$ of $\Aut^0(Y)$, and the orbit $\mathbb{G}_a^d \cdot y$ is isomorphic to $\mathbb{C}^d$, which is open (and hence dense) in $Y$ by dimension reason.
\end{proof}

\subsection{Rigidity for general cases}
We start with the following result, whose proof is similar to \cite[Proposition 3.5]{Park}.
\begin{lem}\label{l.IsoFlat}
Assume $\fg$ is not of type $C$.  Let $\pi: \mcX\rightarrow\Delta\ni 0$ be a  regular family of Fano varieties such that $\mcX_t  \cong X$ for $t \neq 0$, where $X$ is the wonderful compactification of the simple algebraic group $G$ of adjoint type. Then either $\mcX_0 \cong X$ or  the VMRT-structure on $\mcX_0$ is locally flat.
\end{lem}
\begin{proof}
As $\mcX_t$ is Fano for all $t \in \Delta$, 
 we have  $H^q(\mcX_t, T_{\mcX_t})=0$ for all $q \geq 2$ and $t \in \Delta$ by the Akizuki-Nakano vanishing theorem (\cite{AN}). As a consequence, we have that $\chi(\mcX_t, T_{\mcX_t}) = h^0(\mcX_t, T_{\mcX_t}) - h^1(\mcX_t, T_{\mcX_t})$ for all $t \in \Delta$.   Since $\chi(\mcX_t, T_{\mcX_t})$ is invariant under deformations, and $H^1(X, T_X)=0$ by \cite[Proposition 4.2]{BB96}, we have
$$
h^0(\mcX_0, T_{\mcX_0}) - h^1(\mcX_0, T_{\mcX_0}) = h^0(X, T_X) - h^1(X, T_X) = h^0(X, T_X) = 2 \dim \fg.
$$

This gives that $h^1(\mcX_0, T_{\mcX_0}) = h^0(\mcX_0, T_{\mcX_0}) - 2 \dim \fg = \dim \mathfrak{aut}(\mcX_0) - 2 \dim \fg.$
Note that by Theorem \ref{t.invVMRT} (with the case of $B_3$ to be proved in Section 6), the VMRT $\mathcal{C}_x \subset \BP T_x \mcX_0$ is projectively equivalent to the VMRT of $X$ at a general point.  If $\fg$ is not of type $C$, then  $\mathfrak{aut}(\hat{\mathcal{C}_x}) \simeq \fg \oplus \mathbb{C}$ and $\mathfrak{aut}(\hat{\mathcal{C}_x})^{(1)}=0$ by \cite[Proof of Proposition 6.1]{BF15}.
Then we know from Proposition \ref{p.prolong} that $$\dim \mathfrak{aut}(\mcX_0) \leq \dim \fg + \dim \mathfrak{aut}(\hat{\mathcal{C}_x})+\dim\mathfrak{aut}(\hat{\mathcal{C}_x})^{(1)} = 2 \dim \fg +1.$$

We deduce that $h^1(\mcX_0, T_{\mcX_0}) \leq 1$.  If $h^1(\mcX_0, T_{\mcX_0})=0$, then $\mcX_0$ is locally rigid, hence it is isomorphic to $\mcX_t$ for $t$ small enough, namely $\mcX_0 \simeq X$.  If $h^1(\mcX_0, T_{\mcX_0})=1$, then the equality in Proposition \ref{p.prolong} holds, which implies that the VMRT-structure on $\mcX_0$ is locally flat.
\end{proof}
\begin{rmk} \label{r.Fanoness}
(i) When $G$ is of type $C_n$,  the VMRT of $\bar{G}$ is isomorphic to the second Veronese embedding of $\mbP^{2n-1}$, hence
$\mathfrak{aut}(\hat{\mathcal{C}_x}) \simeq \mathfrak{sl}_{2n} \oplus \mathbb{C}$ which is much bigger than $\mathfrak{sp}_{2n} \oplus \mathbb{C}$. This is why the type $C$ case must be treated separately.

(ii) In the precedent proof, it is crucial that $\mcX_t$ is  Fano for all $t$. This is another reason (apart from Theorem \ref{t.InvCones}) that we only deal with rigidity under Fano deformation in this paper. 
\end{rmk}

By applying Proposition \ref{p.SECFlat}, we have the following

\begin{cor}\label{c.equiv}
 In the setting of Lemma \ref{l.IsoFlat}, if $\mcX_0$ is not isomorphic to $X$, then $\mathfrak{aut}(\mcX_0) \simeq \mathbb{C}^g\rtimes (\fg \oplus \mathbb{C})$ and $\mcX_0$ is an equivariant compactification of $\mathbb{G}_a^g$, where $g=\dim \fg$.
\end{cor}

\begin{thm}\label{t.rigidity}
  Let $\pi: \mcX\rightarrow\Delta\ni 0$ be a  regular family of Fano varieties such that $\mcX_t  \cong X$ for $t \neq 0$, where $X$ is the wonderful compactification of a simple algebraic group $G$ of adjoint type.  Assume that $\fg$ is different from the following:  $C_n, G_2, F_4, E_8$.
  Then $\mcX_0 \cong X$.
\end{thm}
\begin{proof}
Let $n$ be the rank of $G$, which is assumed to be at least 2.
By Corollary \ref{c.equiv}, we may assume that $\mcX_0$ is an equivariant compactification of the vector group $\mathbb{G}_a^g$.
 Its boundary is given by $\partial \mcX_0 = \cup_{i=1}^n E_i$.  By Proposition \ref{p.SEC_HT},  the Picard group of $\mcX_0$ is freely generated by $E_1, \cdots, E_n$  and the pseudo-effective cone ${\rm PEff}(\mcX_0)$ is given by $\oplus_{i} \mathbb{R}_{\geq 0} E_i$.  By Example \ref{e.cones}, ${\rm PEff}(X)$ is generated by the boundary divisors $D_i$.  By Theorem \ref{t.InvCones}, the pseudo-effective cones are invariant, hence ${\rm PEff}(\mcX_0) = {\rm PEff}(X).$

Let $\mathcal{D}_i \in {\rm Pic}(\mcX/\Delta) \simeq {\rm Pic}(\mcX_t)$ be the divisor such that $\mathcal{D}_i|_{\mcX_t}= D_i$ for $t \neq 0$.  As $D_i$ and $E_i$ are in the extremal rays of the pseudo-effective cones, we have  $\mathcal{D}_i|_{\mcX_0} = a_i E_i$ for some $a_i \in \mathbb{N}$, up to re-ordering $E_i$. On the other hand,  when $\fg$ is not of type $C_n$,  $\alpha_i$ is primitive in the weight lattice.  The latter is identified with $\Pic(X)$ by Example \ref{e.cones},  hence the divisor $D_i = \sL_X(\alpha_i) \in {\rm Pic}(X)$ is primitive.  This forces $a_i=1$ for all $i$.

As ${\rm Pic}(\mcX_0) = \oplus_{i=1}^n \mathbb{Z} E_i$, through the isomorphisms ${\rm Pic}(\mathcal{X}_0)\simeq {\rm Pic}(\mathcal{X}/\Delta) \simeq {\rm Pic}(X)$ (cf. Theorem \ref{t.InvCones}), we have  ${\rm Pic}(\mathcal{X}/\Delta) = \oplus_{i=1}^n \mathbb{Z} \mathcal{D}_i$ and ${\rm Pic}(X) = \oplus_{i=1}^n \mathbb{Z} D_i$. Hence
${\rm Pic}(X)$ is the same as the root lattice of $G$.
This shows that the weight lattice equals to the root lattice for $G$, which is true if and only if $G$ is of type $G_2, F_4$ or $E_8$.
Hence if $G$ is not $G_2, F_4$ or $E_8$, then this leads to a contradiction, which concludes the proof.
\end{proof}

\subsection{Families of quasi-homogeneous submanifolds} \label{s.limit}

Fix a simple adjoint linear algebraic group $G$ of dimension $g$ and of rank $n$, which is assumed to be  different from $C_n$.
Let $\pi: \mcX \to \Delta$ be a  regular family of Fano manifolds such that $\mcX_t \simeq \bar{G}$ for all $t \neq 0$.  We assume that $\mcX_0$ is not isomorphic to $\bar{G}$.  By Corollary \ref{c.equiv}, $\mcX_0$ is an equivariant compactification of the vector group $\mathbb{G}_a^g$ and
$\mathfrak{aut}(\mcX_0) \simeq \mathbb{C}^g\rtimes (\fg \oplus \mathbb{C})$.

Fix a section $\sigma: \Delta \to \mcX$ such that $x_t:=\sigma(t)$ is a general point in $\mcX_t$ for all $t$.
For $t\in \Delta$, we define $j_t^0: \mathfrak{aut}(\mcX_t) = H^0(\mcX_t, T_{\mcX_t}) \to T_{x_t}\mcX_t$ by sending a global vector field on $\mcX_t$ to its value at the point $x_t$.  Define the jet map  $j_t^1: {\rm Ker}(j_t^0) \to \mathfrak{gl}(T_{x_t} \mcX_t)$ to be the isotropic representation.   The following is immediate from the construction and from our previous discussions.
\begin{lem} \label{l.jet}
(i) For all $t$, the map $\varphi_t: \mathfrak{aut}(\mcX_t) \to \mathfrak{aut}(\mcC_t, x_t)$ is an isomorphism of Lie algebras, where $\mcC_t$ is the VMRT structure on $\mcX_t$.

(ii) For all $t$, the map $j_t^0$ is surjective and the map $j_t^1$ is injective.

(iii)  For $t \neq 0$, we have ${\rm Aut}^0(\mcX_t) \simeq G \times G$, $\mathfrak{aut}(\mcX_t) \simeq \fg \oplus \fg$  and ${\rm Ker}(j_t^0) \simeq {\rm diag}(\fg)$.  The map $j_t^1$ is the adjoint representation of $\fg$.

(iv) For $t=0$, we have ${\rm Aut}^0(\mcX_0) \simeq \mathbb{G}_a^g \rtimes (G \times \mathbb{C}^*)$, $\mathfrak{aut}(\mcX_0) \simeq \mathbb{C}^g \oplus (\fg \oplus \mathbb{C})$  and ${\rm Ker}(j_t^0) \simeq \fg \oplus \mathbb{C}$. The map $j_t^1$ is the adjoint representation on $\fg$ (by identifying $\C^g$ with $\fg$ as modules) and the dilation representation on $\mathbb{C}$.
\end{lem}

Let $\fg_{ad}$, $\fg_l$ and $\fg_r$ be the Lie algebras of $\diag(G)$, $G\times e$ and $e\times G$ respectively. In particular, the Lie algebra of $G\times G$ can be written as $\fg_l\oplus\fg_r$ as a $\fg_{ad}$-module.
Let $\mcV =\pi_* T_{\mcX/\Delta}$, which is a vector bundle over $\Delta$ such that $\mcV_t \simeq \mathfrak{aut}(\mcX_t)$ for $t\neq 0$.  Let $\mcW \subset \mcV$ be the subbundle such that $\mcW_t \simeq {\rm Ker}(j_t^0) \simeq \fg_{ad}$ for $t\neq 0$.
\begin{lem}\label{l.isotropic vector fields}
(i) The vector bundle $\mcW$ is isomorphic to the trivial bundle with fiber $\fg_{ad}$.

(ii) For all $t\in \Delta$, the fiber $\mcV_t$ is a $\fg_{ad}$-module, which is isomorphic to $\fg_l \oplus \fg_r$.
\end{lem}
\begin{proof}
(i) Fix a holomorphic family of linear isomorphisms $\xi_t: T_{x_t} \mcX_t \to \fg$ sending $\hat{\mcC}_{x_t}$ to the variety $\hat{\mcC}_e$ constructed in Theorem \ref{t.VMRTWonderful}. By Lemma \ref{l.jet}, the composition ${\rm Ker}(j_t^0) \xrightarrow{j_t^1} \mathfrak{gl}(T_{x_t} \mcX_t) \xrightarrow{\zeta_t=(\xi_t)_*} \mathfrak{gl}(\fg)$ is injective for all $t\in \Delta$.  For $t \neq 0$, we have
$\zeta_t \circ j_t^1(\mcW_t) = \fg_{ad}$, hence by continuity, we have $\zeta_0 \circ j_0^1(\mcW_0) \subseteq \fg_{ad}$. As both sides have the same dimension, we have the equality.  Hence the map $\zeta_t \circ j_t^1$ gives the trivialization of the bundle $\mcW$.

(ii) For all $t\in \Delta$, $\mcW_t \simeq \fg_{ad}$ is a Lie subalgebra of $\mcV_t$.  For $t\neq 0$, $\mcV_t \simeq \fg \oplus \fg$ as $\fg_{ad}$-modules.
Note that $\mathfrak{aut}(\mcX_0)$ is completely reducible as a $\fg_{ad}$-module, with irreducible factors $\mathbb{C}^g, \fg_{ad}$ and $\mathbb{C}$.  By dimension reason, we have $\mcV_0 \simeq \mathbb{C}^g \rtimes \fg_{ad}$, which is isomorphic to $\fg_l \oplus \fg_r$ as $\fg_{ad}$-modules.
\end{proof}

Let $\tilde{G}$ be the simply-connected cover of $G$. By  Lemma \ref{l.jet}, ${\rm Ker}(j_t^0)$ contains $\fg$ for all $t$, which shows that
 $\tilde{G}$ acts on $\mcX$ fixing the section $\sigma(\Delta)$. As the center of $\tilde{G}$ acts trivially on $\mcX_t$ for $t\neq 0$, so is on $\mcX_0$ by continuity. Then we have an action of $G$ on $\mcX$ fixing the section $\sigma(\Delta)$. Let $T \subset G$ be a maximal torus and let $\mathfrak{h} \subset \fg$ be  its Lie algebra. Then $T$ acts on $\mcX/\Delta$.  Let $\mcY$ be the connected component of the fixed locus $\mcX^T$ which contains $\sigma(\Delta)$. Then the map $\mcY \to \Delta$ is a regular family of  projective varieties by Fogarty (\cite[Theorem 5.4]{F}, see also \cite[Theorem 13.1]{M}).

 For a Lie algebra $\fg$, we denote by $\mathbb{G}_a(\fg)$ the vector group associated to $\fg$ (viewed as a vector space). In particular, for a Lie subalgebra
 $\fh \subset \fg$, we have $\mathbb{G}_a(\fh) \subset \mathbb{G}_a(\fg)$.  The adjoint action of $T$ on $\mcY$ is trivial by construction.

\begin{prop} \label{p.limit}
(1) For $t \neq 0$, the subgroup $T\times T$ of $\Aut^0(\mcX_t)=G\times G$ stabilizes $\mcY_t$. The left $T$-action gives rise to an equivariant open embedding $T\simeq T\cdot x_t\subset\mcY_t$. Furthermore, $\mcY_t = \bar{T}$ is the toric variety with fan consisting of Weyl chambers of $G$ and their faces.

(2) For $t = 0$, the subgroup $\mbG_a(\mfh)\rtimes T$ of $\Aut^0(\mcX_0)=\mbG_a^g\rtimes (G\times\C^*)$ stabilizes $\mcY_0$. The $\mbG_a(\mfh)$-action gives rise to an equivariant open embedding $\mfh\simeq\mbG_a(\mfh)\cdot x_0\subset\mcY_0$. In particular, $\mcY_0$ is an equivariant compactification of the vector group $\mathbb{G}_a(\mfh)$.

(3) There is an action of the Weyl group $W(G)$ on the family $\mcY/\Delta$, which is induced by the $G$-action on the family $\mcX/\Delta$ fixing the section $\sigma(\Delta)$. Furthermore, the action of $W(G)$ on $T\subset \mcY_t$ with $t\neq 0$ and that on $\mathfrak{h}\subset \mcY_0$ are the natural ones.
\end{prop}
\begin{proof}
(1) For $t\neq 0$, the fiber $\mcX_t$ is an equivariant compactification of $G$. The intersection $G\cap\mcY_t$ is equal to the subgroup $T\subset G$, and it is stable under the $T\times T$-action. Then the left $T$-action gives rise to an equivariant embedding $T\simeq T\cdot x_t\subset\mcY_t$. The description of the fan of $\mcY_t=\bar{T}$ follows from \cite[Lemma 6.1.6]{BK}.

(2) Denote by $\mfh_l\subset\mfg_l$, $\mfh_r\subset\mfg_r$ and $\mfh_{ad}\subset\mfg_{ad}$ the Lie subalgebras identified with $\mfh\subset\mfg$. Let $\mcU$ be the vector subbundle of $\mcV:=\pi_*T_{\mcX/\Delta}$ such that $\mcU_t=\mfh_l\oplus\mfh_r$ for $t\neq 0$. Since $\mcU_t$ is  a trivial $T$-module of dimension $2\dim\mfh$ for each $t\neq 0$, so is $\mcU_0$. Then by Lemma \ref{l.jet}(iv) and Lemma \ref{l.isotropic vector fields}, we have $\mcU_0=\mfh\rtimes\mfh_{ad}$. When $t\neq 0$, the manifold $\mcY_t$ is stable under the holomorphic vector fields given by elements in $\mcU_t$. By continuity, the same holds for $t=0$. In particular, $\mathbb{G}_a(\mfh)\subset\Aut^0(\mcY_0)$ and $\mfh\cong \mathbb{G}_a(\mfh) \cdot x_0\subset\mcY_0$. %, where $\mfh_v\subset\mfg_v$ are vector groups associated with $\mfh\subset\mfg$.

(3) Let $N_{G}(T)$ be the normalizer of $T$ in $G$, which stabilizes $\mcY_t=\bar{T}$ for each $t\neq 0$. By continuity, $N_{G}(T)$ stabilizes $\mcY$. The torus $T$ acts trivially on an open subset of $\mcY$, hence it acts trivially on $\mcY$. This gives rise to the action of $W(G)$ on $\mcY/\Delta$. When $t\neq 0$, the action of $N_{G}(T)$ on the open orbit $G\subset\mcX_t$ is by inner automorphisms, and thus the induced action of $W(G)$ on $T=G\cap\mcY_t$ is the natural one. By Lemma \ref{l.jet}(iv), the action of $N_{G}(T)$ on the open orbit $\mfg\simeq \mbG_a^g\subset\mcX_t$ is via the adjoint representation. Then the action of $W(G)$ on $\mfh=\mfg\cap\mcY_0$ is the natural one.
\end{proof}

\begin{example}\label{e.Weylchamberfan}
Given a Weyl chamber $\mfC$ of $G$, denote by $\alpha_1,\ldots,\alpha_n$ the associated simple roots. By \cite[Lemma 6.1.6]{BK}, there is a unique $T$-stable affine space $\mbA^n$ in $\bar{T}$ whose fan consists of $\mfC$ and all its faces. The inclusion $T\subset\mbA^n$ is given by $t\mapsto(\alpha_1(t),\ldots,\alpha_n(t))$.
\end{example}

Now we have the following analogue of Proposition \ref{p.limit}.

\begin{prop}\label{p.family of surfaces}
Suppose there is a homomorphism of algebraic groups $\iota: G'\rightarrow G$ such that $G'$ is the product of a torus and a finite number of general linear groups, and the image $\iota(G')$ contains the maximal torus $T$ of $G$. Then the $G$-action on the family $\mcX/\Delta$ fixing the section $\sigma(\Delta)$ induces an action of $G'$ on this family. Let $\mcY'$ be the connected component of the fixed locus $\mcX^{G'}$ containing the section $\sigma(\Delta)$, which is a subfamily of the family $\mcY/\Delta$ given in Proposition \ref{p.limit}. Then we have the following conclusions.

(i) The morphism $\pi': \mcY'\rightarrow\Delta$ is smooth and projective.

(ii) The identity component $T'$ of the centralizer of $\iota(G')$ in $G$ is a subtorus of $T$.

(iii) For $t \neq 0$, the left $T$-action on $\mcY_t$ induces an equivariant compactification $T'\simeq T'\cdot x_t\subset\mcY'_t$.

(iv) For $t = 0$, the $\mbG_a(\mfh)$-action on $\mcY_0$ induces an equivariant compactification $\mbG_a(\mfh')\simeq\mbG_a(\mfh')\cdot x_0\subset\mcY'_0$, where $\mfh'$ is the Lie algebra of $T'$.

(v) Let $W_{T'}$ be the stabilizer of $T'$ under the $W(G)$-action on $T$. Then the $W(G)$-action on the family $\mcY/\Delta$ induces an action of $W_{T'}$ on the family $\mcY'/\Delta$. Furthermore, the $W_{T'}$-action on $T'\subset \mcY'_t$ with $t\neq 0$ and that on $\mfh'\simeq\mbG_a(\mfh')\subset \mcY_0$ are the natural ones.
\end{prop}

\begin{proof}
Since $T\subset\iota(G')$ and the identity component of the centralizer $C_G(T)$ is $T$ itself, we know that $T'$ is a subtorus of $T$. By definition, $\mcY'$ is a closed subvariety of $\mcY$ and thus $\pi'$ is projective. The smoothness of $\pi'$ follows from  Fogarty (\cite[Theorem 5.4]{F}, see also \cite[Theorem 13.1]{M}). Note that $T\subset\iota(G')$, $\mcY'\subset\mcY$ and the relative dimension of $\mcY'/\Delta$ is the same as $\dim T'$. Then the rest follows from Proposition \ref{p.limit}.
\end{proof}

\subsection{Rigidity of $\bar{G}_2$, $\bar{F}_4$ and $\bar{E}_8$}\label{s.exceptional}

Given a root $\alpha$ of $G$, we denote by $\mfg_\alpha$ and $\mfh_{\alpha^\vee}$ the 1-dimensional root subspace and coroot subspace of $\mfg$. Given a coweight $\omega^\vee$ of $G$, say $\omega^\vee: \C\rightarrow\mfh$, denote by $\mfh_{\omega^\vee}$ the corresponding 1-dimensional subspace of $\mfh$. In particular, $\mfh_{\alpha^\vee}$ and $\mfh_{\omega^\vee}$ are Lie algebras of 1-dimensional torus subgroups of the fixed maximal torus $T$, and $\mfg_\alpha$ is the Lie algebra of a connected subgroup of $G$ which is isomorphic to the 1-dimensional additive group $\mathbb{G}_a$.

%In the section, $G$ is assumed to be of types $E_8$, $F_4$ and $G_2$.

%Let $T'$ be the 2-dimensional torus associated with Lie algebra $\mfh':=\mfh_{\omega_1^\vee}\oplus\mfh_{\omega_n^\vee}\subset\mfh$, where $\omega_i^\vee$ is the $i$-th coweight. Denote by $\bar{T'}$ the closure of $T'$ in the wonderful compactification $\bar{G}$ of $G$.

%\begin{prop}\label{p.G'T'}
%There is a homomorphism of connected reductive groups $j: G'\rightarrow G$ such that the semisimple part of $G'$ is the product of general linear groups, and the centralizer $C_{G}(G')$ is the 2-dimensional torus $T'$ associated with Lie algebra $\mfh':=\mfh_{\omega_1^\vee}\oplus\mfh_{\omega_n^\vee}\subset\mfh$.
%\end{prop}

%Recall that $T'$ is the connected torus subgroup of $T$ associated with the Lie subalgebra $\mfh':=\mfh_{\omega_1^\vee}\oplus\mfh_{\omega_n^\vee}\subset\mfh$. Let $\bar{T'}$ be the closure of $T'$ in the wonderful compactification $\bar{G}$ of $G$. For $t\neq 0$, $\mcY'_t\simeq\bar{T'}$ by Proposition \ref{p.family of surfaces}.

We need the following technical result, which will be proven in the following subsections.

\begin{prop}\label{p.toric surface E8F4G2}
Let $n$ be the rank of $G$, and suppose $G$ is $G_2$, $F_4$ or $E_8$. Then there exists an algebraic group $G'$ which is a torus or a product of general linear groups satisfying the following properties:

(i) There exists a homomorphism of algebraic groups $\iota: G'\rightarrow G$ such that (a)  $\iota(G')$ contains the maximal torus $T$; (b)  the identity component $T'$ of the centralizer $C_G(\iota(G'))$ is the 2-dimensional torus with Lie algebra $\mfh':=\mfh_{\omega_1^\vee}\oplus\mfh_{\omega_n^\vee}$.

(ii) Let $\bar{T'}$ be the closure of $T'$ in the wonderful compactification $\bar{G}$. Then the Picard number $\rho'$ of the toric surface $\bar{T'}$ is equal to 10 (in $G_2$ case) or 6 (in $F_4$ and $E_8$ cases).

(iii) There is a subgroup $W'$ of $W_{T'}$ of order $\rho'+2$ such that its action on the set of irreducible components of the boundary divisor $\partial{\bar{T'}}:=\bar{T'}\setminus T'$ has two orbits.
\end{prop}

We need the following elementary result on finite group actions.

\begin{lem}\label{l.finite group action}
Let $Y$ be a smooth projective variety, and let $A\subset\Aut(Y)$ be a finite subgroup. Suppose the Picard group $\Pic(Y)$  is generated by  divisors $D_{i, j} (1 \leq i \leq m, 1 \leq j \leq k_i)$, where for each $i$,  the set $\{D_{i, 1},\ldots, D_{i, k_i}\}$ is an $A$-orbit. Then $\Pic(Y)^A \otimes \mathbb{Q}$ is generated by $D_i:=\sum_{j=1}^{k_i}D_{i, j}, i=1, \cdots, m$.
\end{lem}

\begin{proof}
The inclusion $\sum_{i=1}^m\mbZ D_i\subset\Pic(Y)^A$ is straight-forward. Any element in  $\Pic(Y)$ can be written as $\mcO(D)$ with $D=\sum_{i=1}^m\sum_{j=1}^{k_i}c_{i, j} D_{i, j}$ for some $c_{i, j}\in\mbZ$. Since $\sum_{a\in A}a\cdot D_{i, j}=\frac{|A|}{k_i}D_i\in\mbZ D_i$, we have $\sum_{a\in A} a\cdot D=\sum_{i=1}^m\sum_{j=1}^{k_i}\frac{c_{i, j}|A|}{k_i}D_i\in\sum_{i=1}^m\mbZ D_i$. If $\mcO(D)\in\Pic(Y)^A$, then $\mcO(|A|D) = \mcO(\sum_{a\in A}a\cdot D)$, which implies that $\mcO(D) \in \sum_i \mathbb{Z} \frac{1}{|A|} D_i$.
\end{proof}

Take the reductive group $G'$, the torus $T'$ and the finite group $W'$ as in Proposition \ref{p.toric surface E8F4G2}. Define the regular family $\mcY'/\Delta$ as in Proposition \ref{p.family of surfaces}, which is the connected component of $\mcX^{G'}$ containing the section $\sigma(\Delta)$.

\begin{prop}\label{p.W' action on Y'0}
Suppose $G$ is $G_2$, $F_4$ or $E_8$. Then the set of irreducible components of the boundary $\partial\mcY'_0:=\mcY'_0\setminus\mfh'$ consists of two $W'$-orbits, and each orbit contains at least two elements.
\end{prop}

\begin{proof}
Since the $W'$-action on the set of irreducible components of $\partial{\bar{T'}}$ has two orbits, the rank of $\Pic(\bar{T'})^{W'}$ is at most two by Lemma \ref{l.finite group action}. The $W'$-action on the family $\mcY'/\Delta$ induces identifications $\Pic(\bar{T'})^{W'}=\Pic(\mcY'/\Delta)^{W'}=\Pic(\mcY'_0)^{W'}$.

Since  $\mcY'_0$ is an equivariant compactification of a vector group, the Picard group of $\mcY'_0$ is freely generated by elements in $S$, where $S$ is the set of irreducible components of the boundary divisor $\partial\mcY'_0:=\mcY'_0\setminus\mfh'$. By Lemma \ref{l.finite group action} again, $S$ consists of one or two $W'$-orbits. Now the Picard number of $\mcY'_0$ is $\rho'\in\{6, 10\}$ and the order of $W'$ is $\rho'+2$. Since $\rho'$ does not divide $\rho'+2$, the set $S$ is not $W'$-transitive. In the case $\rho'=6$ (resp. $\rho'=10$), the set $S$ consists of two orbits of cardinalities 2 and 4 (resp. 4 and 6) respectively.
\end{proof}

Now we can prove the rigidity under Fano deformation of $G_2, F_4$ and $E_8$.
\begin{thm}\label{t.rigidity_EFG}
Assume $G$ is  $G_2, F_4$ or $E_8$ and $X$ its wonderful compactification.
Let $\pi: \mcX\rightarrow\Delta\ni 0$ be a  regular family of Fano varieties such that $\mcX_t  \cong X$ for $t \neq 0$.
Then $\mcX_0 \cong X$.
\end{thm}
\begin{proof}
Assume $\mcX_0$ is not isomorphic to $X$, then by Corollary \ref{c.equiv}, $\mcX_0$ is an equivariant compactification of $\mathbb{G}_a^g$.
Propositions  \ref{p.family of surfaces} and \ref{p.toric surface E8F4G2} give a family of smooth projective surfaces $\mathcal{Y}'/\Delta$  whose general fibers are toric.
Now Proposition \ref{p.W' action on Y'0} implies that the central fiber is an equivariant compactification of $\mathbb{G}_a^2$ whose boundary consists of two $W'$-orbits and each orbit contains two elements. This contradicts Lemma  \ref{p.vector compact surface},  concluding  the proof.
\end{proof}

Now we prove Proposition \ref{p.toric surface E8F4G2} through a  case-by-case argument.

\subsubsection{The case of $G_2$}
Let $e_1, e_2$ and $e_3$ be an orthonormal basis of the inner product space $\mbR^3$. The short roots of $G_2$ are $e_i-e_j$, $i\neq j$, and the long roots are $\pm(e_1+e_2+e_3-3e_k)$, $k=1,2,3$. The simple roots (in Bourbaki's numbering order) are $\alpha_1=e_1-e_2$ and $\alpha_2=-e_1+2e_2-e_3$. The fundamental coweights are $\omega_1^\vee=e_1-e_3$ and $\omega_2^\vee=\frac{1}{3}(e_1+e_2-2e_3)$. The Weyl group $W(G_2)=\langle-id\rangle\times\mathfrak{G}_3$, where $\mathfrak{G}_3$ is the group of permutations of the set $\{e_1, e_2, e_3\}$.  The action of the Weyl group $W(G_2)$ on fundamental coweights is given by:  $W(G_2)\cdot\omega_1^\vee=\{e_i-e_j\mid i\neq j\}$ and $W(G_2)\cdot\omega_2^\vee=\{\pm\frac{1}{3}(e_1+e_2+e_3-3e_k)\mid k=1,2,3\}$.

In this case, take $G'=T'$ to be the maximal torus $T$ and let $W'$  be the whole Weyl group $W(G_2)$ (of order 12). By Proposition \ref{p.limit}, the fan of $\bar{T}$ consists of the Weyl chambers and their faces. Hence $\bar{T}$ is a toric surface of Picard number $12-2=10$.  Now consider the action of $W'$ on the boundary divisors.  An irreducible component of the boundary $\partial\bar{T}$ corresponds to a unique ray of the fan, and this ray is generated by a unique element in $W(G_2)\cdot\{\omega_1^\vee, \omega_2^\vee\}$.  It follows that $W'$ acts on the irreducible components of $\partial{\bar{T'}}$ with two orbits.

\subsubsection{The case of $F_4$}
Let us first recall some  facts on $F_4$ from \cite{A} (page 93-94). Let $e_1, e_2, e_3, e_4$ be an orthonormal basis of the inner product  space $\mbR^4$, then the 24 long roots of $F_4$ are $\pm e_i \pm e_j (i< j)$ while the 24 short roots are $\pm e_i, \frac{1}{2}(\pm e_1 \pm e_2 \pm e_3 \pm e_4)$.
The simple roots (in Bourbaki's numbering order) are $\alpha_1=e_1-e_2$, $\alpha_2=e_2-e_3$, $\alpha_3=e_3$ and $\alpha_4=\frac{1}{2}(-e_1-e_2-e_3+e_4)$. The fundamental coweights are $\omega_1^\vee=e_1+e_4$, $\omega_2^\vee=e_1+e_2+2e_4$, $\omega_3^\vee=e_1+e_2+e_3+3e_4$ and $\omega_4^\vee=2e_4$.

The root system of $F_4$ contains a root subsystem of type $D_4$ such that simple roots of $D_4$ are given by (in Bourbaki's numbering order):  $e_1-e_2, e_2-e_3, e_3-e_4, e_3+e_4$.  In particular, the long roots of $F_4$ are exactly the roots of $D_4$.
The Weyl group $W(F_4) \simeq W(D_4) \rtimes \mathfrak{S}_3$, where $\mathfrak{S}_3$ is the permutation group of 3 letters, being the symmetries of Dynkin diagram of $D_4$. Furthermore, the Weyl group $W(D_4) = H_3 \rtimes \mathfrak{S}_4$, where $\mathfrak{G}_4$ is the group of permutations of the set $\{e_1, e_2, e_3, e_4\}$, and $H_3$ is the kernel of the map $\{\pm 1\}^4 \to \{\pm 1\}$ given by
$(\varepsilon_1, \varepsilon_2, \varepsilon_3, \varepsilon_4) \mapsto \Pi_i \varepsilon_i$. Recall that $\mfg_{\alpha_i}\oplus\mfg_{-\alpha_i}\oplus\mfh_{\alpha_i^\vee}$ is an $\mathfrak{sl}_2$ Lie subalgebra of $\fg$.

\begin{lem}\label{l.G'F4case}
(i) The Lie subalgebra $\mathfrak{gl}_2(\alpha_i):=\mfg_{\alpha_i}\oplus\mfg_{-\alpha_i}\oplus\mfh_{\alpha_i^\vee}\oplus\mfh_{\omega_{i+1}^\vee}$ of $\mfg$ is isomorphic to $\mathfrak{gl}(\C^2)$ for $i\in\{0, 3\}$, where $\alpha_0:=\alpha_2+\alpha_3=e_2$ is a short root of $F_4$.

(ii) There is a homomorphism of algebraic groups $\iota: G':=\rm{GL}(\C^2)\times \rm{GL}(\C^2)\rightarrow G$ such that the Lie algebra of $\iota(G')$ is $\mfg':=\mathfrak{gl}_2(\alpha_0)\oplus \mathfrak{gl}_2(\alpha_3)$.

(iii) We have $\mfh \subset \mfg'$, hence $\iota(G')$ contains a maximal torus of $G$.  Furthermore we have $\mfg^{\mfg'}=\mfh_{\omega_1^\vee}\oplus\mfh_{\omega_4^\vee}$, where
 $\mfg^{\mfg'}:=\{v\in\mfg\mid [v, \mfg']=0\}$.
\end{lem}

\begin{proof}
(i) By considering the Cartan pairings, there is a Lie subalgebra of type $B_2$ such that $\{\alpha_1,\alpha_0\}$ is an ordered set of simple roots, and the restriction of the adjoint representation induces $\mathfrak{gl}_2(\alpha_0)\subset \mathfrak{gl}(\mfg_{\alpha_1}\oplus\mfg_{\alpha_1+\alpha_0}\oplus\mfg_{\alpha_1+2\alpha_0})$, that is the irreducible representation $\Sym^2(\C^2)$ of $\mathfrak{gl}(\C^2)$. Similarly, there is a Lie subalgebra of type $A_2$ such that $\{\alpha_3, \alpha_4\}$ is a set of simple roots, whose adjoint representation induces an isomorphism $\mathfrak{gl}_2(\alpha_3)\simeq  \mathfrak{gl}(\mfg_{\alpha_4}\oplus\mfg_{\alpha_3+\alpha_4})$.

(ii) For $i\in\{0, 3\}$, the adjoint representation of $\mfg$ gives rise to the homomorphism $\iota_i: \rm{GL}(\C^2)\rightarrow G$ of finite kernel such that the Lie algebra of $\iota_i(\rm{GL}(\C^2))$ is $\mathfrak{gl}_2(\alpha_i)$. Since the Cartan pairing $\langle\alpha_0, \alpha_3\rangle=0$, we have $[\mathfrak{gl}_2(\alpha_0), \mathfrak{gl}_2(\alpha_3)]=0$ and thus $\mfg'$ is a Lie subalgebra of $\mfg$. Now $\iota:  G' \ni (a, b) \mapsto\iota_0(a)\cdot\iota_3(b)\in G$ is the required homomorphism.

(iii) Since the coweights $\omega_1^\vee$, $\omega_4^\vee$, $\alpha_0^\vee$ and $\alpha_3^\vee$ are linearly independent, we have $\mfh=\mfh_{\omega_1^\vee}\oplus\mfh_{\omega_4^\vee}\oplus\mfh_{\alpha_0^\vee}\oplus\mfh_{\alpha_3^\vee}$, implying that $\mfh\subset\mfg'$ and $\mfg^{\mfg'}\subset\mfg^{\mfh}=\mfh$. The roots $\alpha_i:\mfh\rightarrow\C$ satisfy that $\mfh_{\omega_1^\vee}\oplus\mfh_{\omega_4^\vee}=\ker\alpha_2\cap\ker\alpha_3=\ker\alpha_0\cap\ker\alpha_3$. It follows that $\mfg^{\mfg'}=\mfh_{\omega_1^\vee}\oplus\mfh_{\omega_4^\vee}$.
\end{proof}

%In this case, the Lie algebra of the subtorus $T'\subset T$ is $\mfh':=\mfh_{\omega_1^\vee}\oplus\mfh_{\omega_4^\vee}\subset\mfh$.

\begin{lem}\label{l.F4 group W'}
Let $\epsilon_k$ be the linear transformation on $\mbR^4$ such that $\epsilon_k(e_i)=(-1)^{\delta_{ik}}e_i$, and let $(e_i, e_j)$ be the linear transformation on $\mbR^4$ induced by permuting  $e_i$ and $e_j$.

(i) Let $W'$ be the subgroup of $W(F_4)$ generated by $\tau_1:=(e_1, e_4)$ and $\tau_2:=\epsilon_1\epsilon_2$. Then $W'$ is of order 8.

(ii) The sublattice $\mbZ\omega_1^\vee\oplus\mbZ\omega_4^\vee$ of the coweight lattice of $F_4$ is stable under the action of $W'$. Moreover, $W'\cdot\omega_1^\vee=\{\pm e_1\pm e_4\}$ and $W'\cdot\omega_4^\vee=\{\pm 2e_1, \pm 2e_4\}$.

(iii) Let $\mbF'$ be the fan in $\mbR\omega_1^\vee\oplus\mbR\omega_4^\vee$ whose rays are those generated by $\pm(2e_1)$, $\pm(2e_4)$, $\pm e_1\pm e_4$ respectively. Then $\{a\cdot(\mbR^+\omega_1^\vee+\mbR^+\omega_4^\vee)\mid a\in W'\}$ is the set of maximal cones.
\end{lem}

\begin{proof}
A direct calculation shows that $W'=\{id, \tau_1, \tau_2, \tau_{2, 1}, \tau_{1, 2}, \tau_{1, 2, 1}, \tau_{2, 1, 2}, \tau_{2, 1, 2, 1}\}$, where $\tau_{i_1, i_2,\ldots, i_k}=\tau_{i_1}\circ\tau_{i_2}\circ\cdots\circ\tau_{i_k}$. The other claims are straight-forward.
\end{proof}

\begin{rmk}
In the setting of Lemma \ref{l.F4 group W'}, we can identify $\mbZ\omega_1^\vee\oplus\mbZ\omega_4^\vee$ with the coweight lattice of $B_2$ such that $\omega_1^\vee$ and $\omega_4^\vee$ are the first and second fundamental coweights of $B_2$, and $W'$ coincides with the Weyl group of $B_2$.  The fan $\mbF'$ is in fact the fan consisting of Weyl chambers of $B_2$ and their faces.  The associated toric surface has 8 boundary divisors, hence it is of Picard number 6.
\end{rmk}

\begin{prop}\label{p.F4 fan}
Let $T'$ be the subtorus of $T$ with Lie algebra $\mfh':=\mfh_{\omega_1^\vee}\oplus\mfh_{\omega_4^\vee}\subset\mfh$, and $\bar{T'}$ its closure in the wonderful compactifcation $\bar{F_4}$. Then the group $W'$ stabilizes the subtorus $T'\subset T$, and the fan of the toric surface $\bar{T'}$ is exactly $\mbF'$.
\end{prop}

\begin{proof}
The subtorus $T'\subset T$ is stabilized by $W'$, since so is the lattice $\mbZ\omega_1^\vee\oplus\mbZ\omega_4^\vee$. Now $\bar{T'}$ is a closed subvariety of the toric variety $\bar{T}$, and the fan of $\bar{T}$ consists of Weyl chambers of $F_4$ and their faces by Proposition \ref{p.limit}(1). The open subset $\bar{T}^{\aff}$ of $\bar{T}$ associated with the positive Weyl chamber $\sum_{i=1}^4\mbR^+\omega_i^\vee$ is isomorphic to the affine space $\mbA^4$. More precisely, we define $ T\rightarrow\mbA^4$ by sending $t\in T$ to $(\alpha_1(t),\ldots, \alpha_4(t))\in\mbA^4$, which makes $\mbA^4$ a toric variety associated with the fan consisting of the positive Weyl chamber $\sum_{i=1}^4\mbR^+\omega_i^\vee$ and its faces, see Example \ref{e.Weylchamberfan}. Hence, $\bar{T'}^{\aff}:=\bar{T'}\cap\bar{T}^{\aff}$ is an affine plane $\mbA^2=\{(v_1, 0, 0, v_4)\in\mbA^4\mid v_1, v_4\in\C\}\subset\mbA^4=\bar{T}^{\aff}$. Furthermore, the fan of $\bar{T'}^{\aff}$ consists of the positive cone $\mbR^+\omega_1^\vee+\mbR^+\omega_4^\vee$ and its faces. Take any $w\in W'$. Then $w\cdot\bar{T'}^{\aff}$ is a $T'$-stable open subset of $\bar{T'}$, whose fan consists of the cone $w\cdot(\mbR^+\omega_1^\vee+\mbR^+\omega_4^\vee)$. By Lemma \ref{l.F4 group W'}, the cones $W'\cdot(\mbR^+\omega_1^\vee+\mbR^+\omega_4^\vee)$ cover the space $\mbR\omega_1^\vee+\mbR\omega_4^\vee$, and thus $W'\cdot\bar{T'}^{\aff}$ gives rise to an open covering of $\bar{T'}$, completing the proof.
\end{proof}

Now Proposition \ref{p.toric surface E8F4G2} for type $F_4$ follows from  Lemma \ref{l.G'F4case}, Lemma \ref{l.F4 group W'} and Proposition \ref{p.F4 fan}.

\subsubsection{The case of $E_8$}
Let us first recall some facts on $E_8$ from \cite[Page 56-59]{A}. Let $e_1,\ldots,e_8$ be an orthonormal basis of the inner product space $\mbR^8$. A set of simple roots of $E_8$ in Bourbaki's numbering order can be given as follows: $\alpha_1=\frac{1}{2}\sum_{i=1}^2 e_i-\frac{1}{2}\sum_{j=3}^8e_j$, $\alpha_2=e_2+e_3$, $\alpha_k=-e_{k-1}+e_k$ for $3\leq k\leq 8$. The roots of $E_8$ are of form $\pm e_i\pm e_j$ with $i\neq j$ or of the form $\frac{1}{2}\sum_{k=1}^8\pm e_k$ with an even number of $-$ signs. There are 112 roots of first form, and 128 roots of second form. The fundamental dominant coweights of $E_8$ are $\omega_1^\vee=2e_1$, $\omega_2^\vee=\frac{5}{2}e_1+\frac{1}{2}\sum_{j=2}^8e_j$, $\omega_3^\vee=\frac{7}{2}e_1-\frac{1}{2}e_2+\frac{1}{2}\sum_{j=3}^8e_j$, $\omega_k^\vee=(9-k)e_1+\sum_{j=k}^8e_j$ for $4\leq k\leq 8$. The Weyl group of $E_8$ is $W(E_8)=\{a\in SO(\mbR^8)\mid a(R)=R\}$ and it is of order $2^{14}\cdot 3^5\cdot 5^2\cdot 7$, where $R$ is the set of roots of $E_8$. The action of $W(E_8)$ on $R$ is transitive.

\begin{lem}\label{l.G'E8case}
(i) There is a Lie subalgebra $\mff_1\subset\mfg$ of type $A_3$ such that $\{\alpha_3,\alpha_4,\alpha_2\}$ is the ordered set of simple roots, and $\mfg'_1:=\mff_1\oplus\mfh_{\omega_1^\vee}$ is a Lie algebra isomorphic to $\mathfrak{gl}(\C^4)$.

(ii) There is a Lie subalgebra $\mff_2\subset\mfg$ of type $A_3$ such that $\{\alpha_0,\alpha_6,\alpha_7\}$ is the ordered set of simple roots, and $\mfg'_2:=\mff_2\oplus\mfh_{\omega_5^\vee}$ is a Lie algebra isomorphic to $\mathfrak{gl}(\C^4)$, where $\alpha_0:=-\sum_{i=2}^7\alpha_i-\sum_{j=4}^6\alpha_j=-e_6-e_7$ is a root of $E_8$.

(iii) There is a homomorphism of algebraic groups $\iota: G':=\rm{GL}(\C^4)\times {\rm GL}(\C^4)\rightarrow G$ such that the Lie algebra of $\iota(G')$ is $\mfg':=\mfg'_1\oplus\mfg'_2$.

(iv)  We have $\mfh \subset \mfg'$, hence $\iota(G')$ contains a maximal torus of $G$.  Furthermore we have $\mfg^{\mfg'}=\mfh_{\omega_1^\vee}\oplus\mfh_{\omega_8^\vee}$, where
 $\mfg^{\mfg'}:=\{v\in\mfg\mid [v, \mfg']=0\}$.
\end{lem}

\begin{proof}
(i) By considering Cartan pairings, there is a Lie subalgebra $\mfl_1\subset\mfg$ of type $A_4$ such that $\{\alpha_1,\alpha_3,\alpha_4,\alpha_2\}$ is the ordered set of simple roots. Let $V$ be the 4-dimensional subspace of $\mfl_1$ generated by root spaces $\mfg_\beta$ for
$$\beta \in \{ \alpha_1, \alpha_1+\alpha_3, \alpha_1+\alpha_3+\alpha_4, \alpha_1+\alpha_2+\alpha_3+\alpha_4\}.$$ The restriction of the adjoint representation of $\mfl_1$ induces $\mfg'_1=\mathfrak{gl}(V)$, and $\mff_1=\mathfrak{sl}(V)$.

(ii) By considering Cartan pairings, there is a Lie subalgebra $\mfl_2\subset\mfg$ of type $D_4$ such that $\{\alpha_0,\alpha_6,\alpha_5,\alpha_7\}$ is the ordered set of simple roots. Let $V$ be the 6-dimensional subspace of $\mfl_2$ generated by root spaces $\mfg_\beta$  for
$$\beta \in \{ \alpha_5, \alpha_5+\alpha_6, \alpha_5+\alpha_6+\alpha_7, \alpha_5+\alpha_6+\alpha_0, \alpha_5+\alpha_6+\alpha_7+\alpha_0, \alpha_5+2\alpha_6+\alpha_7+\alpha_0\}.$$
The restriction of the adjoint representation of $\mfl_2$ induces $\mfg'_2\subset \mathfrak{gl}(V)$, identified with the representation $\wedge^2\C^4$ of $\mathfrak{gl}(\C^4)$. Moreover, $\mff_2=[\mfg'_2, \mfg'_2]\simeq \mathfrak{sl}(\C^4)$.

(iii) For $i\in\{1,2\}$, the adjoint representation of $\mfg$ gives rise to the homomorphism $\iota_i: {\rm GL}(\C^4)\rightarrow G$ of finite kernel such that the Lie algebra of $\iota_i({\rm GL}(\C^4))$ is $\mfg'_i$. Since $[\mfg'_1, \mfg'_2]=0$ and $\mfg'_1\cap\mfg'_2=0$, the space $\mfg':=\mfg'_1\oplus\mfg'_2$ is a Lie subalgebra of $\mfg$. Now $\iota: (a, b)\in G'\mapsto\iota_1(a)\cdot\iota_2(b)\in G$ is the required homomorphism.

(iv) Since the coweights $\omega_1^\vee$, $\omega_5^\vee$, and $\alpha_j^\vee$, $j\in J:=\{0, 2,3,4,6,7\}$, are linearly independent, we have $\mfh=\mfh_{\omega_1^\vee}\oplus\mfh_{\omega_5^\vee}\oplus(\oplus_{j\in J}\mfh_{\alpha_j^\vee})$, implying that $\mfh\subset\mfg'$ and $\mfg^{\mfg'}\subset\mfg^{\mfh}=\mfh$. The roots $\alpha_i:\mfh\rightarrow\C$ satisfy that $\mfh_{\omega_1^\vee}\oplus\mfh_{\omega_8^\vee}=\cap_{j=2}^7\ker\alpha_j=\cap_{j\in J}\ker\alpha_j$, implying the conclusion.
\end{proof}

As in the $F_4$ case, define $\tau_1:=(e_1, e_8)$ and $\tau_2:=\epsilon_1\epsilon_2$ and  $W'$ the finite subgroup of $W(E_8)$ generated by $\tau_1$ and $\tau_2$.  The analogue statements of Lemma \ref{l.F4 group W'} and Proposition \ref{p.F4 fan} hold for $E_8$-case, which  proves Proposition \ref{p.toric surface E8F4G2} for the $E_8$ case. Note that in this case, the toric surface $\bar{T}'$ is the same as that for $F_4$, which is in fact the toric surface with fan consisting of Weyl chambers of $B_2$ and their faces.

\section{Rigidity of the wonderful compactification of $C_n$}

\subsection{A recap on spherical varieties} \label{subsection. Spherical varities}

Wonderful compactifications of simple algebraic groups of adjoint type are spherical varieties. In the following we recall some basic notions and results on spherical varieties. One can consult surveys of Knop \cite{Kn} or Perrin \cite{Pe} for more details.

Let $L$ be a connected reductive group, and let $B \subset L$ be a Borel subgroup. For a closed subgroup $H \subset L$, the homogeneous variety $L/H$ is said to be spherical if it admits an open $B$-orbit. The $B$-stable prime divisors on $L/H$ are called colors of $L/H$, the set of which is denoted by $\mfD(L/H)$.
Let $\C(L/H)^{(B)}$ be the set of nonzero rational functions $f$ on $L/H$ such that there exists a character $\chi_f$ on $B$ satisfying $b\cdot f=\chi_f(b)f$ for all $b\in B$. The constant functions form a subgroup $\C^*$ of $\C(L/H)^{(B)}$, and the quotient group $\Lambda(L/H):=\C(L/H)^{(B)}/\C^*$ is called the weight lattice of $L/H$. In what follows, by a valuation on $L/H$, saying $\nu$,  we always mean it is a valuation $\nu: \C(L/H)^*\rightarrow\mathbb{Q}$ taking values in $\mathbb{Q}$ and vanishing on $\mathbb{C}^*$. Given a valuation $\nu$ on $L/H$, we can define $\varrho_\nu: \Lambda(L/H)\rightarrow\mbQ$ by sending $f$ to $\nu(f)$. This induces a map $\varrho$ from the set of valuations on $\C(L/H)$ to $\Lambda^{\vee}_{\mbQ}(L/H):=\Hom(\Lambda(L/H), \mbZ)\otimes\mbQ$. The set of images of all $L$-invariant valuations on $L/H$ is called the valuation cone of $L/H$ and denoted by $\mcV(L/H)\subset\Lambda^{\vee}_{\mbQ}(L/H)$.

A colored cone is a pair $(\mfC, \mfD)$ such that $\mfD\subset\mfD(L/H)$, $\mfC$ is a cone in $\Lambda^{\vee}_{\mbQ}(L/H)$ generated by $\varrho(\mfD)$ and finitely many elements in $\mathcal{V}(L/H)$ satisfying $\mfC^o\cap\mcV(L/H)\neq\emptyset$, where $\mfC^o$ denotes the relative interior  of the cone $\mfC$.   The pair $(\mfC_0, \mfD_0)$ is called a colored face of $(\mfC, \mfD)$ if $\mfC_0$ is a face of $\mfC$, $\mfC_0^o\cap\mcV(L/H)\neq\emptyset$ and $\mfD_0=\mfD\cap\varrho^{-1}(\mfC_0)$. A colored fan $\mbF$ is a nonempty finite set of colored cones such that colored faces of colored cones in $\mbF$ also belong to $\mathbb{F}$, and $\mfC_1^o\cap\mfC_2^o\cap\mcV(L/H)=\emptyset$ for pairwise different colored cones $(\mfC_1, \mfD_1)$ and $(\mfC_2, \mfD_2)$ in $\mbF$. The colored fan $\mbF$ is called strictly convex if $(0, \emptyset)\in\mbF$.

An $L$-spherical variety is a normal variety $X$ which admits an $L$-equivariant open embedding $L/H\subset X$ for a homogeneous spherical variety $L/H$.  Sometimes we call such $X$ a spherical $L/H$-embedding.  In this case, $B$-stable prime divisors on $X$ consist of closures of $B$-stable prime divisors on $L/H$ and those contained in the boundary $\partial X:=X\setminus L/H$.
There are finitely many $L$-orbits on $X$.
Given an $L$-orbit $Y$ in $X$, we can define the associated colored cone $(\mfC_Y, \mfD_Y)$ such that $\mfC_Y$ is the cone in $\Lambda^{\vee}_{\mbQ}(L/H)$ generated by those $\varrho(D)$, where $D$ runs over the set of $B$-stable prime divisors on $X$ containing $Y$, and $\mfD_Y$ is the set of $D\in\mfD(L/H)$ such that its closure satisfies $\bar{D}\supset Y$.
The set of colored cones associated to all $L$-orbits on $X$ is a strictly convex colored fan in $\Lambda^{\vee}_{\mbQ}(L/H)$, which is called the colored fan of $X$, and denoted by $\mbF_X$. Given two $L$-orbits $Y$ and $Z$ on $X$, then $\bar{Y}\subset\bar{Z}$ if and only if $(\mfC_Z, \mfD_Z)$ is a colored face of $(\mfC_Y, \mfD_Y)$. In particular, $Y=Z$ if and only if $(\mfC_Y, \mfD_Y)=(\mfC_Z, \mfD_Z)$. For the convenience of discussions, we also write $\mfC_{\bar{Y}}:=\mfC_Y$ and $\mfD_{\bar{Y}}:=\mfD_Y$ for the closure $\bar{Y}$ of an orbit $Y$.

Here we summarize some of the basic results of the theory of  spherical varieties.
\begin{prop}\label{p.BasicsSV}
Let $L/H$ be a homogeneous spherical variety.

(i)  Given a strictly convex colored fan $\mbF$ in $\Lambda^{\vee}_{\mbQ}(L/H)$, there is a unique (up to $L$-equivariant isomorphisms) spherical $L/H$-embedding $X$ such that $\mbF_X=\mbF$.

(ii) Let $X$ be a smooth $L/H$-embedding, and let $S_Y$ be the set of $B$-stable prime divisors containing the orbit $Y$ on $X$. Then $\{\varrho(\nu_D)\mid D\in S_Y\}$ is a linearly independent set in $\Lambda^{\vee}_{\mbQ}(L/H)$.

(iii) A spherical $L/H$-embedding  $X$ is complete if and only if $\mcV(L/H)$ is contained in the union of  colored cones of $X$.

(iv) The Chow group $A^1(X)$ of a spherical variety $X$  is generated by $B$-stable prime divisors.  The rational equivalences are generated by relations $\Div(f)=0$, for $f\in\C(L/H)^{(B)}$.

(v) The cycle map $A^1(X)\rightarrow H^2(X, \mbZ)$ of a smooth projective spherical variety $X$ is an isomorphism.

(vi) The anti-canonical divisor of a spherical $L/H$-embedding $X$ is given by $-K_X=\sum_{D\in\mfD(L/H)}m_D D+\sum_{j=1}^n D_j,$ where the coefficients $m_D$ only depend on the open orbit $L/H$, and $D_1,\ldots, D_n$ are the prime boundary divisors.
\end{prop}

Take a spherical $L/H$-embedding $X$ and a spherical $L/H'$-embedding $X'$ with $H\subset H'$. The natural morphism $\psi: L/H\rightarrow L/H'$ induces a homomorphism $\Lambda(L/H')\rightarrow\Lambda(L/H)$, and thus a homomorphism $\psi_*: \Lambda^{\vee}_{\mbQ}(L/H)\rightarrow\Lambda^{\vee}_{\mbQ}(L/H')$. We say $\psi_*$ sends $\mbF_X$ to $\mbF_{X'}$ if given any $(\mfC, \mfD)\in\mbF_X$ there exists $(\mfC', \mfD')\in\mbF_{X'}$ such that $\psi_*(\mfC)\subset\mfC'$, and each element $D\in\mfD$ either dominates $L/H'$ or satisfies $\psi_*(D)\in\mfD'$.
\begin{prop} \label{p.ExtensionSV}
The morphism $\psi: L/H\rightarrow L/H'$ can be extended to an $L$-equivariant morphism $\Psi: X\rightarrow X'$ if and only if $\psi_*$ sends
$\mbF_X$ to $\mbF_{X'}$.
\end{prop}
%The following two lemmas can be found in \cite{Br07} and \cite{BK} respectively.

%\begin{lem}\cite[Example 2.1.3]{Br07}
%The weight lattice of the homogeneous space $G=(G\times G)/\diag(G)$ is $\Lambda(G)=\oplus_{i=1}^n\mbZ\alpha_i$. The dual lattice $\Lambda^{\vee}(G)=\oplus_{i=1}^n\mbZ\omega^{\vee}_i$, satisfying $(\omega^{\vee}_i, \alpha_j)=\delta_{ij}$, is identified with the weight lattice $\Lambda$ of the group $G$. The boundary divisor $D_i$ associated with root $\alpha_i$ satisfies that $\varrho(D_i)=-\omega^{\vee}_i$ and color $D(\omega_j)$ associated with $\omega_j$ satisfies that $\varrho(D(\omega_j))=\alpha^{\vee}_j$. Here $\alpha^{\vee}_1,\ldots,\alpha^{\vee}_n$ is the set of simple roots of the dual root system $S^\vee$ of the root system $S$, and $\omega^{\vee}_1,\ldots, \omega^{\vee}_n$ is the set of fundamental dominant weights of $S^\vee$. The valuation cone of $G$ is the negative Weyl chamber, i.e. $\mcV(G)=\cone\{-\omega^{\vee}_1,\ldots,-\omega^{\vee}_n\}$. The colored cone of the orbit closure $D_I:=\cap_{i\in I}D_I$ is $(\cone\{-\omega_i\mid i\in I\}, \emptyset)$.
%\end{lem}

%\begin{lem}\cite[Proposition 6.1.11]{BK}
%The restriction of the color $D(\omega_k)$ on the closed orbit $D_S:=\cap_{i=1}^n D_i= G/B^-\times G/B\subset\mbP(V_\lambda^*\otimes V_\lambda)$ is the Cartier divisor $\mcO(-w_0\omega_k, \omega_k)$, where $\lambda$ is a regular dominant weight, and $w_0$ is the longest element in the Weyl group of $G$.
%\end{lem}

\begin{example}\cite[Example 2.1.3]{Br07}\label{e.colored cones}
Let $X$ be the wonderful compactification of a simple linear algebraic group $G$ of adjoint type and of rank $n$. It is a $(G\times G)$-spherical variety with open orbit being $O=(G\times G)/\diag(G) \simeq G$. The weight lattice of $O$ as a homogeneous spherical variety coincides with the root lattice of $G$, i.e. $\Lambda(O)=\oplus_{i=1}^n\mbZ\alpha_i$, where $\alpha_1, \ldots, \alpha_n$ is the set of simple roots. The dual lattice $\Lambda^{\vee}(O)=\oplus_{i=1}^n\mbZ\omega^{\vee}_i$ with $(\omega^{\vee}_i, \alpha_j)=\delta_{ij}$ is identified with the coweight lattice of $G$. Recall that $\Pic(X)$ is identified with the weight lattice of $G$. The boundary divisor $D_i$ associated to the root $\alpha_i$ (i.e. $D_i=\alpha_i\in\Pic(X)$) satisfies $\varrho(D_i)=-\omega^{\vee}_i$ and  the color $D(\omega_j)$ associated to $\omega_j$ (i.e. $D(\omega_j)=\omega_j\in\Pic(X)$) satisfies $\varrho(D(\omega_j))=\alpha^{\vee}_j$. Let $R$ be the root system of $G$. Then $\alpha^{\vee}_1,\ldots,\alpha^{\vee}_n$ is the set of simple roots of the dual root system $R^{\vee}$ of $R$, and $\omega^{\vee}_1,\ldots, \omega^{\vee}_n$ is the set of fundamental dominant weights of $R^{\vee}$. The valuation cone of $O$ is the negative Weyl chamber, i.e. $\mcV(O)=\cone\{-\omega^{\vee}_1,\ldots,-\omega^{\vee}_n\}$. The colored fan $\mbF_X$ of $X$ consists of all colored faces of $(\mcV(O), \emptyset)$. The colored cone of the open orbit $D_I^o$ of $D_I:=\cap_{i\in I}D_i$ is $(\cone\{-\omega_i\mid i\in I\}, \emptyset)$. In this case, the anti-canonical divisor $-K_X$ is given by $-K_X =2\sum_{i=1}^n D(\omega_i)+\sum_{j=1}^n D_j$.
\end{example}

\begin{lem}\cite[Propositon 6.1.11]{BK}\label{l.divDR}
Let $X$ be the wonderful compactification of a simple linear algebraic group $G$ of adjoint type. The restriction of the color $D(\omega_k)$ to the closed orbit $D_{1,\ldots,n}:=\cap_{i=1}^n D_i= G/B^-\times G/B$ is the Cartier divisor $\mcO(-w_0\omega_k, \omega_k)$, where $w_0$ is the longest element in the Weyl group of $G$.
\end{lem}

%\begin{proof}
%Take any dominant weight $\lambda$, and let $V_\lambda$ be the irreducible representation of $G$ with highest weight $\lambda$. The closure $X_\lambda\subset\mbP\End(V_\lambda)$ of the projective representation $G\subset\PGL(V_\lambda)$ has a unique closed $(G\times G)$-orbit $G/P_\lambda\times G/P_\lambda$. The pulling-back of $\mcO_{\mbP\End(V_\lambda)}(1)$ through the $(G\times G)$-equivariant morphism $X\rightarrow X_\lambda\subset\mbP\End(V_\lambda)$ is $\mcO_X(D_i)$. Its restriction on $D_S$ coincides with the pulling-back of $\mcO_{\mbP\End(V_\lambda)}(1)$ through the $(G\times G)$-equivariant morphism $D_S=G/B\times G/B\rightarrow G/P_\lambda\times G/P_\lambda\subset\mbP\End(V_\lambda)$.
%\end{proof}

\subsection{Lagrangian Grassmannians} \label{s.LG}

Let $(W, \Omega)$ be a symplectic vector space of dimension $2n$.  Take two copies $(W_i, \Omega_i), i=1,2,$ of $(W, \Omega)$ and endow the vector space $W_1 \oplus W_2$ the symplectic form given by $p_1^* \Omega_1 - p_2^* \Omega_2$, where $p_i: W_1 \oplus W_2 \to W_i$ is the natural $i$-th  projection.  Let ${\rm LG}(2n, W_1 \oplus W_2)$ be the Lagrangian Grassmannian associated to the symplectic vector space $W_1 \oplus W_2$, which parameterizes Lagrangian subspaces in $W_1 \oplus W_2$. Note that ${\rm LG}(2n, W_1 \oplus W_2)$ is a smooth projective $\Sp(2n) \times \Sp(2n)$-equivariant compactification of $\Sp(2n)$ (embedded by sending each $g \in \Sp(2n)$ to its graph). In particular, we have  $\dim {\rm LG}(2n, W_1 \oplus W_2) =\dim \Sp(2n)= n (2n+1)$. We will identify $W_i$ with the $i$-th summand in $W_1 \oplus W_2$.

\begin{lem} \label{l.equaldim}
Let $[V] \in {\rm LG}(2n, W_1 \oplus W_2)$, then $\dim V \cap W_1 = \dim V \cap W_2$.
\end{lem}
\begin{proof}
Let $V_i = V \cap W_i$. Note that $V_1$ is the kernel of $p_2|_V: V \to W_2$.
Let $I_2= p_2(V)$ be the image, then $\dim I_2 = \dim V/V_1$, which gives that $\dim W_2 \cap I_2^\perp = \dim V_1$.
As $V \subset W_1 \oplus I_2$, we get
$V=V^\perp \supset W_1^\perp \cap I_2^\perp = W_2 \cap I_2^\perp, $
which gives $V_2 = V \cap W_2 \supset W_2 \cap I_2^\perp$.  This implies that $\dim V_2 \geq \dim V_1$.  By  symmetry, we  have $\dim V_1 = \dim V_2$.
\end{proof}

For any $0 \leq k \leq n$, we denote by $\IG(k, W_i)$ the $k$-th isotropic Grassmannian in $W_i$, which parameterizes $k$-dimensional isotropic linear subspaces in $W_i$.
Define
$$\LG(k):=\{ [V] \in {\rm LG}(2n, W_1 \oplus W_2) | \dim V \cap W_1 = \dim V \cap W_2 = k\}.$$
Consider the map $\pi_k: \LG(k) \to \IG(k, W_1) \times \IG(k, W_2)$ defined by $[V] \mapsto ([V \cap W_1], [V \cap W_2])$, which is an $\Sp(W_1) \times \Sp(W_2)$-homogeneous fibration.
\begin{lem}\label{l.OrbitLG(k)}
\begin{itemize}
\item[(i)] The fiber of $\pi_k$ is isomorphic to $\Sp(\mathbb{C}^{2n-2k})$.
\item[(ii)] The variety $\LG(k)$ is an $\Sp(W_1) \times \Sp(W_2)$-orbit and ${\rm LG}(2n, W_1 \oplus W_2) = \bigsqcup_{k=0}^n \LG(k)$ is the decomposition of ${\rm LG}(2n, W_1 \oplus W_2)$ into  $\Sp(W_1) \times \Sp(W_2)$-orbits.
\end{itemize}
\end{lem}
\begin{proof}
Consider first the case $k=0$. Any element $[V] \in \LG(0)$ is the graph of a unique linear isomorphism $\ell_V: W_1 \to W_2$.  Note that $V$ is Lagrangian if and only if $\ell_V$ is a  symplectomorphism, which shows that
$\LG(0) \simeq \Sp(W)$.

In general, fix $[V_i] \in \IG(k, W_i)$.  Let $\bar{W_i} = V_i^{\perp, W_i}/V_i$, which is a symplectic vector space of dimension $2n-2k$.  Any element $[V] \in \pi_k^{-1}([V_1], [V_2])$ is uniquely determined by
$\bar{V}:=V/(V_1 \oplus V_2)$. Note that $V$ is Lagrangian if and only if so is $\bar{V} \subset \bar{W_1} \oplus \bar{W_2}$, while the latter satisfies $\bar{V} \cap \bar{W_i} = \{0\}$.  This shows that $\pi_k^{-1}([V_1], [V_2]) \simeq \Sp(\bar{W_i}),$ proving (i).  To prove (ii), we note that the stabilizer in $\Sp(W_1) \times \Sp(W_2)$ of $([V_1], [V_2])$ surjects  to $\Sp(\bar{W_1})$, hence $\LG(k)$ is  $\Sp(W_1) \times \Sp(W_2)$-homogeneous.
\end{proof}

Recall that $\dim \IG(k, W) = \frac{k(4n+1-3k)}{2}$ and $\dim \Sp(2n)= 2n^2+n$, which implies immediately the following
\begin{cor}\label{c.dimLG}
\begin{itemize}
\item[(a)] We have $\dim \LG(k) = (2n^2+n)-k^2$, i.e. ${\rm codim}(\LG(k)) = k^2$.
\item[(b)] The subvariety $\LG(0) \subset {\rm LG}(2n, W_1 \oplus W_2)$ is open and the closure of $\LG(1)$ is the unique $\Sp(W_1) \times \Sp(W_2)$-stable prime divisor. The unique closed $\Sp(W_1) \times \Sp(W_2)$-orbit of ${\rm LG}(2n, W_1 \oplus W_2)$ is $\LG(n)$, which is isomorphic to $\LG(n, W_1) \times \LG(n, W_2)$.
\end{itemize}
\end{cor}

\begin{lem} \label{l.LGClosure}
We have $\overline{\LG(k)} = \cup_{j \geq k} \LG(j)$.
\end{lem}
\begin{proof}
By Corollary \ref{c.dimLG}, we have $\overline{\LG(0)} = {\rm LG}(2n, W_1 \oplus W_2) = \cup_{j \geq 0} \LG(j)$. For $[V_i] \in \IG(k, W_i)$, let $\bar{W_i} = V_i^{\perp, W_i}/V_i$. Then
$$\overline{\pi_k^{-1}([V_1], [V_2])} = \LG(2n-2k, \bar{W_1}\oplus \bar{W_2}),$$
which implies that $\overline{\LG(k)} \cap \LG(j) \neq \emptyset $ if $j \geq k$.  As $\overline{\LG(k)}$ is $\Sp(W_1) \times \Sp(W_2)$-invariant and $\LG(j)$ is an orbit, we have $\LG(j) \subset \overline{\LG(k)}$.
This gives a chain
$$
\LG(n) \subset \overline{\LG(n-1)}\subset  \cdots \subset \overline{\LG(0)} = {\rm LG}(2n, W_1 \oplus W_2),
$$
which proves the claim.
\end{proof}

Consider the symplectic involution $\tau:=(1_{W_1}, -1_{W_2}) \in \Sp(W_1) \times \Sp(W_2)$, which acts naturally on ${\rm LG}(2n, W_1 \oplus W_2)$. Note that an element $[V] \in {\rm LG}(2n, W_1 \oplus W_2)$ is fixed by $\tau$ if and only if $V = V_1 \oplus V_2$, where $V_i = V \cap W_i$. This gives that ${\rm LG}(2n, W_1 \oplus W_2)^\tau = \LG(n) \simeq \LG(n, W_1) \times \LG(n, W_2)$.  Let $ Z={\rm LG}(2n, W_1 \oplus W_2)/\tau$ and $\psi: {\rm LG}(2n, W_1 \oplus W_2) \to Z$ the natural projection.
 Denote $Z^\circ = \LG(0)/\tau$, $Z_j^o:=\LG(j)/\tau$ and  $Z_j : = \overline{\LG(j)}/\tau$ for $j\geq 1$, then
$Z_n \subset Z_{n-1} \subset \cdots \subset Z_1 \subset \overline{Z^\circ}=Z$ by Lemma \ref{l.LGClosure}.  Let $K = {\rm PSp}(W_1) \times {\rm PSp}(W_2)$, which acts on $Z$, preserving each $Z_j$.
\begin{prop} \label{p.LGSpherical}
(i) The Lagrangian Grassmannian ${\rm LG}(2n, W_1 \oplus W_2)$ is spherical under the $\Sp(W_1) \times \Sp(W_2)$-action with the unique open orbit $\LG(0) \simeq \Sp(W)$ and the unique closed orbit $\LG(n) \simeq \LG(n,W_1) \times \LG(n, W_2)$.

(ii) The variety $Z$ is spherical under the $K$-action, with unique open orbit $Z^\circ \simeq  {\rm PSp}(W)$ and  unique closed orbit $Z_n \simeq \LG(n,W_1) \times \LG(n, W_2)$. The boundary $Z_1 = Z \setminus Z^\circ$ is an irreducible divisor and each $Z_i$ is a $K$-orbit closure in $Z$.
\end{prop}
\begin{proof}
Let $W' = \{(w, w) \in W_1 \oplus W_2| w\in W\}$, which is Lagrangian in $W_1 \oplus W_2$.  The isotropy group of $[W']$ is the diagonal subgroup of $\Sp(W) \times \Sp(W)$, hence the orbit of $[W']$ is isomorphic to $\Sp(W)$. This implies that ${\rm LG}(2n, W_1 \oplus W_2)$ is spherical, which proves (i) by Corollary \ref{c.dimLG}.

We denote by $[W'_\tau] \subset  Z$ the image of $[W']$ in $Z$. Then the stabilizer $K_{[W'_\tau]}$ is isomorphic to the diagonal subgroup of ${\rm PSp}(W_1) \times {\rm PSp}(W_2)$, hence the open orbit $Z^\circ$ is isomorphic to ${\rm PSp}(W)$ and $Z$ is spherical.  Then the claim follows from (i).
\end{proof}

\subsection{A construction of the wonderful compactification in type  $C_n$} \label{s.C}

Let $X$ be the wonderful compactification of ${\rm PSp}(W)$ and let $X^\circ$ be the open orbit. By Example \ref{e.colored cones}, $X$ is a spherical variety under the $K$-action. Both $Z$ and $X$ are spherical varieties under the $K$-action, with an isomorphic open orbit
$Z^\circ \simeq X^\circ \simeq {\rm PSp}(W)$. The boundary divisor $\partial X := X \setminus X^\circ$ has a decomposition into irreducible components  $\cup_{i=1}^n D_i$, where $D_i$ is the divisor corresponding to the simple root $\alpha_i$.

\begin{prop}\label{p.LGEtendsion}
The natural isomorphism $\phi^\circ: X^\circ \to Z^\circ$ extends to a $K$-equivariant birational morphism $\phi: X \to Z$.
\end{prop}

\begin{proof}
The isomorphism $\phi^\circ$ induces an identity $\phi^\circ_*: \Lambda^\vee_{\mbQ}(X^\circ)=\Lambda^\vee_{\mbQ}(Z^\circ)$. By Example \ref{e.colored cones}, $\mbF_X$ consists of colored faces of $(\mcV({\rm PSp}(W)), \emptyset)$, where $\mcV({\rm PSp}(W))=\cone\{-\omega^{\vee}_1,\ldots,-\omega^{\vee}_n\}$. As $Z$ is projective, the valuation cone $\mcV({\rm PSp}(W))$ is contained in the union of the colored cones of $Z$ by Proposition \ref{p.BasicsSV}(iii).
As $Z$ has a unique closed orbit $Z_n$,  we have  $\mfC_{Z_n}\supset\mcV({\rm PSp}(W))$. By definition $\phi^\circ_*$ sends any element $(\mfC, \mfD)\in\mbF_X$ to $(\mfC_{Z_n}, \mfD_{Z_n})$, hence $\phi^\circ_*$ sends $\mbF_X$ to $\mbF_Z$. It follows that $\phi^\circ$ can be extended to a $K$-equivariant birational  morphism $\phi: X \to Z$ by Proposition \ref{p.ExtensionSV}.
\end{proof}

\begin{lem}\label{l.Exc}
The morphism $\phi$ sends the open orbit $D_1^o$ of $D_1$ isomorphically to the open orbit of $Z_1$, and the exceptional locus of $\phi$ is $\cup_{i=2}^n D_i$.
\end{lem}

\begin{proof}
By Proposition \ref{p.LGSpherical}(ii), the boundary $Z_1$ of $Z$ is irreducible.
Since $\phi$ is a $K$-equivariant birational morphism with connected fibers, there exists a unique boundary divisor $D_k$ such that $\phi$ sends its open orbit $D_k^o$ isomorphically to the open orbit of $Z_1$, and the exceptional locus is  $\Exc(\phi)=\cup_{i\neq k}D_i$. By the description of orbit closures on wonderful compactifications in Section \ref{s.wonderful compact}, the orbit $D_k^o$ admits a ${\rm PGL}(\C^k) \times {\rm PSp}(\C^{2n-2k})$-fibration over ${\rm IG}(k, W_1)\times {\rm IG}(k, W_2)$. It follows that $k=1$ by Lemma \ref{l.OrbitLG(k)}.
\end{proof}

%{\bf Add the colored cone description of them}

\begin{lem}\label{l.colored fan Z}
The dual weight lattice of $Z^o$ is $\Lambda^\vee(Z^o)=\oplus_{i=1}^{n}\mbZ\omega^\vee_i$, and the valuation cone $\mcV(Z^o)=\cone\{-\omega^\vee_1,\ldots,-\omega^\vee_n\}$. The colored cone of the orbit $Z_k^o$ with $1\leq k\leq n$ is $(\cone\{-\omega_1^\vee, \alpha^\vee_1,\ldots,\alpha^\vee_{k-1}\}, \{D(\omega_1),\ldots, D(\omega_{k-1})\})$.
\end{lem}

\begin{proof}
The descriptions of $\Lambda^\vee(Z^o)$ and $\mcV(Z^o)$ are clear from Example \ref{e.colored cones}, since they  depend only on the open orbit $Z^o={\rm PSp}(W)$. Set $\mfC_k:=\cone\{-\omega_1^\vee, \alpha^\vee_1,\ldots,\alpha^\vee_{k-1}\}$ and $\mfD_k:=\{D(\omega_1),\ldots, D(\omega_{k-1})\}$ for $1\leq k\leq n$. We claim that the colored cone of $Z_n$ is $(\mfC_n, \mfD_n)$. By the claim, there is a chain of elements in $\mbF_Z$ as follows: $(0, \emptyset)\subset(\mfC_1, \mfD_1)\subset\cdots \subset (\mfC_{n-1}, \mfD_{n-1})\subset(\mfC_n, \mfD_n)$.
By the bijective correspondence with orbit closures, these are all colored cones of $Z$ and the colored cone of $Z_k$ is $(\mfC_k, \mfD_k)$.

To verify the claim, we consider the closed orbit $D_{1,\ldots, n}:=G/B^-\times G/B$ on $X$, where $G={\rm PSp}(W)$ and $B$ is the Borel subgroup.
To simplify the notations, let $X_n = D_{1,\ldots, n}$.
The map $\phi$ sends $X_n$ to the unique closed orbit $Z_n=G/P_n^-\times G/P_n$ on $Z$, which is induced by the inclusions $B\subset P_n$ and $B^-\subset P_n^-$. Since $\mcO_X(D(\omega_k))|_{X_n}=\mcO_{X_n}(\omega_k, \omega_k)$ by Lemma \ref{l.divDR}, we know that $\phi(D(\omega_k)\cap X_n)=Z_n$ if and only if $k\neq n$. By the definition of colored cones, $\mfD_{Z_n}=\mfD_n$ and $\mfC_{Z_n}$ is the cone in $\Lambda^\vee(Z^o)=\oplus_{i=1}^{n}\mbZ\omega^\vee_i$ generated by $\varrho(\mfD_n)$ and $\varrho(Z_1)$. By Example \ref{e.colored cones} and Lemma \ref{l.Exc}, $\varrho(Z_1)=-\omega^\vee_1$ and $\varrho(\mfD_n)=\{\alpha^\vee_1,\ldots,\alpha^\vee_{n-1}\}$, completing the proof.
\end{proof}

\begin{lem}
The dual weight lattice of ${\rm LG}(0)$ is $\Lambda^\vee(\LG(0))=\oplus_{i=1}^{n}\mbZ\alpha^\vee_i$, and the valuation cone is $\mcV({\rm LG}(0))=\cone\{-\omega^\vee_1,\ldots,-\omega^\vee_n\}$. For each $k$, the colored cone of the orbit $\LG(k)$ is given by  $\mfC_{\LG(k)}=\cone\{-\omega_1^\vee, \alpha^\vee_1,\ldots,\alpha^\vee_{k-1}\}$ and  $\mfD_{\LG(k)}=\{D(\omega_1),\ldots, D(\omega_{k-1})\}$.
\end{lem}

\begin{proof}
The surjective quotient  map $\LG(0)\rightarrow Z^o$ induces $\Lambda^\vee(Z^o)=\oplus_{i=1}^n\mbZ\omega^\vee_i\supset\Lambda^\vee({\rm LG}(0))=\oplus_{i=1}^n\mbZ\alpha^\vee_i$.
Under the identity $\Lambda^\vee_{\mbQ}({\rm LG}(0))=\Lambda^\vee_{\mbQ}(Z^o)$, we have $\mcV({\rm LG}(0))=\mcV(Z^o)$, $\mfC_{{\rm LG}(k)}=\mfC_{Z_k^o}$, $\mfD_{{\rm LG}(k)}=\mfD_{Z_k^o}$ for each $k$. Then the claim follows from Lemma \ref{l.colored fan Z}.
\end{proof}

\begin{prop} \label{p.LGtoWC}
(1) The morphism $\phi$ satisfies  $\phi(D_k)=Z_k$ for each $k$.

(2) Let $Y$ be the fiber product $X \times_Z {\rm LG}(2n, W_1 \oplus W_2)$ and let $X \xleftarrow{\psi'} Y \xrightarrow{\phi'} {\rm LG}(2n, W_1 \oplus W_2)$ be the two projections. Then $\psi': Y \to X$ is a double cover with branch locus $D_n$.

(3) The birational map $\phi$ is the composition of successive blowups of $Z$ along the strict transforms of $Z_n, Z_{n-1}, \cdots, Z_2$, from the smallest to the biggest.

(4) The degree 2 rational map $\psi' \circ (\phi')^{-1}: {\rm LG}(2n, W_1 \oplus W_2) \dasharrow X$ preserves the minimal rational curves and the VMRT-structures.
\end{prop}

\begin{proof}
(1) By Example \ref{e.colored cones}, $\varrho(D_k)=-\omega_k^\vee$ and $\varrho(D(\omega_k))=\alpha_k^\vee$ for all $k$. It follows that
$\varrho(D_k)=k\varrho(D_1)+\sum\limits_{j=1}^{k-1}(k-j)\varrho(D(\omega_j))$ in $\Lambda^\vee_\mbQ({\rm PSp}(W))$. Then $\mfC_{D_k}^o\subset\mfC_{Z_k}^o$ by Lemma \ref{l.colored fan Z}, which implies that $\phi$ sends the orbit $D_k^o$ onto the orbit $Z_k^o$, and thus $\phi(D_k)=Z_k$.

(2) Since $\psi: {\rm LG}(2n, W_1 \oplus W_2) \to Z$ is a double cover with branch locus $Z_n$, the morphism $\psi': Y \to X$ is a double cover with branch locus $\phi^{-1}(Z_n)=D_n$.

(3) For $0 \leq i\leq n-1$,  set $\mfC'_i:=\cone(\alpha^\vee_1,\ldots,\alpha^\vee_i, -\omega^\vee_1,-\omega^\vee_{i+2},\ldots,-\omega^\vee_n)$ and $\mfD'_i:=\{D(\omega_1),\ldots,D(\omega_i)\}$. The set $\mbF_i$ of all colored faces of $(\mfC'_i, \mfD'_i)$ is a colored fan in $\Lambda^\vee_{\mbQ}({\rm PSp}(W))$. Let $Y_i$ be the unique complete spherical $K/\diag({\rm PSp}(W))$-embedding such that $\mbF_{Y_i}=\mbF_i$. There is a chain of $K$-equivariant birational morphism with connected fibers as follows:
\begin{eqnarray}\label{e.chainY}
X\cong Y_0\xrightarrow{\pi_0} Y_1\xrightarrow{\pi_1}Y_2\xrightarrow{\pi_2}\cdots\xrightarrow{\pi_{n-2}}Y_{n-1} \cong Z.
\end{eqnarray}

For each $1\leq k\leq n-1$, $-\omega_{k+1}^\vee$ lies in the relative interior of $\mfC''_k:=\cone(-\omega_1^\vee,\alpha_1^\vee,\ldots,\alpha_k^\vee)$, and the two colored fans $\{(\mfC, \mfD)\in\mbF_{k-1}\mid -\omega_{k+1}^\vee\notin\mfC\}$ and $\{(\mfC, \mfD)\in\mbF_k\mid \mfC''_k\nsubseteq\mfC\}$ are the same one.
Thus $\pi_{k-1}$ is a birational morphism with connected fibers such that the exceptional locus  is the strict transform $\widetilde{Z}_{k+1}\subset Y_k$ of $Z_{k+1}$ and the exceptional locus on $Y_{k-1}$ is a prime divisor which is the birational image of the boundary divisor $D_{k+1}$.

By \cite[Section 2]{Hu},  the  wonderful compactification $\overline{{\rm Sp}(W)}$ of the simple group ${\rm Sp}(W)$ 
can be obtained by  successive blowups from ${\rm LG}(2n, W_1 \oplus W_2)$. More precisely, there is a chain of ${\rm Sp}(W_1)\times {\rm Sp}(W_2)$-equivariant birational morphism with connected fibers as follows:
\begin{eqnarray}\label{e.chainM}
\overline{{\rm Sp}(W)}\cong M_0\xrightarrow{\pi_0} M_1\xrightarrow{\phi_1}M_2\xrightarrow{\phi_2}\cdots\xrightarrow{\phi_{n-2}}M_{n-1} \cong {\rm LG}(2n, W_1 \oplus W_2).
\end{eqnarray}
%The morphism $\phi_{n-2}$ is the blowup of $M_{n-1}$ along ${\rm LG}(n)$. 
The morphism $\phi_{k-1}$ is the blowup of $M_k$ along the strict transform $N_k \subset M_k$ of the smooth projective subvariety $\overline{{\LG}(k+1)}\subset{\rm LG}(2n, W_1\oplus W_2)$. In particular, $M_0,M_1,\ldots,M_{n-1}$ and $N_1, \cdots, N_{n-1}$ are all smooth projective ${\rm Sp}(W_1)\times {\rm Sp}(W_2)$-spherical varieties. %The construction above is due to Huruguen, see \cite[Section 2]{Hu} for more details. 

Since the orbit closures $\overline{{\LG}(1)},\ldots, \overline{{\LG}(n)}$ are all $\tau$-stable, the $\tau$-action  naturally extends to $M_k$ and $N_k$, and the chain \eqref{e.chainM} is $\tau$-equivariant.  Let $p_k: M_k\rightarrow Y'_k=M_k/\tau$ be the quotient map. In particular, $Y'_k$ is a normal variety with the $K$-action.  The quotient of the chain \eqref{e.chainM} by the $\tau$-action gives a chain of birational morphisms:
\begin{eqnarray}\label{e.chainY'}
Y'_0\xrightarrow{\pi'_0} Y'_1\xrightarrow{\pi'_1}Y'_2\xrightarrow{\pi'_2}\cdots\xrightarrow{\pi'_{n-2}}Y'_{n-1} \cong Z.
\end{eqnarray}
As  quotients by a finite group commute with blowups (see for example \cite[Lemma 3.11]{Ki85}), 
the morphism $\pi'_{k-1}: Y'_{k-1}\rightarrow Y'_{k}$ is in fact the blowup along the quotient $N_{k}/\tau$, which is nothing else but the  the strict transform $\widetilde{Z}'_{k+1}\subset Y'_{k}$ of  $Z_{k+1}\subset Z$ for $k=1, \cdots, n-1$.
This shows that the chain \eqref{e.chainY'} is  a chain of morphisms of spherical $K/\diag({\rm PSp}(W))$-embeddings.

We claim that the two chains \eqref{e.chainY} and \eqref{e.chainY'} coincide. By considering the exceptional loci on $Y_{n-1}=Y'_{n-1}$, one deduces that $\pi_{n-2}$ factors through $\pi'_{n-2}$. The morphisms $Y_{n-2}\rightarrow Y'_{n-2}\rightarrow Y_{n-1}$ induce maps of colored fans $\mbF_{n-2}\rightarrow\mbF_{Y'_{n-2}}\rightarrow\mbF_{n-1}$. There is no proper intermediate colored cone between $(\mfC_{D_n}, \emptyset)$ and $(\mfC_{\widetilde{Z}_n}, \mfD_{\widetilde{Z}_n})$, where $\mfC_{D_n}=\cone\{-\omega^\vee_n\}$, $\mfC_{\widetilde{Z}_n}=\mfC''_{n-1}$, and $\mfD_{\widetilde{Z}_n}=\mfD'_{n-1}$. Hence $\mbF_{n-2}=\mbF_{Y'_{n-2}}$, $Y_{n-2}=Y'_{n-2}$ and $\pi_{n-2}$ coincides with $\pi'_{n-2}$. By an induction on $k$. we see  that $\pi_k$ coincides with $\pi'_k$ for all $k$.

(4) The restriction of $\psi' \circ (\phi')^{-1}$ on the open orbit is the surjective group morphism ${\rm Sp}(W)\rightarrow {\rm PSp}(W)$. The  VMRT on open orbits of ${\rm LG}(2n, W_1 \oplus W_2)$ and $X$ coincides with the minimal nilpotent orbit $\BP \mathscr{O} \subset \BP \mfg$, where $\mfg$ is the Lie algebra of both ${\rm Sp}(W)$ and ${\rm PSp}(W)$. This completes the proof.
\end{proof}

\subsection{Rigidity of $\bar{C}_n$}

Let $G/P$ be an IHSS (cf. Example \ref{e.IHSS}). By Example \ref{e.IHSS2}, there exists a decomposition $\fg = \fg_{-1} \oplus \fg_0 \oplus \fg_1$.
Let  $G_0 \subset G$ be the connected algebraic subgroup  corresponding to $\fg_0$. Then $G/P$ carries a locally flat $G_0$-structure $\mcG_o$ which is given by the VMRT structure.
The following result is collected from the proof of Case (A) of  \cite[Proposition 2.3]{Mok08}.
\begin{prop} \label{p.RigidityIHSSM08}
Let $G/P$ be an IHSS of rank $\geq 2$, and let $\mcV$ be a holomorphic vector bundle over $\Delta$. Let $\mcS\subset\mbP \mcV$ be an irreducible closed subvareity such that  $\mathcal{S}_t \simeq G/P$ for $t \neq 0$. Suppose there exists a proper closed algebraic subset $N\subsetneq\mcS_0$ and a $G_0$-structure $\mcG$ on $\mcS \setminus N$  such that $\mcS\setminus N$ is a regular family over $\Delta$, and for $t \neq 0$, the $G_0$-structure $\mcG|_{\mcS_t}$ is the natural $G_0$-structure $\mcG_o$  on $G/P$.

(i) Take any point $x\in\mcS_0\setminus N$. Then after shrinking $\Delta$ to a smaller disk containing $0$ (denoted again by $\Delta$), there is a meromorphic map $h: G/P\times\Delta\dashrightarrow\mcS$ preserving the $G_0$-structure such that for $t\neq 0$, $h_t$ is a biholomorphic map, and the image of $h_0$ contains an analytic open neighborhood of $x\in\mcS_0$.

(ii) Suppose moreover that for $t\neq 0$, the embedding $\mcS_t\subset\mbP \mcV_t$ is given  by the complete linear system $|\mcO_{G/P}(k)|$ for some $k\geq 1$, and that the germ of $\mcS_0$ at $x$ is linearly non-degenerate in $\mbP \mcV_0$. Then the map $h$ in (i) is biholomorphic. In particular, $\mcS_0\simeq G/P$ and the inclusion $\mcS_0\subset\mbP \mcV_0$ is induced by the complete linear system $|\mcO_{G/P}(k)|$.
\end{prop}
\begin{proof}
For the convenience of readers, we reproduce  the proof given in \cite{Mok08}.
Recall from Remark \ref{r.LocFlatGStructure}, the local flatness of a $G_0$-structure can be detected by the vanishing of certain vector-valued function $c^k, k=0, 1,2, \cdots$ introduced by Guillemin \cite{Gu}. Furthermore,   these functions $c^k$ vanish for the $G_0$-structure on $G/P$. As the function $c^k$ is continuous, it vanishes on the whole $\mcS \setminus N$, which implies that the $G_0$-structure $\mcG$ on $\mcS \setminus N$ is locally flat. Then there is an analytic open neighbohood $U$ of $x$ in $\mcS$ and a holomorphic map $f: U\to G/P$ whose restriction $f_t: U \cap\mcS_t\to G/P$ is a developing map (i.e. $f_t$ preserves $G_0$-structures) whenever $U \cap\mcS_t\neq\emptyset$.  If necessary, shrink $\Delta$ to a smaller disk containing $0$ (denoted as $\Delta$ again). Then there is a holomorphic map $h: G/P\times \Delta^*\to\mbP \mcV$ respecting canonical projections such that $h$ induces a biholomorphic map from $G/P$ onto $\mcS_t\subset\mbP \mcV_t$ for $t\in\Delta^*:=\Delta\setminus\{0\}$, and such that $h$ extends holomorphically across some nonempty subset $U_0$ of $G/P\times\{0\}$  such that  $h(U_0)$ contains an open neighborhood $U_x$ of $x$ in $\mcS_0$. By Hartogs extension of meromorphic functions, $h$ extends to a meromorphic map from $G/P\times\Delta$ into $\mbP \mcV$ respecting canonical projections, to be denoted by the same symbol $h$. Since $h(G/P\times\Delta^*)$ is contained in $\mcS$, so is $h(G/P\times\{0\})$, verifying (i).
The assertion (ii) follows from \cite[Lemma 2.3]{Mok08} immediately.
%Under the assumption of (ii), $h_t^*\mcO_{\mbP(\mcV_t)}(1)=\mcO_{G/P}(k)$ for $t\neq 0$. By continuity, the same holds for $t=0$. Since the image of $h_0$ contains the germ of $x$ in $\mcS_0$, the linearly non-degenerate assumption implies that the linear subsystem of $|\mcO_{G/P}(k)|$ is of dimension $\dim\mcV_0-1$. For $t\neq 0$, the assumption $\mcS_t\subset\mbP(\mcV_t)$ implies that $\dim|\mcO_{G/P}(k)|=\dim\mcV_t-1=\dim\mcV_0-1$. Then $h_0: G/P\to\mcS_0\subset\mbP(\mcV_0)$ is a closed embedding induced by the complete linear system $|\mcO_{G/P}(k)|$. This completes the proof of (ii).
%
%Then $\bar{M}\subset\langle \bar{M}\rangle$ is a linear projection of $\nu_k(G/P)\subset \mbP^{r-1}:=\mbP(H^0(G/P, \mcO(k))^\vee$, where $M$ is the image of $h_0$ in $\mbP(\mcV_0)$, $\langle \bar{M}\rangle$ is the linear span of $\bar{M}$, and $r:=h^0(G/P, \mcO(k))$ is equal to the rank of the vector bundle $\mcV$ over $\Delta$. Recall that $M$ contains the neighborhood $U_x$ of $x$ in $\mcS_0$. Since $U_x$ is linearly non-degenerate in $\mbP(\mcV_0)$, so is $M$ (as well as its closure $\bar{M}$).  Then the linear projection $\mbP^{r-1}\dashrightarrow \langle \bar{M}\rangle$ is an isomorphism. Then $\bar{M}\subset\mbP(\mcV_0)$ is projectively equivalent to $\nu_k(G/P)\subset\mbP^{r-1}$, verifying (ii).
\end{proof}

As a direct consequence of Proposition \ref{p.RigidityIHSSM08}, we have the following result.
\begin{prop} \label{p.RigidityIHSS}
Let $G/P$ be an IHSS of rank $\geq 2$.
Let $\pi: \mathcal{S} \to \Delta$ be a projective flat map such that $\mathcal{S}_t \simeq G/P$ for $t \neq 0$ and the central fiber $\mathcal{S}_0$ is irreducible and reduced.
Assume that  there exists a proper closed algebraic subset $N \subsetneq \mcS_0$ and a $G_0$-structure $\mcG$ on $\mcS \setminus N$ such that for $t \neq 0$, the $G_0$-structure $\mcG|_{\mcS_t}$ is the natural $G_0$-structure $\mcG_o$  on $G/P$.  Then $\mcS_0$ is isomorphic to $G/P$.
\end{prop}
\begin{proof}
Take a holomorphic very ample line bundle $\mcL$ on $\mcS/\Delta$.  This induces a closed immersion $ \mcS\subset\mbP^k_\Delta$ over $\Delta$ such that $\mcO_{\mbP^k_\Delta}(1)|_{\mcS}=\mcL$. Then the  Hilbert polynomial  of $\mcS_t$ satisfies that $P_{\mcS_t}(m)=h^0(\mcS_t, \mcO_{\mcS_t}(m))$ for all sufficiently large integers $m$. Since $\pi$ is a flat family, Hilbert polynomials are independent of $t\in\Delta$. Fix a sufficiently large integer $m$. The  line bundle $\mathcal{L}^{\otimes m}\in\Pic(\mcS/\Delta)$, which restricts to $\mcO_{G/P}(m)$ on general fibers, gives rise to a closed immersion $f:\mcS\to\mbP^r_\Delta$ over $\Delta$, where $r:=P_{\mcS_t}(m)$ is independent of $t\in\Delta$. For each $t$, the morphism $f_t: \mcS_t\to\mbP^r=\mbP(H^0(\mcS_t, \mcO_{\mcS_t}(m))^\vee)$ is the morphism induced from the $r$-dimensional complete linear system $|\mcO_{\mcS_t}(m)|$ on the irreducible and reduced variety $\mcS_t$. In particular, $f_t(\mcS_t)$ is a linearly non-degenerate closed subvariety of $\mbP^r$ for each $t$. When we consider the closed immersion $\mcS\subset\mbP^r_\Delta$ induced by $\mcO(m)$, all the assumptions in Proposition \ref{p.RigidityIHSSM08}(ii) are fulfilled, implying that $\mcS_0\simeq G/P$.
\end{proof}

\begin{prop}\label{p.DR}
Let $X$ be the wonderful compactification of an adjoint simple group $G$ of rank $n$.
Let $\pi: \mcX \to \Delta$ be a regular family of projective manifolds such that $\mcX_t \simeq X$ for all $t \neq 0$.
 Assume  (i) the $G \times G$-action on $\mcX_t$ extends to a $G\times G$-action on $\mcX$;  (ii)  there exists a $G \times G$-stable open subset $\mcU \subset \mcX$ such that the surjective map
$\pi|_\mcU: \mcU \to \Delta$ is a $(G\times G)/\diag(G)$-fibration. Then $\mcX_0 \simeq X$.
\end{prop}
\begin{proof}
By assumption,  $\mcX_0$ is also a $G \times G$-spherical variety. Set $\partial\mcX_t:=\mcX_t\setminus\mcU_t$ for each $t\in\Delta$. By Proposition \ref{p.BasicsSV} (v), we have  $\Pic(\mcX_t)=H^2(\mcX_t, \mbZ)$ for each $t\in\Delta$ . Then the invariance of cohomology under deformations implies that $\Pic(\mcX_0)\simeq \Pic(X)$. Since the open orbit $(G\times G)/\diag(G)$ is affine, the boundary $\partial\mcX_0$ is the union of $n$ prime divisors $E_1,\ldots,E_n$.

Let $\mcD_i$ (resp. $\mcD(\omega_j)$) be the irreducible divisor on $\mcX$ which restricts to $D_i$ (resp. $D(\omega_j)$) on $\mcX_t$ for $t\neq 0$. Since $\mcU/\Delta$ is a trivial fibration, so is $(\mcD(\omega_j)\cap\mcU)/\Delta$. Then by Example \ref{e.colored cones}, $\varrho(\mcD(\omega_j)|_{\mcX_0})=\varrho(D(\omega_j))=\alpha_j^\vee\in\Lambda_\mbQ^\vee(G)$. We have $\partial\mcX_t=-K_{\mcX_t}-2\sum_{j=1}^n\mcD(\omega_j)|_{\mcX_t}$ for each $t\in\Delta$ by Proposition \ref{p.BasicsSV}(vi). The right hand side of this formula is invariant in $\Pic(\mcX/\Delta)\simeq\Pic(\mcX_t)$, then so is the left hand side. Combining with the fact $\sum_{i=1}^n\mcD_i|_{\mcX_t}=\partial\mcX_t\in\Pic(\mcX_t)$ for $t\neq 0$, we have 
\begin{eqnarray}\label{e.boundivX0}
\sum_{i=1}^n\mcD_i|_{\mcX_0}=\partial\mcX_0=\sum_{k=1}^nE_k\in\Pic(\mcX_0).
\end{eqnarray}
The Weil divisor $\mcD_i|_{\mcX_0}$ is represented by the scheme-theoretic intersection of $\mcD_i$ with $\mathcal{X}_0$, and it is supported at the boundary of $\mathcal{X}_0$. Thus $\mcD_i|_{\mcX_0}=\sum_{k=1}^nc_{ik}E_k$ for some nonnegative integers $c_{ik}$ with $\sum_{k=1}^nc_{ik}\geq 1$. Then the formula \eqref{e.boundivX0} implies that $\mcD_i|_{\mcX_0}=E_i$, up to reordering $E_1,\ldots,E_n$. It follows that $\varrho(E_i)=\varrho(D_i)=-\omega_i^\vee\in\Lambda_\mbQ^\vee(G)$ for each $i$.

Let $\mcM$ be the irreducible closed subvariety on $\mcX$ which restricts to the closed orbit $D_{1,\ldots,n}:=\cap_{i=1}^n D_i$ on $\mcX_t$ for $t\neq 0$. Take a closed orbit $F$ that is contained in $\mcM_0$. Since $\mcM_t\subset\mcD_i|_{\mcX_t}$ for $t\neq 0$, we have $F\subset \mcD_i|_{\mcX_0}=E_i$ by continuity.
Set $\mcX'_0:=\mcX_0\setminus(\cup Y)$, where $Y$ runs over the set of orbits on $\mcX_0$ such that $F \nsubseteq\overline{Y}$. Then the colored fan $\mbF_{\mcX'_0}$ is exactly the set of colored faces of the colored cone $(\mfC_F, \mfD_F)$ of $F$, and the cone $\mfC_F$ is generated by $-\omega^\vee_1,\ldots,-\omega^\vee_n$, and $\varrho(\mfD_F)$. Since $-\omega^\vee_1,\ldots,-\omega^\vee_n$ form a basis of $\Lambda_\mbQ^\vee(G)$, we have $\mfD_F=\emptyset$ by Proposition \ref{p.BasicsSV}(ii).
It follows that $(\mfC_F, \mfD_F)=(\mcV(G), \emptyset)=(\mfC_{D_{1,\ldots,n}}, \mfD_{D_{1,\ldots,n}})$, and thus the colored fans satisfy  $\mbF_{\mcX'_0}=\mbF_X$ by Example \ref{e.colored cones}. This gives $\mcX'_0\cong X$, implying $\mcX_0\cong X$.
\end{proof}

\begin{thm} \label{t.RigidityTypeC}
Let $(W, \omega)$ be a symplectic vector space.
Let $\pi: \mcX\rightarrow\Delta\ni 0$ be a  regular family of Fano varieties such that $\mcX_t  \cong X$ for $t \neq 0$,  where $X$ is the wonderful compactification of ${\rm {\rm PSp}(W)}$.
  Then $\mcX_0 \cong X$.
\end{thm}
\begin{proof}
By Theorem \ref{t.InvCones}, the contraction $\phi: X \to Z$ constructed by Proposition \ref{p.LGEtendsion} extends to a contraction (over $\Delta$) $\Phi: \mcX \to \mcZ$ such that for $t \neq 0$, $\Phi_t$ coincides with $\phi$. The central morphism $\Phi_0: \mcX_0 \to \mcZ_0$ is a Mori contraction, and $\mcZ_0$ is a normal projective variety.
Let $\mcD$ be the irreducible divisor on $\mcX$ which restricts to $D_n$ on $\mcX_t$ for $t \neq 0$ and let $\mcD_0 \subset \mcX_0$ be the central fiber of $\pi|_{\mcD}: \mcD \to \Delta. $

We first show that the central fiber $\mcD_0$ is a prime divisor on $\mcX_0$.  Otherwise we can write $\mcD_0 = E_1 + E_2$ for two non-zero effective Weil divisors on $\mcX_0$. By Example \ref{e.cones},  $D_n$ generates an extremal ray of ${\rm PEff}(X)$, then so does $\mcD_0$ by Theorem \ref{t.InvCones}(ii). This implies that $E_1, E_2$ and $\mcD_0$ are proportional in ${\rm Pic}(\mcX_0)$.   As $\mcD_0$ corresponds to $\alpha_n \in \Lambda \simeq {\rm Pic}(\mcX_0)$ and $\frac{1}{2}\alpha_n = \omega_{n} - \omega_{n-1} $ is a primitive lattice vector  in $\Lambda$, we have
$E_1 = E_2 = \frac{1}{2}\alpha_n \in \Lambda$.  This implies that $h^0(\mcX_0, \mathcal{O}(\mcD_0)) \geq 2$, while $h^0(X, \mathcal{O}(D_n)) =1$ by \cite[Theorem 2.2.3]{Br07}. This contradicts Theorem
 \ref{t.InvCones}(iii), hence $\mcD_0$ is a prime divisor on $\mcX_0$.

Let $\Psi': \mcY \to \mcX$ be the double cover of $\mcX$ ramified along $\mcD$, then by Proposition \ref{p.LGtoWC}, $\mcY_t \simeq Y$ for $t \neq 0$.
The central fiber $\mcY_0$ is the double cover of the smooth variety $\mcX_0$ ramified along the prime divisor $\mcD_0$, hence $\mcY_0$ is irreducible and reduced.
Take the Stein factorization of $\mcY \to \mcZ$ as $\mcY \xrightarrow{\Phi'} \mcZ' \xrightarrow{\Psi} \mcZ$.  Then for $t \neq 0$, $\mcZ'_t \simeq {\rm LG}(2n, W_1 \oplus W_2)$ and  $\Phi'_t$  (resp. $\Psi'_t$) coincides with $\phi'$ (resp. $\psi'$).
Since $\mcY_0$ is irreducible and reduced, so is $\mcZ'_0$.
As $\Phi'$ has connected fibers and $\mcZ'_0$ has the same dimension as $\mcY_0$, the map $\Phi'_0: \mcY_0 \to \mcZ'_0$ is also birational.

We first show that $\mcZ'_0$ is isomorphic to ${\rm LG}(2n, W_1 \oplus W_2)$. The VMRT of ${\rm LG}(2n, W_1 \oplus W_2)$ is the second Veronese embedding $\nu_2: \mathbb{P}W \to \mathbb{P} {\rm Sym}^2 W$, which induces a $G_0$-structure on ${\rm LG}(2n, W_1 \oplus W_2)$ with $G_0$ being the automorphism group of the affine cone of $\nu_2(\mathbb{P}W)$. Let $\mcC\subset\mbP(T_{\mcX/\Delta})$ be the VMRT-structure on $\mcX$, and denote by $\mcC':=h^*\mcC\subset\mbP(T_{\mcZ'/\Delta})$ the pull-back by $h:=\Psi' \circ (\Phi')^{-1}: \mcZ' \dasharrow \mcX$. For each $t$, the rational map $h_t:\mcZ'_t\dasharrow\mcX_t$ is generically finite, and  by Proposition \ref{p.invVMRT} the VMRT at a general point of $\mcX_t$ is projectively equivalent to that of ${\rm LG}(2n, W_1 \oplus W_2)$, which implies that $\mcC'_z\subset\mbP(T_z\mcZ'_t)$ at a general point $z\in\mcZ'_t$ is projectively equivalent to the VMRT of ${\rm LG}(2n, W_1 \oplus W_2)$. It induces a $G_0$-structure $\mcG^\circ\to\mcZ'$ whose domain of definition  contains an open subset of $\mcZ'_t$ for each $t$. By Proposition \ref{p.LGtoWC}, $\mcC'|_{\mcZ'_t}$ is itself the VMRT structure on $\mcZ'_t$ for $t\neq 0$. Hence the $G_0$-structure $\mcG^\circ$ is well-defined on $\mcZ'|_{\Delta^*}$, and the indeterminacy locus of $\mcG^\circ$ is a proper closed subvariety of $\mcZ'_0$. Recall that the central fiber $\mcZ'_0$ is  irreducible and reduced,
then $\mcZ'_0$ is isomorphic to ${\rm LG}(2n, W_1 \oplus W_2)$ by Proposition \ref{p.RigidityIHSS}.

Up to shrinking $\Delta$, we may assume that $\mcZ'$ is the trivial family $\LG(2n, V) \times \Delta$, where $V$ is a symplectic vector space of dimension $4n$ (and thus isomorphic to $W_1\oplus W_2$).
As $\Psi: \mcZ' \to \mcZ$ is a degree 2 map, it induces an involution of  $\mcZ'$ and $\mcZ$ is the quotient.  The involution is given by $\iota: \Delta \to {\rm Aut}(\LG(2n, V)) \simeq {\rm PSp}(V)$. As $\Sp(V) \to {\rm PSp}(V)$ is \'etale, we can choose a lift $\tilde{\iota}: \Delta \to \Sp(V)$. For $t\neq 0$, the morphism $\Psi_t$ is the quotient map of ${\rm LG}(2n, W_1 \oplus W_2)$ given by $\tau = (1_{W_1}, -1_{W_2})\in\Sp(W_1\oplus W_2)$. In other words, there exists a symplectomorphism $\theta(t): V\to W_1\oplus W_2$ such that $\tilde{\iota}(t) = \theta(t)^{-1} \circ \tau \circ \theta(t)$. Then for each $t\neq 0$, $V$ is a direct sum of $2n$-dimensional eigenspaces of $\tilde{\iota}(t)$ with eigenvalues $1$ and $-1$ respectively. By continuity, the same holds for $t=0$. The latter gives a symplectomorphism $\theta(0): V\to W_1\oplus W_2$ such that $\tilde{\iota}(0) = \theta(0)^{-1} \circ \tau \circ \theta(0)$.
This implies that $\mcZ_0$ is also isomorphic to $Z$, the quotient of ${\rm LG}(2n, W_1 \oplus W_2)$ by $\tau$.  In particular, $\mcZ$ is normal.   As $\Phi: \mcX \to \mcZ$ is birational, $\Phi$ has connected fibers, hence $\Phi_0: \mcX_0 \to \mcZ_0$ is also birational.

Recall that $K = {\rm PSp}(W_1) \times {\rm PSp}(W_2)$ acts on $X$ and on $Z$ such that $\phi: X \to Z$ is $K$-equivariant. This induces $K$-actions on $\mcZ$ and on $\mcX^\circ:=\mcX|_{\Delta^*}$.  Let $\mathfrak{k}$ be the Lie algebra of $K$. Each $v\in \mathfrak{k}$ induces a vector field  on $\mcX^\circ$, which descends to the vector field on $\mcZ=Z\times\Delta$.  As $\Phi_0: \mcX_0 \to \mcZ_0$ is birational, this vector field extends to an open subset of $\mcX_0$.  By Hartogs extension theorem, we get a vector field on $\mcX$, hence on $\mcX_0$. This shows that the universal covering $\tilde{K}$ of $K$ acts on $\mcX_0$. As the center of $\tilde{K}$ acts trivially on $\mcX_t$, we get a $K$-action on $\mcX_0$ whose stabilizer at a general point contains the diagonal subgroup $\diag(G)$ of $K$ as a subgroup of finite index. By descending the $K$-action onto $\mcZ_0=Z$, the index is one. In particular, the open $K$-orbit in $\mcX_0$ is isomorphic to $(G \times G)/\diag(G)$.  By Proposition \ref{p.DR}, we have $\mcX_0 \simeq X$, concluding the proof.
\end{proof}

\section{Invariance of varieties of minimal rational tangents for $B_3$}

\subsection{A construction of the wonderful compactification in type  $B_n$}
In this subsection, we will give a construction of the wonderful compactification of $B_n$ by successive blowups from a spinor variety.

Let $(W, o)$ be a vector space of dimension $2n+1$ endowed with a non-degenerate symmetric quadric form $o$.  Take two copies $(W_i, o_i), i=1,2,$ of $(W, o)$ and endow the vector space $W_1 \oplus W_2$ the quadratic form given by $p_1^* o_1 - p_2^* o_2$, where $p_i: W_1 \oplus W_2 \to W_i$ is the natural $i$-th  projection.
Let $W'=\{(w,w) | w \in W\}  \subset W_1 \oplus W_2$, which is an isotropic vector space of dimension $2n+1$.
There are two families of isotropic subspaces of dimension $2n+1$ in $W_1 \oplus W_2$, each of which is parameterized by a spinor variety.
Let ${\rm OG}(2n+1, W_1 \oplus W_2)$ be the irreducible component containing $[W']$.

% which parameterizes isotropic subspaces of dimension $2n+1$ in $W_1 \oplus W_2$.  %Note that $\dim {\rm LG}(2n, W_1 \oplus W_2) = n (2n+1)$, which is the same as $\dim \Sp(2n)$.
 As the construction is similar to that in Section \ref{s.C}, we will only give the corresponding statements while omitting the proofs.
The following is analogous to Lemma \ref{l.equaldim}.
\begin{lem} \label{l.equaldimSO}
Let $[V] \in {\rm OG}(2n+1, W_1 \oplus W_2)$, then $\dim V \cap W_1 = \dim V \cap W_2$.
\end{lem}

For any $0 \leq k \leq n$, denote by $\OG(k, W_i)$ the $k$-th isotropic  Grassmannian in $W_i$, which parameterizes $k$-dimensional orthogonal linear subspaces in $W_i$.
Define
$$\OG(k):=\{ [V] \in {\rm OG}(2n+1, W_1 \oplus W_2) | \dim V \cap W_1 = \dim V \cap W_2 = k\}.$$
Consider the map $\pi_k: \OG(k) \to \OG(k, W_1) \times \OG(k, W_2)$ defined by $[V] \mapsto ([V \cap W_1], [V \cap W_2])$, which is $\SO(W_1) \times \SO(W_2)$-equivariant.

Similar to Lemma \ref{l.OrbitLG(k)} and Lemma \ref{l.LGClosure}, we have the following:

\begin{lem}\label{l.OrbitOG(k)}
\begin{itemize}
%\item[(i)] The fiber of $\pi_k$ is isomorphic to $\Sp(\mathbb{C}^{2n-2k})$.
\item[(i)] Each $\OG(k)$ is an $\SO(W_1) \times \SO(W_2)$-orbit and ${\rm OG}(2n+1, W_1 \oplus W_2) = \bigsqcup_{k=0}^n \OG(k)$ is the decomposition of ${\rm OG}(2n+1, W_1 \oplus W_2)$ into  $\SO(W_1) \times \SO(W_2)$-orbits.
\item[(ii)] We have $\overline{\OG(k)} = \cup_{j \geq k} \OG(j)$.
\end{itemize}
\end{lem}

In  a similar way as Proposition \ref{p.LGSpherical}, we have
\begin{prop} \label{p.OG}
 The spinor variety  ${\rm OG}(2n+1, W_1 \oplus W_2)$ is spherical under the $\SO(W_1) \times \SO(W_2)$-action with the unique open orbit $\OG(0) \simeq \SO(W)$ and the unique closed orbit $\OG(n) \simeq \OG(n,W_1) \times \OG(n, W_2)$.
%(2) The variety $Z$ is spherical under the $K$-action, with unique open orbit $Z^\circ \simeq  {\rm PSp}(W)$ and  unique closed orbit $Z_n \simeq \LG(n,W_1) \times \LG(n, W_2)$.
\end{prop}

The following result is analogous  to Proposition \ref{p.LGEtendsion} and Proposition \ref{p.LGtoWC}.
\begin{prop} \label{p.OGtoSpinor}
Let $(W,o)$ be an orthogonal vector space of dimension $2n+1$ and
let $\bar{B}_n$ be the wonderful compactification of ${\rm SO}(W)$. Then

(i) the natural isomorphism $\phi^\circ:  {\rm SO}(W) \to \OG(0)$ extends to an $\SO(W) \times \SO(W)$-equivariant birational morphism $\phi: \bar{B}_n \to {\rm OG}(2n+1, W_1 \oplus W_2)$.

(ii) The morphism $\phi$  is the composition of successive blowups along the strict transforms of  $\overline{\OG(k)}$ from the smallest to the biggest.
\end{prop}

As in Example \ref{e.colored cones}, we denote by $D(\omega_1),\ldots,D(\omega_n)$ the colors on $\bar{B}_n$, and by   $D_1,\ldots,D_n$ the prime boundary divisors. Now $\phi$ gives an isomorphism between the open orbit on $\bar{B}_n$ and that on $\OG(2n+1, W_1\oplus W_2)$. To distinguish, let $D(\omega_i)^{\OG}$ be the image of $D(\omega_i)$, which is the corresponding color on ${\rm OG}(2n+1, W_1 \oplus W_2)$.
As an analogue of Lemma \ref{l.Exc} and Lemma \ref{l.colored fan Z}, we have the following result.

\begin{lem}\label{l.colored fan OGW}
(i) The morphism $\phi$ sends the open orbit $D_1^o$ of $D_1$ isomorphically to $\OG(1)$, and the exceptional locus of $\phi$ is $\cup_{i=2}^n D_i$.

(ii) The dual weight lattice  of $\OG(0)$ is $\Lambda^\vee(\OG(0))=\oplus_{i=1}^{n}\mbZ\omega^\vee_i$, and the valuation cone is $\mcV(\OG(0))=\cone\{-\omega^\vee_1,\ldots,-\omega^\vee_n\}$.  The colored cone $(\mfC_{\OG(k)}, \mfD_{\OG(k)})$ of $\OG(k)$ is given by $\mfC_{\OG(k)}=\cone\{-\omega_1^\vee, \alpha^\vee_1,\ldots,\alpha^\vee_{k-1}\}$ and $\mfD_{\OG(k)}= \{D(\omega_1)^{\OG},\ldots, D(\omega_{k-1})^{\OG}\}$.
\end{lem}

\begin{cor} \label{c.OGD}
We have $\phi^*(D(\omega_n)^{\OG})=D(\omega_n)$.
\end{cor}

\begin{proof}
The description of colored cones in Lemma \ref{l.colored fan OGW} implies that there is no $\SO(W_1)\times \SO(W_2)$-orbit $\OG(k)$ contained in the divisor $D(\omega_n)^{\OG}$. Since the morphism $\phi$ is the composition of successive blowups along the strict transforms of  $\overline{\OG(k)}$, the pulling back of $D(\omega_n)^{\OG}$ is equal to its strict transform, verifying that $\phi^*(D(\omega_n)^{\OG})=D(\omega_n)$.
\end{proof}

\begin{lem} \label{l.OGD}
Under the natural embedding ${\rm OG}(2n+1, W_1 \oplus W_2)\subset\mbP^{2^{2n}-1}$ given by the half-spin representation, we have $\overline{\OG(1)}=\mcO(2)$, $D(\omega_n)^{\OG}=\mcO(1)$, and $D(\omega_j)^{\OG}=\mcO(2)$ for $j\neq n$.
\end{lem}

\begin{proof}
The unique prime boundary divisor of ${\rm OG}(2n+1, W_1 \oplus W_2)$ is $\overline{\OG(1)}$,  and the colors are $D(\omega_1)^{\OG},\ldots,D(\omega_n)^{\OG}$. By Proposition \ref{p.BasicsSV}(iv), these divisors generate the Picard group of ${\rm OG}(2n+1, W_1 \oplus W_2)$, and the rational equivalence relations  among them are defined by elements in the weight lattice $\Lambda(\OG(0))$.
In the dual lattice $\Lambda^\vee(\OG(0))=\oplus_{k=1}^n\mbZ\omega_k^\vee$, we have $\varrho(\overline{\OG(1)})=-\omega_1^\vee$, and $\varrho(D(\omega_j)^{\OG})=\alpha_j^\vee$ for $j=1,\ldots,n$.  The rational equivalence relation in the Picard group of ${\rm OG}(2n+1, W_1 \oplus W_2)$ defined  by the root $\alpha_k\in\Lambda(\OG(0))$ is  $-\delta_{1, k}\overline{\OG(1)}+\sum_{j=1}^n\langle \alpha_k, \alpha_j\rangle D(\omega_j)^{\OG}=0$, where $\langle, \rangle$ is the Cartan pairing. It follows that $\overline{\OG(1)}=D(\omega_j)^{\OG}=2D(\omega_n)^{\OG}$ for $j=1,\ldots,n-1$. Since the Picard group is generated by these divisors, we have  $D(\omega_n)^{\OG}=\mcO(1)$, implying the conclusion.
\end{proof}

\begin{prop} \label{p.B3PreservesLines}
The morphism $\phi$ sends minimal rational curves on $\bar{B}_n$ to lines on ${\rm OG}(2n+1, W_1 \oplus W_2)$.
\end{prop}

\begin{proof}
Under the identification of the Picard group of $\bar{B}_n$ with the weight lattice of $B_n$, we have $D_i=\alpha_i$ and $D(\omega_j)=\omega_j$ for each $i$ and $ j$. It follows that $D(\omega_n)=\frac{1}{2}\sum_{k=1}^nkD_k$ in the Picard group of $\bar{B}_n$.
Let $C$ be a minimal rational curve on $\bar{B}_n$. By Lemma 3.4 of \cite{BF15}, $D_2\cdot C=1$ and $D_j\cdot C=0$ for $j\neq 2$, hence $D(\omega_n) \cdot C=1$.  This gives $D(\omega_n)^{\OG}\cdot \phi_*(C)=\phi^*D(\omega_n)^{\OG}\cdot C=D(\omega_n)\cdot C=1$, implying that $\phi_*(C)$ is a line on ${\rm OG}(2n+1, W_1 \oplus W_2)$.
\end{proof}

\subsection{Minimal rational curves in the limit fiber}

Let $\mathbb{S}$ be the spinor variety $\OG(7, \mathbb{C}^{14})$, which is a 21-dimensional Fano manifold of Picard number one. The VMRT of $\mathbb{S}$ is the Pl\"ucker embedding ${\rm Gr}(2, 7) \subset \mathbb{P}^{20}$.  By Lemma \ref{l.OrbitOG(k)},  $\mathbb{S} = \OG(0) \cup \OG(1) \cup \OG(2) \cup OG(3)$.  The strata $\OG(1)$ has codimension 1 while $\OG(2)$ has dimension 17.  The closed strata
$\OG(3)$ is isomorphic to $\mathbb{Q}^6 \times \mathbb{Q}^6$.

 By Proposition \ref{p.OGtoSpinor}, there exists a birational  contraction $\phi: \bar{B}_3 \to \mbS$ which is a successive blowup along $\OG(3)$ and the strict transform of  $\overline{\OG}(2)$.
 Let $D_1, D_2, D_3$ be the prime boundary divisors on $\bar{B}_3$ corresponding respectively to the simple roots $\alpha_1, \alpha_2, \alpha_3$.    Then $\phi|_{D_1}: D_1 \to \overline{\OG}(1)$ is a birational morphism, $\phi|_{D_2}:  D_2 \to \overline{\OG}(2)$ is a $\mathbb{P}^3$-bundle over $\OG(2)$, and $\phi|_{D_3}: D_3\to \OG(3)$ is a $\bar{A_2}$-bundle. The exceptional locus of $\phi$ is $D_2\cup D_3$.

Let $C_\theta$ be a general minimal rational curve on $\bar{B}_3$.  Let $C_i$ be irreducible  curves of $\bar{B}_3$ such that $C_i \cdot D(\omega_j) = \delta_{ij}$ for all $i, j$, where
$D(\omega_1), D(\omega_2),D(\omega_3)$ are the colors on $\bar{B}_3$.   Then the Mori cone $\overline{{\rm NE}}(\bar{B}_3)$ is the cone generated by $[C_1], [C_2], [C_3]$.
  By Lemma 3.4 of \cite{BF15}, $D_2\cdot C_\theta=1$ and $D_j\cdot C_\theta=0$ for $j\neq 2$. 
  Recall that on any wonderful compactification, the rational equivalence is the same as the numerical equivalence, which we will simply denote by $\sim$.
   This gives that
$C_\theta \sim C_1 + 2 C_2 + C_3$.

\begin{lem} \label{l.extremalOG}
The two rays $\mbR^+[C_1]$ and $\mbR^+[C_2]$ are extremal in the Mori cone $\overline{{\rm NE}}(\bar{B}_3)$ and $\phi$ is the Mori contraction associated with the extremal face $\mbR^+[C_1]+\mbR^+[C_2]$.
\end{lem}

\begin{proof}
The first assertion follows from Example \ref{e.cones}. To verify the second assertion, it suffices to show that $C_1$ and $C_2$ are contracted by $\phi$. The morphism $\phi|_{D_{2, 3}}: D_{2,3}\to\OG(3)$ factors through $\psi: D_{2, 3}:=D_2 \cap D_3 \to B_3/P_{\omega_2+\omega_3}^-\times {B_3/P_{\omega_2+\omega_3}}$, and the latter is a $\mbP^3$-fibration. The fiber of $\psi$ at the base point is the wonderful compactification  $\bar{A}_1$ whose minimal rational curves are rationally equivalent to $C_1$. Since $\psi$ contracts $C_1$, so is $\phi$. Similarly, by considering the morphism $\phi|_{D_{1, 3}}: D_{1,3}\to\OG(3)$ we can conclude that $\phi$ contracts $C_2$.
\end{proof}

\begin{prop} \label{p.intersectionOG}
Let $C \subset \bar{B}_3$ be a free irreducible rational curve.  Then $-K_{\bar{B}_3} \cdot C \geq 9$ and equality holds if and only if $C \sim C_\theta$.
Moreover:

(1) $-K_{\bar{B}_3} \cdot C \neq 10$ or $11$.

(2) If  $-K_{\bar{B}_3} \cdot C = 12$, then $C$ is rationally equivalent to $C_\theta + C_1$ and $\phi_*(C)$ is a line on $\mathbb{S}$.

\end{prop}
\begin{proof}
The first claim follows from Corollary \ref{c.minimal} as $-K_{\bar{B}_3} \cdot C_\theta=9$.  It remains to prove (1) and (2).
As $C$ is free, it can be deformed to an effective 1-cycle, one of whose irreducible reduced components is a minimal rational curve. By the uniqueness of minimal rational curves on $\bar{B}_3$, we can write
$C \sim C_\theta + \sum_{j=1}^3 k_j C_j$ with $k_j \geq 0$.  Note that the divisors $D_i$ and $D(\omega_j)$ are related by the Cartan matrix, which gives
$$
-K_{\bar{B}_3} = 2 (\sum_{i=1}^3 D(\omega_i)) + \sum_{j=1}^3 D_j  = 3 D(\omega_1) + 2 D(\omega_2) + 2 D(\omega_3).
$$

It follows that $-K_{\bar{B}_3} \cdot C = 9 + 3k_1 + 2k_2 + 2 k_3$.  As $k_j \geq 0$, we have $-K_{\bar{B}_3} \cdot C  \neq 10$.  Assume now  $-K_{\bar{B}_3} \cdot C  = 11$, then either
$C \sim C_\theta + C_2$ or $C \sim C_\theta + C_3$.
On the other hand, we have
$$
D_1 = -2D_2 - 3 D_3 + 2 D(\omega_3), \ D(\omega_1) = -D_2 - 2D_3 + 2 D(\omega_3), \ D(\omega_2) = -D_3 + 2 D(\omega_3),
$$
which gives  $-K_{\bar{B}_3}  =  12 D(\omega_3) - 3 D_2 - 8 D_3$.
If $C \sim  C_\theta + C_2 \sim  C_1 + 3C_2 + C_3$, then
$$-K_{\bar{B}_3}  \cdot C = 12 D(\omega_3) \cdot  (C_1 + 3C_2 + C_3) -  3 D_2 \cdot C - 8 D_3 \cdot C =12 - 3 D_2 \cdot C - 8 D_3 \cdot C,$$
which cannot be equal to $11$ as $D_2 \cdot C \geq 0$ and $D_3 \cdot C \geq 0$.

Now assume $C \sim  C_\theta + C_3 \sim  C_1 + 2C_2 + 2C_3$, then
$$-K_{\bar{B}_3}  \cdot C = 24 -  3 D_2 \cdot C - 8 D_3 \cdot C$$ which again cannot be equal to 11.   This proves (1).

If $-K_{\bar{B}_3} \cdot C = 12$, then we have $3k_1 + 2k_2+2k_3 = 3$, which implies that $k_1=1, k_2=k_3=0$, namely $C \sim C_\theta + C_1$.
As $\phi_* (C_1) =0$,  the pushforward $\phi_*(C) = \phi_*(C_\theta)$ is a line on $\mathbb{S}$ by Proposition \ref{p.B3PreservesLines}.

\end{proof}

\begin{cor} \label{c.stricttransform}
Let $\ell \subset \mathbb{S}$ be a line such that $\ell \cap \OG(0) \neq \emptyset$. Denote by $C \subset \bar{B}_3$ the strict transform of $\ell$ under $\phi$. Then $C$ is an irreducible reduced free rational curve on $\bar{B}_3$ with $\phi_*(C) = \ell$ and one of the following holds:

(i) $\ell$ is disjoint from $\overline{\OG}(2)$ and $C \sim C_\theta+C_1$ or

(ii) $\ell$ intersects  $\overline{\OG}(2)$ at a single reduced point and $\ell$ is disjoint from $\OG(3)$. In this case, $C \sim C_\theta$ is a minimal rational curve on $\bar{B}_3$.
\end{cor}
\begin{proof}
The first claims are immediate. Assume now $\ell \cap \overline{\OG}(2) =\emptyset$, then $C \cap D_2 = C \cap D_3 = \emptyset$.
As $\mathbb{S}$ has index 12, its canonical divisor satisfies $K_{\mathbb{S} }= \mathcal{O}(-12)$ hence $12 D(\omega_3) = \phi^* (-K_{\mathbb{S}})$ by Lemma \ref{l.OGD} and Corollary \ref{c.OGD}.
This gives $-K_{\bar{B}_3}  =  \phi^* (-K_{\mathbb{S}}) - 3 D_2 - 8 D_3$.  We deduce that
$$
-K_{\bar{B}_3} \cdot C =  \phi^* (-K_{\mathbb{S}}) \cdot C =  -K_{\mathbb{S}} \cdot \ell =12.
$$
By Proposition \ref{p.intersectionOG}, we have $C \sim C_\theta + C_1$.

Now assume $\ell \cap \overline{\OG}(2) \neq \emptyset$.  Then either $D_2 \cdot C \geq 1$ or $D_3 \cdot C \geq 1$.  This shows that $-K_{\bar{B}_3} \cdot C < 12$ while as $C$ is a free irreducible rational curve on $\bar{B}_3$, we have $-K_{\bar{B}_3} \cdot C \geq 9$.  By Proposition \ref{p.intersectionOG}, this implies $-K_{\bar{B}_3} \cdot C = 9$ and $C \sim C_\theta$. Moreover
$C \cdot D_2 =1$ and $C \cdot D_3 = 0$, which shows that $\ell$ intersects  $\overline{\OG}(2)$ at a single reduced point.
\end{proof}

Let $\pi: \mcX\rightarrow\Delta\ni 0$ be a  regular family of Fano varieties such that $\mcX_t  \cong \bar{B}_3$ for $t \neq 0$.   By Theorem \ref{t.InvCones}, the map $\phi$ extends to a morphism $\Phi: \mcX \to \mcS$, such that for $t \neq 0$, we have $\mcS_t \simeq \mbS$ and $\Phi_t = \phi$. After taking a suitable base change, we may assume that $\Pic(\mcS_0)=\Pic(\mcS/\Delta)=\Pic(\mbS)$ by Corollary 4.7 of \cite{dFH}. The central morphism $\Phi_0: \mcX_0 \to \mcS_0$ is also the Mori contraction associated with the extremal face $\mbR^+[C_1]+\mbR^+[C_2]$ of $\overline{NE}(\mcX_0)=\overline{NE}(\bar{B}_3)$, hence it is birational and $\mcS_0$ is normal. As $\mcS_t$ is covered by lines, so is $\mcS_0$. Later we will show that general lines are contained in the smooth locus of $\mcS_0$.

\begin{lem}
Let $\mcD_1^\mcX,  \mcD_2^\mcX, \mcD_3^\mcX$ be the boundary divisors.  Then $\Phi_0$ has exceptional locus $\mcD_{2, 0}^\mcX \cup \mcD_{3,0}^\mcX$.
Let $\mcD_i^\mcS = \Phi(\mcD_i^\mcX)$, then $\mcD_1^\mcS$ is a prime divisor on $\mcS$, while $\mcD_2^\mcS$ and $\mcD_3^\mcS$ are of relative dimension 17 and 12. Furthermore we have
$\mcD_3^\mcS \subset \mcD_2^\mcS\subset \mcD_1^\mcS$.
\end{lem}

\begin{proof}
Since the exceptional locus of $\phi$ is $D_2\cup D_3$, the semi-continuity of the dimension of fibers implies that $\mcD_{2, 0}^\mcX \cup \mcD_{3,0}^\mcX$ is contained in the exceptional locus $\Exc(\Phi_0)$. Any effective 1-cycle $C$ on $\Exc(\Phi_0)$ is identified with $\lambda_1 C_1+\lambda_2 C_2$ in $\overline{NE}(\mcX/\Delta)$ for some non-negative numbers $\lambda_1$ and $\lambda_2$.  By the numerical equivalence of $D_i$ and $D(\omega_j)$ given in Example \ref{e.cones},  we have $(D_i\cdot C_1, D_i\cdot C_2, D_i\cdot C_3)=(2, -1, 0), (-1, 2, -2), (0, -1, 2)$ for $i=1, 2, 3$ respectively.
This  implies that either $\mcD^\mcX_{2, 0}\cdot C<0$ or $\mcD^\mcX_{3, 0}\cdot C<0$. So $C$ is contained in $\mcD_{2, 0}^\mcX \cup \mcD_{3,0}^\mcX$, and the latter is the exceptional locus of $\Phi_0$. The rest is immediate from the information on $\phi$.
\end{proof}

\begin{prop} \label{p.smoothB3}
The variety $\mcS_0$ is smooth along general points of $\mcD_{2, 0}^\mcS$, and the singular locus of $\mcS_0$ is a proper subset of $\mcD_{2, 0}^\mcS$. In particular, a general minimal rational curve on $\mcX_0$ is sent isomorphically onto a line contained in the smooth locus of $\mcS_0$.
\end{prop}
\begin{proof}
Let $s \in  \mcD_{2, 0}^\mcS \setminus \mcD_{3, 0}^\mcS$ be a general point.
Let $\varsigma: \Delta \to \mcD_2^\mcS$ be a general section through $s$, which gives a holomorphic family $F_t:=\Phi^{-1}(\varsigma(t))$, $t\in\Delta$.

The restriction of $\phi$ on $D_2\setminus(D_2\cap D_3)$ is an $\SO(\C^7)\times\SO(\C^7)$-equivariant fiber bundle over the orbit $\OG(2)$, whose fiber at the base point is the wonderful compactification of $\PSL(\C^2)$. The latter group is the adjoint group of the subgroup of $\SO(\C^7)$ normalized by $T$ and  whose roots are $\pm\alpha_1$. Consequently,  $D_2\setminus(D_2\cap D_3) \to \OG(2)$ is a $\mbP^3$-bundle and each line on a fiber is of the class $[C_1]$ in the Mori cone $\overline{NE}(\mcX_0)=\overline{NE}(\bar{B}_3)$.

When $t\neq 0$, $F_t$ is identified with a fiber of $D_2\setminus(D_2\cap D_3) \to \OG(2)$. The $\Phi$-ample line bundle $\mcD:=\mcD(\omega_1)+\mcD(\omega_2)+\mcD(\omega_3)$ has degree one on each line of $F_t=\mbP^3$ with $t\neq 0$. Then the limits in $F_0$ of these lines are irreducible and reduced rational curves, which are also called lines by abuse of notions.
Take any two points $x, y \in F_0$, we can take  local sections $x_t, y_t$ such that $x_0=x, y_0=y$, then the limit of lines $\overline{x_t y_t} \subset F_t$ becomes a line on $F_0$ joining $x$ and $y$.
 This implies that the normalization map of $F_0$ is isomorphic to $\mbP^3$ by Theorem 3.6 of \cite{Kebekus}.

Suppose that $F_0$ is not  isomorphic to $\mbP^3$, then  $\dim H^0(F_0, \mcO(1))\leq 3$, where $\mcO(1):=\mcD|_{F_0},$ as the pull-back of $\mcD|_{F_0}$ on $\mbP^3$ is $\mcO(1)$. The limits of hyperplane sections of $F_t=\mbP^3$ as $t$ goes to zero give rise to elements of the linear system of $\mcO(1)$ on $F_0$. In particular, $\dim H^0(F_0, \mcO(1))\geq 1$. Consider the induced rational map $f: F_0\dashrightarrow \mbP^k$, where $k:=\dim H^0(F_0, \mcO(1))-1$. Let $\Gamma$ be (the resolution of) the graph of $f$ and let  $p_1$, $p_2$ be the projections to $F_0$ and $\mbP^k$ respectively. Since $k<3$, there is an irreducible curve $C$ on $\Gamma$ such that it is a strict transform of an irreducible curve on $F_0$ and $p_2(C)$ is a single point. It follows that $\deg\mcO(1)|_{pr_1(C)}=\deg pr_1^*\mcO(1)|_{C}=0$, contradicting the fact that $\mcO(1)$ is ample on $F_0$.  This shows that $F_0 \simeq \mathbb{P}^3$.

Now the fiber $\Phi^{-1}(s)=F_0$ at $s \in  \mcD_{2, 0}^\mcS$ is isomorphic to $\mbP^3$. Since $D_2\cap C_1=-1$, the restriction $\mcO_\mcX(\mcD_2)|_{F_0}=\mcO_{\mbP^3}(-1)$. It follows that $s$ is a smooth point of $\mcS_0$.
\end{proof}

We denote by $\mcK^\mcX$ (resp. $\mcK^\mcS$) the family of minimal rational curves on $\mcX$ (resp. $\mcS$) and denote by  $\mcU^\mcX$ (resp. $\mcU^\mcS$) the corresponding  universal family. The rest of this subsection is devoted to study the family $\mcK^\mcS$ and the associated VMRT $\mcC^\mcS$, which will be helpful to understand the VMRT $\mcC^\mcX$ associated to the family $\mcK^\mcX$. For this, we start from the strict transforms of elements in $\mcK^\mcS$ under $\Phi$.

Let $\mathcal{M}$ be the unique irreducible component of the Hilbert scheme ${\rm Hilb}(\mathcal{X}/\Delta)$ that is sent to $\mathcal{K}^{\mathcal{S}}$ by the induced map of $\Phi$.
%Let $\mcM$ be the unique irreducible component of the Chow scheme $\Chow(\mcX/\Delta)$ such that the Hilber-Chow morphism sends $\mcK^\mcS$ onto $\Phi_*\mcM$. 
In particular, general elements of $\mcM$ are the strict transforms of general lines on $\mcS$. Let $\mcM'$ be the open subset of $\mcM$ whose elements represent irreducible and reduced rational curves on $\mcX$.

Take $x_t\in\mcX_t$ general with $t\neq 0$. Then the elements of $\mcM'_{x_t}$ and $\mcM_{x_t}\backslash\mcM'_{x_t}$ are clearly described by Proposition \ref{p.intersectionOG} and Corollary \ref{c.stricttransform}. The next two lemmas show that the same description holds for $t=0$.

\begin{lem}\label{l.reduciblecurveB3}
Take a general point $x\in\mcX_0$ and any $[C]\in\mcM_x\setminus\mcM'_x$. Then we have:

(i) The 1-dimensional cycle $C$ consists of two irreducible components $C'$ and $C''$, where $C'$ is a minimal rational curve passing through $x$ and $C''$  is contained in $\mcD^\mcX_{2, 0}$.
Both $C'$ and $C''$ are reduced.  Moreover $C'$ intersects $\mcD^\mcX_{2, 0}$ at a single reduced point and $C'\cap\mcD^\mcX_{3, 0}=\emptyset$.  The component $C''$  has empty intersection with $\mcD^\mcX_{3, 0}$. In particular, $x\notin C''$.  Moreover, $[C'']=[C_1]$ in the Mori cone  $\overline{NE}(\mcX_0)=\overline{NE}(\bar{B}_3)$.

(ii) The intersection of $C'$ and $C''$ consists of a single reduced point.

(iii) The morphism $\Phi_0$ contracts $C''$ and sends $C'$ isomorphically onto a line on $\mcS_0$.
\end{lem}

\begin{proof}
By Corollary \ref{c.stricttransform}, elements in $\mcM_t$, $t\neq 0$, represent 1-cycles with rational components and of anti-canonical degree 12. Then the same holds for $t=0$. Let $C'$ be the irreducible component, with induced reduced structure, of $C$ passing through $x$. Then $C'$ is a free rational curve on $\mcX_0$ of anti-canonical degree less than 12. As $C'$ is free on the total space $\mcX$, we can deform it to a free rational curve $C'_t$ of the same anti-canonical degree on $\mcX_t$ with $t\neq 0$. By Proposition \ref{p.intersectionOG}, $C'_t$ is a minimal rational curve for each $t\neq 0$. Then by Proposition \ref{p.MRC}, $C'$ is a minimal rational curve on $\mcX_0$, verifying the conclusion (i).

By Corollary \ref{c.stricttransform} the class $[C]-[C']=[C_1]$. Since the anti-canonical degree of $[C]-[C']$ is less than 9, there is at least an irreducible component $C''$ of $C$ that is different from $C'$. Since $\mbR^+[C_1]$ is an extremal ray of $\overline{NE}(\mcX_0)=\overline{NE}(\bar{B}_3)$, see Example \ref{e.cones} and Theorem \ref{t.InvCones}, we know that $\mbR^+[C'']=\mbR^+[C_1]$. Since $[C_1]$ is a primitive lattice point in $H_2(\mcX_0, \mbZ)=H_2(\bar{B}_3, \mbZ)$ (alternatively, since the intersection number of $C_1$ with the ample Cartier divisor $D(\omega_1)+D(\omega_2)+D(\omega_3)\in\Pic(\bar{B_3})$ is 1), there is a positive integer $m$ such that $[C'']=m[C_1]$. Since the anti-canonical degree of $C$, $C'$ and $C_1$ are 12, 9 and 3 respectively, which shows that $m=1$.

Since $\mcD^\mcX_{2, 0}\cdot C''=D_2\cdot C_1=-1$, the curve $C''$ is contained in $\mcD^\mcX_{2, 0}$. As  $\mcD^\mcX_{3, 0}\cap C'=\emptyset$ and $C$ is connected, the curve $C''$ is not contained in $\mcD^\mcX_{3, 0}$. Moreover $\mcD^{3, 0}\cap C''=D_3\cdot C_1=0$, so we have $C''\cap\mcD^\mcX_{3, 0}=\emptyset$.  The rest is immediate.
\end{proof}

\begin{lem}\label{l.freecurveB3}
Take a general point $x\in\mcX_0$. Then $\mcM'_x\neq\emptyset$. Take any $[C]\in\mcM'_x$. Then the following holds.

(i) The 1-dimensional scheme $C$ is a free rational curve of anti-canonical degree 12.

(ii) The intersections $C\cap\mcD^\mcX_{2, 0}=\emptyset$ and $C\cap\mcD^\mcX_{3, 0}=\emptyset$. In particular, $\Phi_0$ sends $C$ isomorphically onto a line contained in the smooth locus of $\mcS_0$.
\end{lem}

\begin{proof}
It suffices to show $\mcM'_x\neq\emptyset$ and the rest is trivial. Recall that for $x\in B_3\subset\bar{B}_3$ and $s:=\phi(x)\in\OG(0)$, we have $\mcC_x=\OG(2, 7)\subsetneq\mcC_s=\Gr(2, 7)$ under the identification $\mbP T_x\bar{B}_3=\mbP T_s\mbS$ induced by $\phi$. Since $\dim\mcC_x<\dim\mcC_s$, we may choose a family $[C_t]\in\mcM$ such that $C_t\subset \mcX_t$ passing through general points on $\mcX_t$ and the projectivized tangent direction (of a branch) of the limit 1-cycle $C_0$ at $x$ is outside $\mcC^\mcX_x$. Then by Lemma \ref{l.reduciblecurveB3}, $C_0$ is irreducible and reduced.
\end{proof}

\begin{lem}\label{l.surjfamilyB3}
Let $x$ be a general point in $\mcX_0$ and denote by $s:=\Phi(x)\in\mcS_0$. Take any $[l]\in\mcK^\mcS_s$ there exists $[C]\in\mcK^\mcX_x$ such that $\Phi_*(C)=l$.
\end{lem}

\begin{proof}
Take a section $x_t\in\mcX_t$, $t\in\Delta$ such that $x_0=x\in\mcX_0$ and $x_t$ is general in $\mcX_t$ for each $t$. Denote by $s_t=\Phi(x_t)\in\mcS_t$ and choose a section $[l_t]\in\mcK^\mcS_{s_t}$, $t \in\Delta$ such that $[l_0]=[l]$ and $l_t\cap\mcD_2=\emptyset$ when $t\neq 0$. By Corollary \ref{c.stricttransform} the strict transform $C_t$ of $l_t$ under $\Phi$ is an element in $\mcM'_{x_t}$ for each $t\neq 0$. Let $[C]$ be the limit in $\mcM_x$. Since $\Phi(C_t)=l_t$ when $t\neq 0$, $\Phi(C)$ is contained in $l$ by continuity. Then $\Phi_*(C)=l$ by Lemmas \ref{l.reduciblecurveB3} and \ref{l.freecurveB3}.
\end{proof}

By Lemma \ref{l.freecurveB3}, there do exist lines contained in the smooth locus of $\mcS_0$. Take a general point $s\in\mcS_0$ and an element $[l]\in\mcK^\mcS_s$ such that $l$ is contained in the smooth locus of $\mcS_0$. Then $l$ is a free rational curve on $\mcS_0$ and it is also free on $\mcS$. In particular, $\mcK^\mcS$ is the unique irreducible component of ${\rm Hilb}(\mcS/\Delta)$ containing $[l]$, which is the alternative way to define $\mcK^\mcS$. The freeness of $l$ implies that the families $\mcK^\mcS$, $\mcK^\mcS|_{\mcS_0}$ and $\mcK^\mcS_s$ are all smooth at $[l]$. In particular, $\mcK^\mcS|_{\mcS_0}$ and $\mcK^\mcS_s$ are generically reduced.

Next we will show the family of lines on the VMRT $\mcC^\mcS_s$ and its subfamily parametrizing  lines passing through a fixed general point are generically reduced. Recall that the tangent map $\tau_s: \mcU^\mcS_s \to \mcC^\mcS_s$ is a birational finite morphism by Theorems 3.4 and 3.5 of \cite{Kebekus}. So it suffices to study the family of lines on $\mcU^\mcS_s$, i.e. the rational curves on $\mcU^\mcS_s$ sent onto lines on $\mcC^\mcS_s$ by $\tau_s$.

\begin{prop}
Let $s \in \mcS_0$ be a general point and let $q \in \mcU_s^{\mcS}$ be a general point. Let $C$ be a general line on $\mcU_s^{\mcS}$ passing through $q$.  Then $C$ is contained in the smooth locus of
$\mcU_s^{\mcS}$ and if we denote by $f: \mathbb{P}^1 \to C$ the normalization map, then
$$f^* T_{\mcU_s^{\mcS}} = \mathcal{O}(2) \oplus \mathcal{O}(1)^5 \oplus \mathcal{O}^4.$$
\end{prop}
\begin{proof}
As $\Phi_0$ is birational, the pre-image of $s$ is a single point, denoted by $x \in \mcX_0$.
Through the natural morphism $\mcU^\mcS\to\mcK^\mcS$, we may regard $C$ as a 1-dimensional family of minimal rational curves on $\mcS_0$ passing through the point $s$.

Take any $p\in C$, denote by $l_p$ the corresponding minimal rational curve. By Lemma \ref{l.surjfamilyB3}, there is an element $[\ell_p]\in\mcM_x$ such that $\Phi_*(\ell_p)=l_p$. By the general choice of $q$ in $\mcU^\mcS_s$, the element $[\ell_p]$ belongs to $\mcM'_x$ for $p$ lying in an open neighborhood  of $q$. The corresponding point $p$ lies in the smooth locus of $\mcU^\mcS_s$ by Lemma \ref{l.freecurveB3}.

It is straight-forward to see that given a general point $s_t\in\mcS_t$ and any point $p_t\in\mcC^\mcS_{s_t}$ there is a point $q_t\in\mcC^\mcX_{x_t}$ such that the projective line connecting $p_t$ and $q_t$ in the projective space $\mbP T_{x_t}\mcX_t=\mbP T_{s_t}\mcS_t$ is contained in $\mcC^\mcS_{s_t}$, where $\{x_t\}=\Phi^{-1}(s_t)$ and $t\neq 0$. By continuity, this holds when $t=0$. By considering the incidence variety, the general choice of $C$ in $\mcU^\mcS_s$ implies that if $[\ell_p]\in\mcM_x\backslash\mcM'_x$ then $\ell_p = \ell'_p + \ell''_p$ by Lemma  \ref{l.reduciblecurveB3} where
$\ell'_p$ is a general minimal rational curve on $\mcX_0$.
%where $C_p$ is the union of $C'_p$ and $C''_p$ as in Lemma \ref{l.reduciblecurveB3}. Recall that general minimal rational curves on a uniruled projective manifold has empty intersection with any fixed closed subset of codimension at least two.
Then $l_p=\Phi(\ell'_p)$ lies in the smooth locus of $\mcS_0$ by Proposition \ref{p.smoothB3}, and thus $p$ lies in the smooth locus of $\mcU^\mcS_s$. In summary, the line $C$ is contained in the smooth locus of $\mcU^\mcS_s$.

Since the tangent map $\tau_s$ is a finite morphism sending $C$ onto a projective line in $\mbP T_s\mcS_0$, then $f^*T_{\mcU^\mcS_s}$ is a subbundle of $(\tau_s\circ f)^*T_{\mbP T_s\mcS_0}=\mcO(2)\oplus\mcO(1)^{19}$. Note that $C$ is free on $\mcU^\mcS_s$, we have $f^*T_{\mcU^\mcS_s}=\mcO(2)\oplus\mcO(1)^a\oplus\mcO^b$ with some $a, b\geq 0$. The freeness of $C$ implies that it can be deformed to a projective line $C_t$ on $\mcU^\mcS_{s_t}=\Gr(2, 7)$, where $s_t$ is general in $\mcS_t$ with $t\neq 0$ and it lies in a small neighborhood on $s\in\mcS$. Since $\dim\mcU^\mcS_s=\dim\mcU^\mcS_{s_t}$ and since the anti-canonical degree of $C$ in $\mcU^\mcS_s$ is equal to that of  $C_t$ on $\Gr(2, 7)$,  $a=5$ and $b=4$.
\end{proof}

\subsection{Invariance of VMRT for $B_3$}

Let $\sigma: \Delta \to \mcX$ be a section such that $\sigma(t) \in \mcX_t$ is general for all $t$. We denote by $s_t = \Phi(\sigma(t)) \in \mcS_t$.
Let $\mcC^\mcX_\sigma = \cup_t \mcC^\mcX_{\sigma(t)}$ be the family of VMRT of $\mcX$ along the section $\sigma$. Note that for $t \neq 0$, the VMRT $\mcC^\mcX_{\sigma(t)}$ is isomorphic to the orthogonal Grassmannian $B_3/P_2={\rm OG}(2,7)$ by Theorem \ref{t.VMRTWonderful}, which is covered by lines. As a consequence, the central fiber
$\mcC^\mcX_{\sigma(0)} \subset \BP T_{\sigma(0)} \mcX_0$ is  also covered by lines.  Take $\epsilon: \Delta \to   \mcC^\mcX_\sigma$ a general section, hence $ \epsilon(t) \in \mcC^\mcX_{\sigma(t)}$ is a general point.  We denote by $\mathbf{L}^\mcX_{\epsilon(t)} \subset \BP T_{\epsilon(t)} \mcC^\mcX_{\sigma(t)}$ the lines on $\mcC^\mcX_{\sigma(t)}$ passing through $\epsilon(t)$.
Then for $t \neq 0$, we have $\mathbf{L}^\mcX_{\epsilon(t)} \simeq \BP^1 \times \mbQ^1 \subset \BP^5$ by Example \ref{e.contact}, which is a surface of degree $4$ in $\BP^5$.  As a consequence, the central fiber $\mathbf{L}^\mcX_{\epsilon(0)} \subset \BP T_{\epsilon(0)}\mcC^\mcX_{\sigma(0)}$ is a surface of degree $4$. In \cite{HL2}, there is a description of $\mathbf{L}^\mcX_{\epsilon(0)}$ as follows.

%We recall the following classification due to del Pezzo.

%\begin{prop}\label{l.delPezzo}
%An irreducible non-degenerate surface of degree 4 in $\BP^5$ is projectively equivalent to one of the following:

%(1) a cone over a rational normal curve of degree 4;

%(2) the Veronese surface $\nu_2(\BP^2) \subset \BP^5$;

%(3) the rational normal scrolls $S_{2,2} \simeq \BP^1 \times \mbQ^1$ or $S_{1,3} = \BP_{\BP^1}(\mathcal{O}(1) \oplus \mathcal{O}(3))$.
%\end{prop}

\begin{prop} \label{p.twoCaseB3}
The subvariety $\mathbf{L}^\mcX_{\epsilon(0)} \subset \BP T_{\epsilon(0)}\mcC^\mcX_{\sigma(0)}=\BP^6$ is linearly  degenerate and in its linear span, it is projectively equivalent to one of the rational normal scrolls in $\BP^5$: $S_{2,2}:=\BP^1 \times \mbQ^1$ or $S_{1,3}:=\BP_{\BP^1}(\mathcal{O}(1) \oplus \mathcal{O}(3))$.
Furthermore, in the former case, we have $\mcC^\mcX_{\sigma(0)} \simeq {\rm OG}(2,7)$.
\end{prop}
\begin{proof}
This is from \cite{HL2}. For the convenience of readers, we give a sketch of the proof. Let $F^\mcX$ be the family of lines on the family $\mcU^\mcX_\sigma:=\cup_t\mcU^\mcX_{\sigma(t)}$, and let $V^\mcX\rightarrow F^\mcX$ be the universal map. The given general section $\epsilon: \Delta\rightarrow\mcC^\mcX_\sigma$ lifts uniquely to a section $\varepsilon: \Delta\rightarrow\mcU^\mcX_\sigma$ sending $t\in\Delta$ to $\varepsilon(t)\in\mcU^\mcX_{\sigma(t)}$.
Now $V^\mcX_\varepsilon:=\cup_t V^\mcX_{\varepsilon(t)}\rightarrow\Delta$ is a regular family of  projective surfaces. The family of tangent maps $\tau_{\epsilon(t)}: V^\mcX_{\varepsilon(t)}\rightarrow\BP T_{\varepsilon(t)}\mcU^\mcX_{\sigma(t)}\rightarrow \BP T_{\epsilon(t)}\mcC^\mcX_{\sigma(t)}$, $t\in\Delta$, satisfies that $\tau_{\epsilon(t)}$ sends $V^\mcX_{\varepsilon(t)}$ biholomorphically onto $\mathbf{L}^\mcX_{\epsilon(t)}$ for each $t\neq 0$, and $\tau_{\epsilon(0)}$ is the normalization map of $\mathbf{L}^\mcX_{\epsilon(0)}$. As a smooth deformation of $V^\mcX_{\varepsilon(t)}\simeq\mathbf{L}^\mcX_{\epsilon(t)}\simeq\BP^1\times\mbQ^1$, the variety $V^\mcX_{\varepsilon(0)}$ is a Hirzebruch surface $\BP_{\BP^1}(\mcO(-k)\oplus\mcO(k))$ for some $k\geq 0$. We can determine $\tau_{\epsilon(0)}^*\mcO(1)\in\Pic(V^\mcX_{\varepsilon(0)})$ from the fact $\tau_{\epsilon(t)}^*\mcO(1)=\mcO_{\BP^1\times\BP^1}(1, 2)$ for $t\neq 0$. By counting degrees of the surfaces and curves on them, one has $k=0$ or $1$, and thus $V^\mcX_{\varepsilon(0)}$ is biholomorphic to $S_{2, 2}$ or $S_{1, 3}$ respectively. By an analogue with the proof of \cite[Proposition 9]{HM98}, we can show that the linear span of $\mathbf{L}^\mcX_{\epsilon(0)}$ in $\BP T_{\epsilon(0)}\mcC^\mcX_{\sigma(0)}$ is of dimension 5. Hence in its linear span, the variety $\mathbf{L}^\mcX_{\epsilon(0)}$ is projectively equivalent to one of the rational normal scrolls: $S_{2,2}$ or $S_{1,3}$. By the main theorem of \cite{Mok08}, the normalized VMRT $\mcU^\mcX_{\sigma(0)} \simeq {\rm OG}(2,7)$ in the former case.  By an argument similar to that in the proof of Proposition \ref{p.invVMRT}, we have
$\mcC^\mcX_{\sigma(0)} \simeq {\rm OG}(2,7)$.
\end{proof}

From now on, we will assume that $\mathbf{L}^\mcX_{\epsilon(0)} \simeq S_{1,3}$ and then  deduce a contradiction, which will then conclude our proof of the invariance of VMRT for type $B_3$.

Let $\mcG \subset \BP T \mcX$ be the pull-back of the VMRT structure on $\mcS$, which is given by $\mcG_x = d \Phi^{-1} (\mcC^\mcS_{\Phi(x)})$ for $x \in \mcX$ general.  As $\Phi$ preserves minimal rational curves by Proposition \ref{p.B3PreservesLines}, we have $\mcC^\mcX \subset \mcG$ at general (hence all) points of $\mcX$.  For $t \neq 0$, let $\mathbf{L}^\mcG_{\epsilon(t)} \subset \BP T_{\epsilon(t)} \mcG_{\sigma(t)}$ be the set of lines on $\mcG_{\sigma(t)}$ through $\epsilon(t)$, which is projectively equivalent to the Segre variety $\BP^1 \times \BP^4 \subset \BP^9$.  For $t \neq 0$, let $Y_t$ be the unique subvariety isomorphic to $\BP^1 \times \BP^2$ which satisfies $\BP^1 \times \mathbb{Q}^1\simeq \mathbf{L}^\mcX_{\epsilon(t)} \subset Y_t \subset \mathbf{L}^\mcG_{\epsilon(t)} \simeq \BP^1 \times \BP^4$. Set $Y= \overline{\cup_{t\neq 0} Y_t}$ and $Y_0$ the central fiber of $Y$.  We will first describe $Y_0$ in more detail.

Recall $\mathbf{L}^\mcX_{\epsilon(0)} \simeq S_{1,3}$, hence $S_{1,3} \subset Y_0$.  The linear span of $S_{1,3}$ is a 5-dimensional projective space, denoted by $\BP^5_0$. Inside $\BP^5_0$, there exists a line $C$ and the
rational normal cubic $\nu_3(C)$ contained in a $\BP^3$ of $\BP^5_0$ disjoint from $C$.
Then the rational normal scroll $S_{1,3}$ is the union of lines $\cup_{\lambda \in C} {\bf l}_\lambda$, where ${\bf l}_\lambda$ is the line $\overline{\lambda \nu_3(\lambda)}$. Now we can describe $Y_0$ as follows:
\begin{lem} \label{l.Y0}
Let $\BP^2_\lambda = \langle C, {\bf l}_\lambda \rangle$ be the plane generated by $C$ and ${\bf l}_\lambda$.  Then $Y_0 = \cup_{\lambda \in C} \BP^2_\lambda$ and it is linearly non-degenerate in $\BP^5_0$.
\end{lem}
\begin{proof}
Since $Y_t=\BP^1\times\BP^2$ for $t\neq 0$, the variety $Y_0$ is the union of the limits of $\{pt\}\times\BP^2\subset Y_t$. The latter is a linear span of $\{pt\}\times\mathbb{Q}^1\subset\mathbf{L}^\mcX_{\epsilon(t)}$. The limit of $\{pt\} \times \mathbb{Q}^1 \subset Y_t$ is the cycle $C + {\bf l}_\lambda$ for some $\lambda \in C$, whose linear span is $\BP^2_\lambda$. It follows that $Y_0 = \cup_{\lambda \in C} \BP^2_\lambda$.
Since the subvariety $S_{1, 3}$ is linearly non-degenerate in $\BP^5_0$, so is $Y_0$.
\end{proof}

The two factors $\BP^1$ and $\BP^2$ of $Y_t\simeq \BP^1 \times \BP^2$ with $t\neq 0$ induce two integrable distributions  $\mcT'$ and $\mcT''$ of ranks 1 and 2  respectively on an open subset of $Y$, which extend to meromorphic distributions on $Y$ with singular loci properly contained in $Y_0$. Being degenerations of $\BP^1$ and $\BP^2$, the leaf closures of $\mcT'$ and $\mcT''$ at a general point $y\in Y_0$ is a projective line and a projective plane respectively, denoted by $\BP^1_y$ and $\BP^2_y$ respectively.
%
%\begin{lem}
%Given $y\in Y_0$ general, $\BP^2_y=\BP^2_\lambda$ for some $\lambda \in C$.
%\end{lem}
%
%\begin{proof}
%As $\BP^2_{\beta_t}$ is the linear span of $\mathbb{Q}^1_{\beta_t}$ and the limit of $\{pt\} \times \mathbb{Q}^1_{\beta_t}$ is the cycle $C + {\bf l}_b$ for some
%$b \in C$, we have $\BP^2_{\beta_0} = \BP^2_b$.
%\end{proof}
%
Given a general point $\beta_0\in Y_0$, we can take a section $\beta: \Delta \to Y$ passing through $\beta_0$. Then $\BP^1_{\beta_0}= \lim_{t\to 0} \BP^1_{\beta_t}$ and $\BP^2_{\beta_0}= \lim_{t\to 0} \BP^2_{\beta_t}$.
\begin{lem} \label{l.P1P2}
For general $\beta_0 \in Y_0$, we have $\BP^1_{\beta_0} \nsubseteq \BP^2_{\beta_0}$.
\end{lem}
\begin{proof}
Assume $\BP^1_{\beta_0} \subset \BP^2_{\beta_0}$, we will derive a contradiction.
%Take a general section passing through $\beta_0\in Y_0$, denoted by $\beta: \Delta \to Y$. Then $\BP^1_{\beta_0}= \lim_{t\to 0} \BP^1_{\beta_t}$ and $\BP^2_{\beta_0}= \lim_{t\to 0} \BP^2_{\beta_t}$.
For $t\neq 0$, we can write $\beta_t=(\beta^1_t, \beta^2_t)\in\BP^1\times\BP^2=Y_t$. The plane $\BP^2_{\beta_t}$ is the linear span of $\{\beta^1_t\}\times\mathbb{Q}^1$ and the limit of $\{\beta^1_t\}\times\mathbb{Q}^1$ is the cycle $C + {\bf l}_b$ for some
$b \in C$. It follows that $\BP^2_{\beta_0} = \BP^2_b$. As $\beta_0$ is general in $Y_0$, we have $\beta_0 \notin C$, and thus $\BP^2_{\beta_0} = \BP^2_b = \langle C, \BP^1_{\beta_0} \rangle$.

Take a family of lines $l_t \subset \BP^2_{\beta_t}$ through $\beta_t$ such that $l_0 \neq \BP^1_{\beta_0}$ in $\BP^2_{\beta_0}$. Consider the family of
 quadric surfaces $M_t = \BP^1 \times l_t\subset Y_t = \BP^1_{\beta_t} \times \BP^2_{\beta_t}$ and set $M_0= \lim_{t \to 0} M_t$.  As the linear span of $M_t$ is of dimension three for $t \neq 0$, $M_0$ is linearly degenerate in $\langle S_{1,3} \rangle=\langle Y_0\rangle = \BP^5_0$.

Now we will show that $M_0 \subset \BP^5_0$ is linearly non-degenerate, which will conclude the proof. Since $\BP^2_b\cap M_0$ contains two distinct lines $l_0$ and $\BP^1_{\beta_0}$, the plane $\BP^2_b$ is contained in the linear span of $M_0$. Take $\lambda$ in a small neighborhood of $b\in C$. The plane $\BP^2_\lambda\subset Y_0$ is the limit of $\{a_t\}\times\BP^2\subset Y_t$ for a suitable choice of $a_t\in\BP^1$. Then the limit of $\{a_t\}\times l_t\subset M_t$ is a line $l'_\lambda$ in $\BP^2_\lambda$. In particular, $l'_\lambda\neq C$, because $l'_\lambda\subset\BP^2_\lambda$ is a small deformation of $l_0\subset\BP^2_b$ and $l_0\neq C$.

Take a section $\gamma: \Delta\rightarrow Y$ such that $y:=\gamma_0\in l'_\lambda\setminus (l'_\lambda\cap C)$ is general. By assumption, $\BP^1_y=\lim_{t\to 0}\BP^1_{\gamma_t}$ is a line in $\BP^2_\lambda=\BP^2_y=\lim_{t\to 0}\BP^2_{\gamma_t}$ which is different from $l'_\lambda$, because they are small deformations of two distinct lines $l_0\subset\BP^2_{\beta_0}$ and $\BP^1_{\beta_0}\subset\BP^2_{\beta_0}$ respectively.
For $t\neq 0$, the point $\gamma_t\in M_t$ and thus the line $\BP^1_{\gamma_t}=\BP^1\times\{pt\}\subset M_t$. It follows that we have
$\BP^1_y=\lim_{t\to0}\BP^1_{\gamma_t}\subset M_0$, so $\BP^2_\lambda\cap M_0$ contains two distinct lines $l'_\lambda$ and $\BP^1_y$. This implies that the plane $\BP^2_\lambda$ is contained in the linear span of $M_0$. By Lemma \ref{l.Y0}, the variety $Y_0$ is contained in the linear span of $M_0$, and thus $M_0$ is linearly non-degenerate in $\BP^5_0$.
\end{proof}

Let $\zeta: \Delta \to \mcG$ be a general section over the section $\sigma:\Delta\rightarrow\mcX$ so that $\zeta(t) \in \mcG_{\sigma(t)}$ is a general point. Let $\mathbf{L}_\zeta^\mcG := \cup_t \mathbf{L}_{\zeta(t)}^\mcG$,
which is a flat family of projective subvarieties, with general fiber being $\BP^1 \times\BP^4 \subset \BP^9$.  The two factors $\BP^1$ and $\BP^4$ induce two integrable distributions  $\mcD^\mcG$ and ${\mcD'}^\mcG$ of rank respectively 1 and 4 on an open subset of $\mathbf{L}^\mcG_\zeta$, which extend to  meromorphic distributions on $\mathbf{L}^\mcG_\zeta$ with singular loci properly contained in $\mathbf{L}^\mcG_{\zeta(0)}$.  Let  $\mcD_0^\mcG$ (resp. ${\mcD'}_0^\mcG$) be the restriction of  $\mcD^\mcG$  (resp. ${\mcD'}^\mcG$) to
$\mathbf{L}^\mcG_{\zeta(0)}$, which are of ranks 1 and 4 respectively.
Denote by $\BP^1_u$ (resp. $\BP^4_u$) the leaf closure of $\mcD^\mcG$ (resp. ${\mcD'}^\mcG$) at a point $u$ in the domain, which are projective spaces of dimension 1 and 4 respectively.
\begin{lem} \label{l.generateB3}
We have $\mcD_0^\mcG \nsubseteq {\mcD'}_0^\mcG$ on $\mathbf{L}^\mcG_{\zeta(0)}$. In other words, the two distributions $\mcD_0^\mcG$ and ${\mcD'}_0^\mcG$ generate the tangent sheaf of $\mathbf{L}^\mcG_{\zeta(0)}$ at general points.
\end{lem}
\begin{proof}
Suppose not. Given a general section $\gamma: \Delta\rightarrow\mathbf{L}^\mcG_\zeta$,  the projective line $\BP^1_{\gamma_0}=\lim_{t\to 0}\BP^1_{\gamma_t}$ is contained in the 4-dimensional projective space $\BP^4_{\gamma_0}=\lim_{t\to 0}\BP^4_{\gamma_t}$.
%Since the inclusion condition is preserved by specialization, the projective line $\lim_{t\to 0}\BP^1_{\gamma_t}$ is contained in the 4-dimensional projective space $\lim_{t\to 0}\BP^4_{\gamma_t}$, if $\gamma: \Delta\rightarrow\mathbf{L}^\mcG_\zeta$ is an arbitrary section. Note that even if $\gamma_0$ lies in the singular loci of $\mcD_0^\mcG$ or ${\mcD'}_0^\mcG $, the limits $\lim_{t\to 0}\BP^1_{\gamma_t}$ and $\lim_{t\to 0}\BP^4_{\gamma_t}$ are well-defined.
%
The given general sections $\epsilon: \Delta \to   \mcC^\mcX_\sigma$ and $\beta: \Delta \to Y\subset \mathbf{L}^\mcG_{\epsilon}=\cup_t\mathbf{L}^\mcG_{\epsilon(t)}$ are specializations of $\zeta$ and $\gamma$ respectively. Since the inclusion condition is preserved by specializations, the projective line $\BP^1_{\beta_0}=\lim_{t\to 0}\BP^1_{\beta_t}$ is contained in the projective space $\BP^4_{\beta_0}=\lim_{t\to 0}\BP^4_{\beta_t}$.

For $t\neq 0$, $Y_t\simeq\BP^1\times\BP^2\subset\mathbf{L}^\mcG_{\epsilon(t)}\simeq\BP^1\times\BP^4$. In particular, $\BP^2_{\beta_0}=\lim_{t\to 0}\BP^2_{\beta_t}$ is contained in $\BP^4_{\beta_0}=\lim_{t\to 0}\BP^4_{\beta_t}$. By Lemma \ref{l.Y0}, the 3-dimensional variety $Y_0$ is not contained in $\BP^4_{\beta_0}$. Since $\beta_0\in Y_0$ is general, the plane $\BP^2_{\beta_0}$ is the unique irreducible component of $\BP^4_{\beta_0}\cap Y_0$ passing through $\beta_0$. On the other hand, $\BP^1_{\beta_0}=\lim_{t\to 0}\BP^1_{\beta_t}$ is contained in $Y_0=\lim_{t\to 0}Y_t$, and thus $\beta_0\in\BP^1_{\beta_0}\subset Y_0\cap\BP^4_{\beta_0}$. It follows that $\BP^1_{\beta_0}\subset\BP^2_{\beta_0}$, contradicting Lemma \ref{l.P1P2}.
\end{proof}

Now we consider the analogous statement for $\mcS$.  Let $\xi: \Delta \to \mcU^\mcS$ be a general section over the section $\varsigma: \Delta \ni t \mapsto s_t\in\mcS_t$ so that
$\xi(t) \in \mcU^\mcS_{s_t}$ is a general point. Let $\mathbf{L}_{\xi(t)}^\mcS \subset \BP T_{\xi(t)} \mcU_{s_t}^\mcS$ be the variety of lines on $\mcU_{s_t}^\mcS$  through $\xi(t)$.
 Let $\mathbf{L}_\xi^\mcS := \cup_t \mathbf{L}_{\xi(t)}^\mcS$,
which is a flat family of projective subvarieties, with general fiber being $\BP^1 \times\BP^4 \subset \BP^9$.  The two factors $\BP^1$ and $\BP^4$ induce two integrable distributions  $\mcD$ and $\mcD'$ of rank respectively 1 and 4 on an open subset of $\mathbf{L}^\mcS_\xi $, which extend to  meromorphic distributions on $\mathbf{L}^\mcS_\xi$.
%For $\alpha \in \mathbf{L}_\xi^\mcS $ general, the leaf closure of $\mcD$ (resp. $\mcD'$) is a degeneration of linear $\BP^1$ (resp. $\BP^4$), hence it is isomorphic to $\BP^1$ (resp. $\BP^4$). In this case, we denote the leaf closures through $\alpha$ by $\BP^1_\alpha$ and $\BP^4_\alpha$ respectively.
As an immediate corollary of Lemma \ref{l.generateB3}, we have
\begin{cor} \label{c.generateB3}
The two distributions $\mcD$ and $\mcD'$ generate the tangent sheaf of $\mathbf{L}^\mcS_{\xi(0)}$ at general points.
\end{cor}

\begin{lem} \label{l.birSegreB3}
For $s \in \mcS_0$ and $q \in \mcU^\mcS_s$ general points,
there exists a birational map $f_0: \BP^1 \times \BP^4 \dasharrow \mathbf{L}^\mcS_q \subset \BP T_q \mcU^\mcS_s$ such that $f_0^*(\mathcal{O}(1)) = \mathcal{O}(1,1)$.
\end{lem}
\begin{proof}
Consider a general section $\varsigma: \Delta \to \mcS$ through $s$, i.e. $s= \varsigma(0) \in \mcS_0$.   Take a  general section $\xi: \Delta \to \mcU^\mcS$  over the section $\varsigma: \Delta \to \mcS$ such that   $\xi(0) = q$.
Fix a general section $\eta: \Delta\rightarrow\mathbf{L}_\xi^\mcS:= \cup_t \mathbf{L}_{\xi(t)}^\mcS$, and set $\BP^{1, 4}_\eta:=\cup_t\BP^1_{\eta_t}\times\BP^4_{\eta_t}$, which is a $\BP^1\times\BP^4$-bundle over $\Delta$. Let $\Gamma\subset\mathbf{L}_\xi^\mcS\times_\Delta\BP^{1, 4}_\eta$ be the closure of those $(\alpha_t, a, b)\in\Gamma_t$ such that $t\in\Delta$ is arbitrary, $\alpha_t\in\mathbf{L}_{\xi(t)}^\mcS$ is general, $a\in\BP^4_{\alpha_t}\cap\BP^1_{\eta_t}$ and $b\in\BP^1_{\alpha_t}\cap\BP^4_{\eta_t}$.

 Both projections of $\Gamma_t$ are isomorphisms for $t\neq 0$, implying that both projections of $\Gamma_0$ have connected fibers. The fiber of $pr'_0: \Gamma_0\rightarrow \mathbf{L}_{\xi(0)}^\mcS$ at $\eta_0\in \mathbf{L}_{\xi(0)}^\mcS$ consists of a single point, namely $(\eta_0, \eta_0, \eta_0)$. By semi-continuity of dimensions of fibers, there exists an  open subset $O_{\eta_0}$ of $\eta_0\in\mathbf{L}_{\xi(0)}^\mcS$ whose inverse image in $\Gamma_0$ is biholomorphically sent onto it.

By Corollary \ref{c.generateB3}, there is an analytic open neighborhood $U_{\eta_0}$ of $\eta_0\in O_{\eta_0}$ such that the parameter spaces of leaves of $\mcD'|_{U_{\eta_0}}$ (resp. $\mcD|_{U_{\eta_0}}$) admit an analytic  open embedding into $\BP^1_{\eta_0}$ (resp. $\BP^4_{\eta_0}$), defined by sending the leaf at a point $\alpha\in U_{\eta_0}$ to the point $\BP^4_\alpha\cap\BP^1_{\eta_0}$ (resp. $\BP^1_\alpha\cap\BP^4_{\eta_0}$).

Now $(\eta_0, \eta_0, \eta_0)\in\Gamma_0$ is the unique point in the fiber of $pr'_0: \Gamma_0\rightarrow \mathbf{L}_{\xi(0)}^\mcS$ at $\eta_0\in\mathbf{L}_{\xi(0)}^\mcS$. Suppose the fiber of $pr''_0: \Gamma_0\rightarrow\BP^1_{\eta_0}\times\BP^4_{\eta_0}$ at $(\eta_0, \eta_0)$ consists of at least two points. Since $pr''_0$ has connected fibers, there is an irreducible curve $A\subset(pr''_0)^{-1}(\eta_0, \eta_0)$ passing through the point $(\eta_0, \eta_0, \eta_0)\in\Gamma_0$, which is sent onto an irreducible curve $A'\subset\mathbf{L}_{\zeta(0)}^\mcS$ passing through $\eta_0\in U_{\eta_0}$. By the descriptions of parameter spaces in last paragraph, $\BP^1_{\beta}=\BP^1_{\eta_0}$ and $\BP^4_{\beta}=\BP^4_{\eta_0}$ for all $\beta\in A'\cap U_{\eta_0}$. It follows that $A'\cap U_{\eta_0}\subset\BP^1_{\eta_0}\cap\BP^4_{\eta_0}$, which contradicts the fact $\BP^1_{\eta_0}\cap\BP^4_{\eta_0}=\{\eta_0\}$. We have that  the fiber of $pr''_0$ at $(\eta_0, \eta_0)\in\BP^1_{\eta_0}\times\BP^4_{\eta_0}$ consists of a single point in $\Gamma_0$, namely $(\eta_0, \eta_0, \eta_0)$.

By semi-continuity of dimensions of fibers, there exists an open neighborhood of $(\eta_0, \eta_0)\in\BP^1_{\eta_0}\times\BP^4_{\eta_0}$ whose inverse image in $\Gamma_0$ is biholomorphically sent onto it. %Recall that both projections of $\Gamma_t$ are isomorphisms for $t\neq 0$.
Through the incidence variety $\Gamma\subset\mathbf{L}_{\xi}^\mcS\times_{\Delta}\BP^{1, 4}_\eta$, we obtain a family of birational maps $f_t: \BP^1_{\eta_t}\times\BP^4_{\eta_t}\dashrightarrow\mathbf{L}_{\xi(t)}^\mcS\subset \BP T_{\xi(t)} \mcU^\mcS_{s_t}$, $t\in\Delta$, such that for each $t\neq 0$, $f_t$ is an isomorphism sending $(\eta_t, \eta_t)\in\BP^1_{\eta_t}\times\BP^4_{\eta_t}$ to $\eta_t\in\mathbf{L}_{\xi(t)}^\mcS$, and $f_0(\eta_0, \eta_0)=\eta_0\in\mathbf{L}_{\xi(0)}^\mcS$. Since $f_t^*\mcO(1)=\mcO(1, 1)$ for $t\neq 0$, one has $f_0^*\mcO(1)=\mcO(1, 1)$.
 %Consider the family of graphes:
 %$$
 %\Gamma_t:=\{(a_t, b_t, b'_t) \in \mathbf{L}_{\xi(t)}^\mcS \times \BP^1_{p_t} \times \BP^4_{p_t}| b_t \in \BP^4_{a_t} \cap \BP^1_{p_t}, b'_t \in \BP^1_{a_t} \cap \BP^4_{p_t} \}.
% $$
\end{proof}

\begin{prop} \label{p.isomSegreB3}
The birational map $f_0$ is an isomorphism, which induces a projective equivalence between $\mathbf{L}^\mcS_q \subset \BP T_q \mcU^\mcS_s$ and $\BP^1 \times \BP^4 \subset \BP^9$.
\end{prop}
\begin{proof}
By Lemma \ref{l.birSegreB3}, $\mathbf{L}^\mcS_q \subset \BP T_q \mcU^\mcS_s \simeq \BP^9$ is a linear projection of $\BP^1 \times \BP^4 \subset \BP^9$, hence it suffices to show that  $\mathbf{L}^\mcS_q$ is linearly non-degenerate in  $\BP T_q \mcU^\mcS_s$.

Consider the meromorphic distribution $\mcW \subset T \mcU^\mcS_s$ given by the linear span of $\mathbf{L}_q^\mcS$ for general points $q \in \mcU^\mcS_s$.
By a similar argument as \cite[Proposition 10]{HM98},  the variety of tangential lines of $\widehat{\mathbf{L}_q^\mcS} \subset T_q \mcU_s^\mcS$ is contained in the kernel of the Frobenius bracket
$[ \cdot, \cdot]:  \wedge^2 \mcW_q \to T_q \mcU_s^\mcS/\mcW_q.$

As the variety of tangential lines of $\BP^1 \times \BP^4 \subset \BP^9$ is linearly non-degenerate, so is the variety of tangential lines of its linear projection $\mathbf{L}^\mcS_q \subset \BP \mcW_q$, which implies that $[\mcW_q, \mcW_q ]\subset \mcW_q$, namely $\mcW$ is an integrable distribution.  As for $t \neq 0$, the variety $\mcU_{s_t}^\mcS \simeq {\rm Gr}(2,7)$ is chain connected by  lines, so is the central fiber $\mcU_s^\mcS$.  Note that lines through general points of $\mcU_s^\mcS$ are tangent to $\mcW$ and the latter is integrable, hence we have $\mcW = T \mcU_s^\mcS$.  This shows that $\mathbf{L}_q^\mcS$ is linearly non-degenerate in  $\BP T_q \mcU^\mcS_s$, concluding the proof.
\end{proof}

\begin{rmk}
There is a  statement very similar to Proposition \ref{p.isomSegreB3} in \cite[Proposition 8]{HM98}, but in our situation, the central fiber $\mathbf{L}^\mcS_q$ (as well as $\mathcal{U}_s^\mathcal{S}$) is not a priori known to be  smooth.  Our approach here is an adaption of the proof from \cite{Mok08}.
\end{rmk}

\begin{prop} \label{p.isoGrassB3}
For $s \in \mcS_0$ general, we have

(i) The normalized  VMRT $\mcU^\mcS_s$ is isomorphic to ${\rm Gr}(2,7)$.

(ii) The VMRT $\mcC^\mcS_s \subset \BP T_s \mcS_0$ is projectively equivalent to ${\rm Gr}(2,7) \subset \BP^{20}$.

(iii) The central fiber $\mcS_0 $ is isomorphic to $\mathbb{S}$.
\end{prop}
\begin{proof}
(i) Take a section $\varsigma: \Delta \to \mcS$ through $s$ such that $\varsigma(t)$ is a general point in $\mcS_t$.  Let $\mcU_\varsigma^\mcS = \cup_t \mcU_{\varsigma(t)}^\mcS$, which is a family of projective varieties with general fiber isomorphic to ${\rm Gr}(2,7)$. Let $G=({\rm GL}(2) \times {\rm GL}(5))/\{(\lambda id, \lambda^{-1}id)\mid\lambda\in\C^*\}$, then the VMRT structure on ${\rm Gr}(2,7)$ induces a $G$-structure on  $\mcU_\varsigma^\mcS$ outside the central fiber.
%
%By Proposition \ref{p.isomSegreB3}, the variety of lines on $\mcC^\mcS_s$ through a general point $p$ is projectively equivalent to  $\BP^1 \times \BP^4 \subset \BP^9$, which gives a $G$-structure on an open subset of $\mcC^\mcS_s$.  As  $\mcU^\mcS_s \to \mcC^\mcS_s$ is birational, we hence get a $G$-structure on an open subset of $\mcU^\mcS_s$.
By Proposition \ref{p.isomSegreB3}, it extends to a $G$-structure on  $\mcU_\varsigma^\mcS$ outside a codimension at least 2 subset.
By Proposition \ref{p.RigidityIHSS},  we have that $\mcU^\mcS_s$ is isomorphic to ${\rm Gr}(2,7)$.

(ii) Consider the family of tangent maps $\tau_t: \mcU^\mcS_{\varsigma(t)} \to \mcC^\mcS_{\varsigma(t)} \subset \BP T_{\varsigma(t)} \mcS_t$, which is an isomorphism for $t \neq 0$ and $\tau_t^* \mathcal{O}(1) = \mathcal{O}(1)$.  Hence $\mcC^\mcS_s \subset \BP T_{s} \mcS_0$ is a linear  projection of ${\rm Gr}(2,7)\subset \BP^{20}$.
As ${\rm Gr}(2,7)\subset \BP^{20}$ has linearly non-degenerate variety of tangential lines, so is $\mcC^\mcS_s \subset \BP T_{s} \mcS_0$.
By an analogue of the proof of Proposition \ref{p.isomSegreB3}, one can see that $\mcC^\mcS_s \subset \BP T_{s} \mcS_0$ is linearly non-degenerate, hence $\tau_0$ is an isomorphism.

(iii) The VMRT structure on $\mbS$ induces a $G$-structure on  $\mcS|_{\Delta^*}$. By (ii), this $G$-structure extends to an open subset of $\mcS_0$.
By Proposition \ref{p.RigidityIHSS},   $\mcS_0 $ is isomorphic to $\mathbb{S}$.
\end{proof}

\begin{prop} \label{p.invVMRTB3}
For $x \in \mcX_0$ general, the VMRT $\mcC_x^\mcX \subset \BP T_x \mcX_0$ is projectively equivalent to $\OG(2,7) \subset \BP^{20}$.
\end{prop}
\begin{proof}
Let $\sigma: \Delta\rightarrow\mcX$  be a general section such that $\sigma(0)=x$ and set $s_t= \Phi(\sigma(t)) \in \mcS_t$.
Let $\epsilon: \Delta \to \mcC^\mcX$ be a section over $\sigma: \Delta\rightarrow\mcX$ such that $\epsilon(t) \in \mcC^\mcX_{\sigma(t)}$ is a general point.
By Proposition \ref{p.twoCaseB3}, we may assume $\mathbf{L}^\mcX_{\epsilon(0)} \simeq S_{1,3}$ and then deduce a contradiction.
The birational map $\Phi: \mcX \to \mcS$ induces an injective morphism $\Psi_t: \mcC_{\sigma(t)}^\mcX \to \mcC_{s_t}^\mcS \simeq {\rm Gr}(2,7)$ for all $t\in \Delta$ by Proposition \ref{p.isoGrassB3}.

 Let $\varepsilon: \Delta \to \mcC^\mcS$ be the composition of $\epsilon$ with $\Psi$. %which is a section over $\varsigma: \Delta\rightarrow\mcS$.
The family of lines $\mathbf{L}_{\epsilon}^\mcX = \cup_t \mathbf{L}^\mcX_{\epsilon(t)}$ has $\BP^1 \times \mathbb{Q}^1$ as general fibers and $S_{1,3}$ as central fiber.
The family of lines $\mathbf{L}_{\varepsilon}^\mcS= \cup_t \mathbf{L}^\mcS_{\varepsilon(t)}$ is the trivial family of $(\BP^1 \times \BP^4)\times\Delta$. 
By previous discussions,  we have an injective holomorphic map:
$g: \mathbf{L}_{\epsilon}^\mcX \to   \mathbf{L}_{\varepsilon}^\mcS.$ In fact, for each $t\in\Delta$ the morphism $\Psi_t:  \mcC_{\sigma(t)}^\mcX \to \mcC_{s_t}^\mcS$ is induced by the tangent map $(d\Phi_t)_{\sigma(t)}: T_{\sigma(t)}\mathcal{X}_t\rightarrow T_{s_t}\mathcal{S}_t$. Then $\Psi_t$ sends lines on $\mcC_{\sigma(t)}^\mcX$ to lines on $\mcC_{s_t}^\mcS$, giving rise to $g_t: \mathbf{L}_{\epsilon(t)}^\mcX \to   \mathbf{L}_{\varepsilon(t)}^\mcS$. As $\Psi_t$ is injective, so is $g_t$.
 When $t \neq 0$, the map $g_t: \BP^1 \times \mathbb{Q}^1 \to \BP^1 \times \BP^4$ is induced from $\mathbb{Q}^1 \subset \BP^2 \subset \BP^4$,  in particular, the composition
$h_t: \BP^1 \times \mathbb{Q}^1 \xrightarrow{g_t} \BP^1 \times \BP^4 \to \BP^1$ contracts the section $\BP(\mathcal{O}(2))$.  By deforming to $t=0$, we see that the morphism $h_0: S_{1,3} \xrightarrow{g_0} \BP^1 \times \BP^4 \to \BP^1$ contracts both the minimal section and the ruling of $S_{1,3}$, hence it contracts $S_{1,3}$ to one point.  As $g_0$ is injective, this implies that $S_{1,3}$ is contained in $\BP^4$, which is not possible, concluding the proof.
\end{proof}

This completes the proof of Theorem \ref{t.invVMRT}  and hence the main result  Theorem \ref{t.main}, proving that all wonderful compactifications of semisimple complex Lie groups are rigid under Fano deformation.

\bigskip
Baohua Fu (bhfu@math.ac.cn)

HLM and MCM,  Academy of Mathematics and System Science, Chinese Academy of Sciences, Beijing 100190, China
and School of Mathematical Sciences, University of Chinese Academy of Sciences, Beijing 100049, China \\

Qifeng Li (qifengli@sdu.edu.cn)

School of Mathematics, Shandong University, Jinan 250100, China


\begin{thebibliography}{6}

\bibitem[Ad96]{A}
Adams, J. Frank: \emph{Lectures on exceptional Lie groups}, Chicago Lectures in Mathematics. University of Chicago Press, Chicago, IL, 1996. xiv+122 pp

%\bibitem[B]{B} Bia{\l}ynicki-Birula, A. Some theorems on actions of algebraic groups. Ann. of Math. (2) 98 (1973), 480--497.

\bibitem[AN54]{AN}
  Akizuki, Y.;  Nakano, S.:  Note on Kodaira-Spencer's proof of Lefschetz theorems, Proc. Jap.
Acad. 30 (1954), 266--272.

\bibitem[BB96]{BB96}  Bien, Fr\'{e}d\'{e}ric; Brion, Michel: Automorphisms and local rigidity of regular varieties. Compositio Math. 104 (1996), no. 1, 1--26.


%\bibitem[B93]{Br}  Brion, Michel: Vari\'{e}t\'{e}s sph\'{e}riques et th\'{e}orie de Mori. Duke Math. J. 72(2), 369--404 (1993).

\bibitem[Br07]{Br07}  Brion, Michel: The total coordinate ring of a wonderful variety.
J. Algebra  313  (2007),  no. 1, 61--99.

\bibitem[BF15]{BF15}  Brion, Michel; Fu, Baohua: Minimal rational curves on wonderful group compactifications. J. \'{E}c. polytech. Math. 2 (2015), 153--170.

%\bibitem[BI]{BI}  Brion, M.; Inamdar, S. P.: Frobenius splitting of spherical varieties. Algebraic groups and their generalizations: classical methods (University Park, PA, 1991), 207--218, Proc. Sympos. Pure Math., 56, Part 1, Amer. Math. Soc., Providence, RI, 1994.

\bibitem[BK05]{BK}
Brion, Michel; Kumar, Shrawan:
\emph{Frobenius splitting methods in geometry and representation theory},
Progress in Mathematics, 231, Birkh\"auser Boston, Inc., Boston, MA, 2005.

%\bibitem{CS00} A. \v{C}ap, H. Schichl, Parabolic geometries and canonical Cartan connections. Hokkaido Math. J. 29 (2000), no. 3, 453-505.

\bibitem[dCP83]{dCP}
De Concini, Corrado; Procesi, Claudio: Complete symmetric varieties.
In \emph{Invariant theory (Montecatini, 1982)}, 1--44, Lecture Notes in Math. 996, Springer, Berlin 1983.


\bibitem[dFH12]{dFH}
de Fernex, Tommaso; Hacon, Christopher:   Rigidity properties of Fano varieties,  in  {\em Current developments in algebraic geometry}, 113--127,
Math. Sci. Res. Inst. Publ., 59, Cambridge Univ. Press, Cambridge, 2012.
\bibitem[Fo73]{F}
Fogarty, John: Fixed point schemes. Amer. J. Math. 95 (1973), 35--51.

\bibitem[FH12]{FH}
Fu, Baohua; Hwang, Jun-Muk: Classification of non-degenerate projective varieties with
non-zero prolongation and application to target rigidity,   Invent. Math. 189 (2012) 457--513.

\bibitem[Gu65]{Gu}
Guillemin, Victor: The integrability problem for G-structures. Trans. Amer. Math. Soc. 116 (1965), 544--560.

\bibitem[HT99]{HT} Hassett, Brendan; Tschinkel, Yuri: Geometry of equivariant compactifications of $\mathbb{G}_a^n$. Internat. Math. Res. Notices 1999, no. 22, 1211--1230.

\bibitem[HH08]{Hong-Hwang}   Hong, Jaehyun; Hwang, Jun-Muk: Characterization of the rational homogeneous space associated to a long simple root by its variety of minimal rational tangents. Algebraic geometry in East Asia-Hanoi 2005, 217--236, Adv. Stud. Pure Math., 50, Math. Soc. Japan, Tokyo, 2008.

\bibitem[Hu15]{Hu} Huruguen, Mathieu: Log homogeneous compactifications of some classical groups,  Documenta Math. 20 (2015)  1--35.


\bibitem[Hw97]{Hw97} Hwang, Jun-Muk: Rigidity of homogeneous contact manifolds under Fano deformation. J. reine angew. Math. 486 (1997) 153--163.

%\bibitem[H01]{Hw01} Hwang, Jun-Muk: Geometry of minimal rational curves on Fano manifolds. in {\em School on Vanishing Theorems and Effective Results in Algebraic Geometry. Trieste, 2000},  ICTP Lect. Notes, {\bf 6}, Abdus Salam Int. Cent. Theoret. Phys., Trieste, 2001, pp. 335-393.

%\bibitem{IL16} Th. A. Ivey, J. M. Landsberg, Cartan for beginners. Differential geometry via moving frames and exterior differential systems. Second edition. Graduate Studies in Mathematics, 175. American Mathematical Society, Providence, RI, 2016. xviii + 453 pp. ISBN: 978-1-4704-0986-9

\bibitem[Hw06]{H06} Hwang, Jun-Muk: Rigidity of rational homogeneous spaces. International Congress of Mathematicians. Vol. II, 613-626, Eur. Math. Soc., Z\"{u}rich, 2006.

%\bibitem[H12]{H12} Hwang, Jun-Muk:  Geometry of varieties of minimal rational tangents, in {\em  Current developments in algebraic geometry}, 197--226, Math. Sci. Res. Inst. Publ., 59, Cambridge Univ. Press, Cambridge, 2012

\bibitem[HL21]{HL1} Hwang, Jun-Muk; Li, Qifeng:  Characterizing symplectic Grassmannians by varieties of minimal rational tangents, J. Differential Geom. 119 (2021), no. 2, 309--381

\bibitem[HL22]{HL2}  Hwang, Jun-Muk; Li, Qifeng:  Recognizing the $G_2$-horospherical manifold of Picard number $1$ by varieties of minimal rational tangents, arXiv:2212.09226,  Transform.  Groups (to appear)

\bibitem[HLT]{HLT}  Hwang, Jun-Muk; Li, Qifeng; Timashev, Dmitry A. :  Characterizing flag manifolds of
Picard number one via varieties of minimal rational tangents. in preparation.

\bibitem[HM98]{HM98} Hwang, Jun-Muk; Mok, Ngaiming:  Rigidity of irreducible Hermitian symmetric spaces of the compact type under K\"{a}hler deformation. Invent. Math. 131 (1998), no. 2, 393--418.



\bibitem[HM02]{HM02} Hwang, Jun-Muk; Mok, Ngaiming: Deformation rigidity of the rational homogeneous space associated to a long simple root. Ann. scient. Ec. Norm. Sup. 35 (2002) 173--184.


\bibitem[HM04]{HM2}
Hwang, Jun-Muk; Mok, Ngaiming:
Birationality of the tangent map for minimal rational curves,
Asian J. Math. 8 (2004), no. 1, 51--63.

\bibitem[HM05]{HM} Hwang, Jun-Muk; Mok, Ngaiming: Prolongations of infinitesimal linear automorphisms of projective varieties and rigidity of rational homogeneous spaces of Picard number 1 under K\"ahler deformation. Invent. Math.  160 (2005) 591--645.

%\bibitem[JS]{JS} Jelisiejew, Joachim; Sienkiewicz, {\L}ukasz: Bia{\l}ynicki-Birula decomposition for reductive groups. J. Math. Pures Appl. (9) 131 (2019), 290--325.

\bibitem[KY00]{KY} Kaji, H., Yasukura, O.:  Tangent loci and certain linear sections of adjoint varieties. Nagoya Mathematical Journal, 158 (2000), 63--72.

\bibitem[Ki85]{Ki85}
Kirwan, Frances Clare:  Partial desingularisations of quotients of nonsingular varieties and their Betti numbers. Ann. of Math. (2) 122 (1985), no. 1, 41--85.

\bibitem[Kn91]{Kn} Knop, Friedrich: The Luna-Vust theory of Spherical Embeddings. In \emph{Proceedings of the Hyderabad Conference on Algebraic Groups, (Hyderabad, 1989)}, 225-249, Manoj-Prakashan, Madras, 1991.

\bibitem[Ke02]{Kebekus}
Kebekus, Stefan:
Families of singular rational curves,
J. Algebraic Geom. 11 (2002), no. 2, 245--256.

\bibitem[KO73]{KO}
 Kobayashi, Shoshichi; Ochiai, Takushiro: Characterizations of complex projective spaces
and hyperquadrics. J. Math. Kyoto Univ., 13:31--47, 1973.

\bibitem[Ko96]{Kollar}
Koll\'ar, J\'anos:
\emph{Rational curves on algebraic varieties},
Ergeb. Math. Grenzgeb. (3) 32, Springer-Verlag, 1996.
%\bibitem{Kn91} Knop, Friedrich: The Luna-Vust theory of spherical embeddings. Proceedings of the Hyderabad Conference on Algebraic Groups (Hyderabad, 1989), 225-249, Manoj Prakashan, Madras, 1991.


\bibitem[Li18]{Li}
Li, Qifeng: Deformation of the product of complex Fano manifolds. C. R. Math. Acad. Sci. Paris 356 (2018), no. 5, 538--541.

\bibitem[Li22]{Li2}
Li, Qifeng: Fano deformation rigidity of rational homogeneous spaces of submaximal Picard numbers, Math. Ann. 383 (2022), no. 1-2, 203--257

\bibitem[Mi17]{M}
Milne, J. S.:
\emph{Algebraic groups. The theory of group schemes of finite type over a field}. Cambridge Studies in Advanced Mathematics, 170. Cambridge University Press, Cambridge, 2017. xvi+644 pp

\bibitem[Mo08]{Mok08} Mok, Ngaiming: Recognizing certain rational homogeneous manifolds of Picard number 1 from their varieties of minimal rational tangents. In \emph{ Third International Congress of Chinese Mathematicians}. Part 1, 2, 41-61, AMS/IP Stud. Adv. Math., 42, pt. 1, 2, Amer. Math. Soc., Providence, RI, 2008.

%\bibitem[N]{N} Nori, Madhav V.:  Zariski's conjecture and related problems. Ann. Sci. \'{E}cole Norm. Sup. (4) 16 (1983), no. 2, 305-344.

\bibitem[Pa16]{Park} Park, Kyeong-Dong: Deformation rigidity of odd Lagrangian Grassmannians. J. Korean Math. Soc.  53  (2016),  no. 3, 489--501.

%\bibitem[Pas]{Pas09} Pasquier, Boris: On some smooth projective two-orbit varieties with Picard number 1. Math. Ann. 344 (2009), no. 4, 963--987.

\bibitem[PP10]{PP}
Pasquier, Boris; Perrin, Nicolas:  Local rigidity of quasi-regular varieties. Math. Z. 265 (2010), no. 3, 589--600.

\bibitem[Pe14]{Pe} Perrin, Nicolas: On the geometry of spherical varieties. Transform. Groups 19 (2014), no. 1, 171--223.

%\bibitem{Stern83} Sternberg, Shlomo Lectures on differential geometry. Second edition. With an appendix by Sternberg and Victor W. Guillemin. Chelsea Publishing Co., New York, 1983. xviii+442 pp. ISBN: 0-8284-0316-3
\bibitem[Si91]{S}
Siu, Yum-Tong:  Uniformization in several complex variables. In {\em Contemporary Geometry:
J.-Q Zhong memorial volume} (ed. by Hung-Hsi Wu). Univ. Ser. Math., Plenum, New
York 1991, 95--130.

%\bibitem[T]{T} Tanaka, Noboru: On differential systems, graded Lie algebras and pseudogroups. J. Math. Kyoto Univ. 10 (1970), 1-82.

\bibitem[WW18]{WW}
Weber, Andrzej; Wi\'sniewski, Jaroslaw A.: On rigidity of flag varieties.
Int. Math. Res. Not. IMRN  2018,  no. 9, 2967--2979.
\end{thebibliography}
 \end{document}